\documentclass[12pt]{article}
\usepackage{amsmath,amsthm,amsfonts,amssymb}
\usepackage{pictex,dcpic}
\usepackage{graphicx}
\usepackage[usenames]{color} 
\usepackage{pdfsync}
\usepackage[normalem]{ulem}

\newtheorem{thm}{Theorem}[section]
\newtheorem{lemma}[thm]{Lemma}
\newtheorem{corollary}[thm]{Corollary}
\newtheorem*{prop*}{Proposition}
\newtheorem*{lemma*}{Lemma}

\newtheorem{conv}[thm]{Convention}
\newtheorem{prop}[thm]{Proposition}

\theoremstyle{definition}
\newtheorem{definition}[thm]{Definition} 
\newtheorem{defns}[thm]{Definitions}
\newtheorem{notn}[thm]{Notation}
\newtheorem{ex}[thm]{Example}

\newtheorem{review}[thm]{Review}

\newtheorem{remark}[thm]{Remark}

\newcounter{enumitemp}

\newcommand\pref[1]{(\ref{#1})}
\DeclareMathOperator{\Fix}{Fix}
\DeclareMathOperator{\Per}{Per}
\DeclareMathOperator{\PA}{P}
\DeclareMathOperator{\rk}{rank}
\DeclareMathOperator{\lin}{Lin}

\newcommand{\Z}{{\mathbb Z}}
\newcommand{\cT}{{\mathcal T}}

\newcommand{\f}{F_n}
\newcommand{\E}{{\mathcal E}}
\newcommand{\V}{\mathcal V}
\newcommand{\W}{\mathcal W}
\newcommand{\oone}{\phi} 

\newcommand{\aone}{\Phi}
\newcommand{\athree}{\Theta}
\newcommand{\Out}{\mathsf{Out}}
\newcommand{\Aut}{\mathsf{Aut}}
\newcommand{\Inn}{\mathsf{Inn}}

\newcommand{\ffs}{free factor system}
\newcommand{\F}{\mathcal F}
\newcommand{\rtt}{relative train track map}
\def\L{\mathcal L}
\newcommand{\A}{\mathcal A}
\newcommand{\fG} {f : G \to G}
\newcommand{\ti} {\tilde}
\newcommand{\iNp} {indivisible Nielsen path}

\newcommand{\filt}{\emptyset = G_0 \subset G_1 \subset \ldots  \subset G_N = G}
\newcommand{\eg}{EG}
\newcommand{\noneg}{NEG}
\newcommand{\ipNp}{indivisible periodic Nielsen path}
\newcommand{\ct}{\textup{CT}}
\newcommand{\K}{\mathcal K}

\newcommand\restrict{\bigm |}
\def\sig{\eta}
\def\CS{CS}
\def\twistpath{twist path}
\def\stallings{{\hat S}} 
\def\pstallings{S} 
\def\circuit{\gamma}
\def\trans{T}
\def\marking{\mu}
\def\cR{\mathcal R}

\newcommand{\hK} {h : K \to K}
\def\AS{\pstallings}
\def\CS{\pstallings}
\def \zt{Z_\cT}

\newcommand{\urn}[1]{\uppercase\expandafter{\romannumeral#1}}

\title{Algorithmic constructions of relative train track maps and CTs} 
\author{Mark Feighn\thanks{This material is based upon work supported by the
National Science Foundation under Grant No.~DMS-1406167 and also under Grant~No.~DMS-14401040 while the author was in residence at the Mathematical Sciences Research Institute in Berkeley, California, during the Fall 2016 semester.}\and
Michael Handel\thanks{This material is based upon work supported by
the National Science Foundation under Grant No.~DMS-1308710 and by PSC-CUNY  grants in Program Years 46 and 47.}}

\begin{document}
\maketitle
\begin{abstract}  
Building on \cite{bh:tracks, bfh:tits1}, we proved in \cite{fh:recognition} that every element $\psi$ of the outer automorphism group of a finite rank free group is represented by a particularly useful relative train track map. In the case that $\psi$ is rotationless (every outer automorphism has a rotationless power), we showed that there is a type of relative train track map, called a \ct, satisfying additional properties. The main result of this paper is that the constructions of these relative train tracks can be made algorithmic.  A key step in our argument is proving that it is algorithmic to check if an inclusion $\F \sqsubset \F'$ of $\phi$-invariant free factor systems is reduced.  We also give applications of the main result.
\end{abstract}
\tableofcontents

\section{Introduction} \label{s:intro}
An automorphism $\Phi$ of the rank $n$ free group $F_n$ is typically represented by giving its effect on a basis of $F_n$. Equivalently, if we identify the edges of the rose $R_n$ (the graph with one vertex $*$ and $n$ edges) with basis elements of $F_n$, then $\Phi$ may be represented as a self homotopy equivalence of $R_n$ preserving $*$. In this paper, we are interested in outer automorphisms, that is we are interested in elements of  the quotient $$\Out(F_n):=\Aut(F_n)/\Inn(F_n)$$ of the automorphism group $\Aut(F_n)$ of $F_n$ by its subgroup $\Inn(F_n)$ of inner automorphisms. Since outer automorphisms are defined only up to conjugation, the base point loses its special role and $\phi\in\Out(F_n)$ is typically represented as a self homotopy equivalence of $R_n$ that is not required to fix $*$. More generally it is also typical to take advantage of the flexibility gained by representing $\phi$ as a self homotopy equivalence $\fG$ of a marked graph $G$, i.e.\ a graph $G$ equipped with a homotopy equivalence $\marking: R_n\to G$.   The marking identifies the fundamental group of $G$ with $F_n$, but only up to conjugation. In an analogy with linear maps, representing $\phi$ as $\fG$ corresponds to writing  a linear map in terms of a particular basis. 

In \cite{bh:tracks}, Bestvina-Handel showed that every element of $\Out(F_n)$ has a representation as a relative train track map, that is a representation $\fG$ as above but with strong properties. In the analogy with linear maps, a relative train track map corresponds to a normal form. \cite{bh:tracks} goes on to use relative train track maps to solve the Scott conjecture: the rank of the fixed subgroup of an element of $\Aut(F_n)$ is at most $n$. Further applications spurred the further development of the theory of relative train tracks. See for example, \cite{bfh:tits1} where improved relative train tracks (IRTs) were used to show that $\Out(F_n)$ satisfies the Tits alternative or \cite{fh:abeliansubgroups} where completely split relative train tracks (\ct s) \cite{fh:recognition} were used to classify abelian subgroups of $\Out(F_n)$. 

\ct s are relative train tracks that were designed to satisfy the properties that have proven most useful (to us) for investigating elements of $\Out(F_n)$. By convention,  the identity map on the rose with one petal is a \ct\ representing the trivial element of $\Out(F_1)$. For the definition when $n \ge 2$, see Section~\ref{the definition}. Not every $\phi\in\Out(F_n)$ is represented by a \ct , but all rotationless  (see Definition~\ref{def:principal auto}) elements are. This is not a big restriction since there is a specific $M>0$  depending only on $n$ (see Corollary~\ref{c:kn}) such that $\phi^M$ is rotationless. We believe that \ct s will be of general use in approaching algorithmic questions about $\Out(F_n)$. As a first step in that process, our main theorem (Theorem~\ref{algorithmic ct}) verifes that \ct s can be constructed algorithmically.

Many arguments involving \ct s go by induction up through the strata. It is therefore useful  if a \ct\ $\fG$ satisfies: 
\begin{description}
\item (Inheritance)  The restriction of $f$ to each component of each core filtration element is   a \ct.
\end{description}

As Example~\ref{e:example} shows,   not every \ct\ satisfies (Inheritance).   By using an upward induction argument (see Section~\ref{s:extension}) instead of the downward induction arguments used in previous relative train track constructions, we prove that \ct s satisfying (Inheritance) can be  constructed algorithmically.

\begin{thm} \label{algorithmic ct} \label{ct plus}  There is an algorithm whose input is a  rotationless $\phi \in \Out(F_n)$ and whose output is a \ct\ $\fG$  that represents $\phi$ and  satisfies  (Inheritance).   Moreover, for any nested sequence $\cal C$   of $\phi$-invariant free factor systems, one can choose $\fG$ so that  each non-empty element of $\cal C$ is realized by a core filtration element.
\end{thm} 
 
A \ct\ $\fG$ representing rotationless $\phi\in\Out(\f)$ is a graphical representation of $\phi$, and may be used to find graphical representations of some important invariants of $\phi$. For example, there is an algorithm to compute a finite core graph $\pstallings(f)$ immersing to $G$ such that a closed path in $G$ represents a fixed $\phi$-conjugacy class iff and only if it lifts to a closed path in $\pstallings(f)$. There is also a graph $\pstallings_N(f)$ immersing to $G$, obtained from $\pstallings(f)$ by attaching finitely many rays, that additionally records fixed points at infinity (see Sections~\ref{s:fix phi} and \ref{s:stallings_N}). Arbitrarily large neighborhoods of $\pstallings(f)$ in $\pstallings_N(f)$ may be computed algorithmically.

The proof of Theorem~\ref{algorithmic ct}  is in Sections~\ref{s:extension} and \ref{s:extend or reduce proof}. It relies heavily on our paper \cite{fh:recognition} and, more specifically, on the proof of Theorem~4.28 of \cite{fh:recognition} which states that every rotationless outer automorphism of $F_n$ is represented by a \ct. We strongly recommend that the reader have a copy of \cite{fh:recognition} handy while reading the current one. We will also rely on \cite{fh:recognition} for complete references.  Much of the proof of \cite[Theorem~4.28]{fh:recognition} is already algorithmic and the main work in the current paper is to make algorithmic the parts of that proof that are not explicitly algorithmic. In fact, there are only three places in the proof of Theorem~4.28 where non-algorithmic arguments are given. The first has to do with relative train track maps for general elements of $\Out(F_n)$ \cite{bh:tracks}. In this case there is an underlying algorithm to the arguments and we make it explicit in Sections~\ref{section:rtt} and \ref{sec:rotationless iterates}. 

The second part of the proof to be made algorithmic involves checking whether the filtration by invariant free factor systems induced by the given filtration by invariant subgraphs is reduced. (See Section~\ref{standard} for a review of the relevant definitions.) This requires new arguments and is carried out in Section~\ref{section:reducibility}.  The main result of Section~\ref{section:reducibility} is captured in Corollary~\ref{reducible} below, which we believe to be of independent interest. Although we present this result as a corollary of Theorem~\ref{algorithmic ct}, a stand-alone proof could be given using the methods in Section~\ref{section:reducibility}; see in particular Proposition~\ref{find reduction}.  These methods are in turn key ingredients in the proof of the Theorem~\ref{algorithmic ct}.

Recall that if  $\psi\in\Out(\f)$ and  $\F_1 \sqsubset\F_2$ are $\psi$-invariant free factor systems
then $\psi$ is said to be {\it fully irreducible relative to $\F_1 \sqsubset\F_2$}  if for all $k\ge 1$ there are no   $\psi^k$-invariant free factor systems properly contained between $\F_1$ and $\F_2$.  Equivalently,  if $\phi$ is a rotationless iterate of $\psi$ then there are no $\phi$-invariant free factor systems properly contained between $\F_1$ and $\F_2$ ; see \cite[Lemma~3.30]{fh:recognition}. 

\begin{corollary}\label{reducible} There is an algorithm with input   $\psi\in\Out(F_n)$ and $\psi$-invariant free factor systems $\F_1 \sqsubset\F_2$ and output {YES} or {NO} depending on whether or not $\psi$ is fully irreducible relative to $\F_1 \sqsubset\F_2$.  In the case that $\psi$ is not fully irreducible,  $k\ge 1$ and a  $\psi^k$-invariant free factor system that is properly contained between $\F_1$ and $F_2$ are found.  
\end{corollary}

\begin{proof}
Construct a \ct\ $\fG$ with filtration $\filt$ for a rotationless $\phi = \psi^M$ with $M$ as in Corollary~\ref{c:kn} in which $\F_1 \sqsubset\F_2$ are realized by core subgraphs  $G_r \subset G_t$.  Then, by defining property (Filtration) of a \ct, $\psi$ is fully irreducible relative to $\F_1 \sqsubset\F_2$ if there are no core subgraphs $G_s$ properly contained between $G_r$ and $G_t$ .  If there is such a subgraph then its associated free factor system  is properly contained between $\F_1$ and $\F_2$. 
\end{proof}

\begin{remark} In the special case that $\F_1 = \emptyset$, Corollary~\ref{reducible} is an algorithm for checking if $\psi$ is fully irreducible.  Our algorithm in this special case is different from the ones given in   \cite{ik:iwip}  and  \cite{cmp:iwip}. More recently, Kapovich \cite{ik:algorithm} has produced a polynomial time algorithm to detect full irreducibility.
\end{remark}

The third and final non-algorithmic part of the proof of \cite[Theorem~4.28]{fh:recognition} requires a new fixed point result (Lemma~\ref{finding a fixed point}) allowing us to properly attach the terminal endpoints of \noneg\ edges; this takes place in Section~\ref{section:sliding}.

We include some sample applications of Theorem~\ref{algorithmic ct} -- most already known. 
\begin{itemize}
\item In Proposition~\ref{p:fix}, we give another proof of the result of Bogopolski and Maslakova \cite{bm:fix} that it is algorithmic to compute $\Fix(\Phi)$ for $\Phi\in\Aut(F_n)$. \item An outer automorphism $\phi$ is {\em primitively atoroidal} if it does not act periodically on the conjugacy class of a {\it primitive} element of $\f$, i.e.\ an element of some basis.   In Corollary~\ref{c:primitively atoroidal} we give an algorithm to decide if $\phi$ is primitively atoroidal.  
\item  In Section~\ref{s:index} we prove that the well-known index invariant $i(\phi)$ of \cite{gjll:index} can be computed algorithmically and introduce a similar invariant $j(\phi) \ge i(\phi)$ that can also be computed algorithmically.  In Proposition~\ref{p:index} we show that  $j(\phi) \le n-1$ and so the key  inequality satisfied by $i(\phi)$ is still satisfied by $j(\phi)$. 
\item In Corollary~\ref{c:brinkmann} we reprove a result of Ilya Kapovich \cite{ik:hyperbolic} that it is algorithmic to tell if a given $\phi\in\Out(F_n)$ is hyperbolic.  The Kapovich result is stronger in that $\phi$ is only assumed to be an injective endomorphism.
\end{itemize}

We thank the referee for suggesting applications to be added to this paper, specifically: algorithmically computing the index (Proposition~\ref{old index is algorithmic}); algorithmically deciding if an automorphism is primitively atoroidal
(Corollary~\ref{c:primitively atoroidal}); and algorithmically finding the possibilities for
$[\Fix(\Phi)]$ for a $\Phi$ representing a rotationless outer automorphism (Corollary~\ref{c:possibilities}).

\section{Relative train track maps in the general case} \label{section:rtt}
In this section we revisit three existence theorems for relative train track maps representing arbitrary $\phi \in \Out(F_n)$, rotationless or not.   The original statements of these results did not mention algorithms and the original proofs did not emphasize their algorithmic natures. In this section we give the algorithmic versions of two of these results; see Theorem~\ref{thm:rtt}  and Lemma~\ref{cor:rtt}.  The third existence result that we revisit is \cite[Theorem~2.19]{fh:recognition} which in the current paper is reproduced as Theorem~\ref{2.19}. We will need the algorithmic version of Theorem~\ref{2.19} and its proof is postponed until Section~\ref{proving 2.19} because the proof depends on consequences of Theorem~\ref{2.19} established in Section~\ref{sec:rotationless iterates}.  

Rather than cut and paste arguments from \cite{bh:tracks}, \cite{bfh:tits1} and \cite{fh:recognition} into this paper, we will point the reader to specific sections in those papers and explain how they fit together to give the desired results.  In some cases, we refer to arguments that occur in lemmas whose hypotheses  are not  satisfied in our current context.  Nonetheless the arguments that we refer to  will apply.

\subsection{Some standard notation and definitions}  \label{standard} In this section we recall the basic definitions of relative train track theory, assuming that the reader has some familiarity with this material.  Complete details can be found in any of \cite[Sections~ 1 and 5]{bh:tracks}, \cite[Sections~2 and 3]{bfh:tits1}, \cite[Section 2]{fh:recognition} and  Part \urn{1} of \cite{hm:subgroups}.

Identify $F_n$  with $\pi_1(R_n,*)$ where the rose $R_n$ is the graph with one vertex $*$ and $n$ edges.  A (not necessarily connected) graph is {\it core} if it is the union of its immersed circuits.  A {\it marked graph} is a finite core graph $G$ equipped with a homotopy equivalence $\mu :R_n \to G$ called the {\em marking} by which we  identify $\Out(\pi_1(G))$ with $\Out(\pi_1(R_n))$ and hence with $\Out(F_n)$.  In this way, a homotopy equivalence $\fG$ determines an element $\phi \in \Out(F_n)$; we say that $\fG$ {\em represents $\phi$}  or that $\fG$ is a {\em topological representative of $\phi$}.   Conversely, each $\phi \in \Out(F_n)$ is represented (non-uniquely) by a homotopy equivalence of any marked graph $G$, cf.\ Lemma~\ref{l:restrictions extend}. 

\begin{conv}\label{c:linear}
Unless otherwise stated, we assume that  $f$ is an immersion when restricted to edges and that if $J$ is an   interval in the interior of an  edge then some $f$-iterate of $J$ either contains a vertex or is contained in a periodic edge. This can be arranged for example by assigning each edge length one and making the restriction of $f$ to each edge linear.
\end{conv}

A (finite, infinite or bi-infinite) path in $G$ is an  immersion $\sigma : J \to G$ defined on an interval $J$ such  that the image of each end of $J$ crosses infinitely many edges of $G$; equivalently, $\sigma$ lifts to a proper map into the universal cover $\ti G$. We allow the possibility that $J$ is a single point in which case we say that the path is {\em trivial}. Subdividing at the full pre-image of the set $\V$ of vertices of $G$, we view $J$ as a simplicial complex  and $\sigma$ as a map whose restriction to an edge is one-to-one with image an edge or  partial edge. In this way we view $\sigma$ as an {\em edge path}; i.e.\ a concatenation of edges of $G$, where we allow the first and last to be partial edges if the endpoints are not at vertices.  We do not distinguish between paths that have the same associated edge path and we often identify a path with its associated edge path and write $\sigma \subset G$. A path has {\em height $r$} if it is contained in $G_r$ but not $G_{r-1}$.
 
 A bi-infinite path $\sigma$ is called a {\em line}.  Each   lift $\ti \sigma \subset \ti G$ of a line  $\sigma \subset G$ has well defined endpoints in the set $\partial G$ of ends of $G$.  A singly infinite path is called a {\em ray}.  Each lift of a ray has one endpoint in $\ti G$ and one ideal endpoint in $\partial \ti G$.  Conversely, each ordered pair of distinct points in $\partial \ti G$ is the endpoint set of a unique line in $\ti G$.  For any marked graph $G$, one can identify $\partial G$ with $\partial F_n$;  see Section~\ref{sec:rotationless iterates}.  In this way lines in $G$ are identified with $F_n$-orbits of ordered pairs $(P,Q)$  of distinct points in $\partial F_n$.  We sometimes refer to $(P,Q)$ as an {\em abstract line} or even just a {\em line}.  Thus each line in $G$ determines an $F_n$-orbit of abstract lines. Each lift $\ti f :\ti G \to \ti G$ extends to a homeomorphism (also called $\ti f$)  of the set $\partial G$ of the ends of $\ti G$.   

We denote the conjugacy class of a subgroup $A$ of $F_n$ by $[A]$.     If $A_1 \ast \ldots \ast A_m$ is a free factor of $F_n$ and each $A_i$ is non-trivial  then $\{[A_1],\ldots,[A_m]\}$ is a {\em free factor system} and each $[A_i]$ is a {\em component} of that free factor system. We also allow the {\it{trivial free factor system} $\emptyset$.}    For any marked graph $G$ and   subgraph $C$ with non-contractible components $C_1,\ldots, C_m$,  the fundamental group of each $C_i$ determines a well defined conjugacy class that we denote $[C_i ]$ and  $[C] := \{[C_1],\ldots,[C_m]\}$ is a free factor system. We say that $C \subset G$ {\em realizes} $[C]$.   Every free factor system is realized by some subgraph of some marked graph. There is a partial order $\sqsubset$ on free factor systems  defined by $\{[A_1],\ldots,[A_k]\}  \sqsubset \{[B_1],\ldots,[B_l]\}$ if each $A_i$ is conjugate to a subgroup of some $B_j$. There is a natural action of $\phi\in \Out(F_n)$ on conjugacy classes of free factors.  If  each element of a free factor system $\F$  is $\phi$-invariant then we say that {\em $\F$ is $\phi$-invariant}.

More generally, a {\em subgroup system} is a finite collection $\{[A_1],\dots, [A_k]\}$ of distinct conjugacy classes of finitely generated non-trivial subgroups of $F_n$. The conjugacy class $[c]$  of $c \in F_n$ is {\em carried by} the subgroup system $\{[A_1],\ldots,[A_k]\} $ if  $c$ is conjugate to an element of some $A_i$.  Equivalently, the fixed points for the action of a conjugate of $c$ on $\partial F_n$ are contained in some $\partial A_i$.  More generally if $P,Q \in \partial A_i$  then the $F_n$-orbit of {\em the abstract line} with endpoints $P$ and $Q$ {\em is carried by $\{[A_1],\ldots,[A_k]\} $. }  If a subgraph $C$ of a marked graph $G$ realizes a free factor system $\F$ then  a conjugacy class is carried by $\F$ if and only if the circuit representing it in $G$ is contained in $C$; similarly,    the $F_n$-orbit of an abstract line   is carried by $\F$ if and only if the corresponding line in $G$ is contained in $C$.  By \cite[Lemma~2.6.5]{bfh:tits1} (see also Lemma~\ref{free factor support}),  for each collection of conjugacy classes of elements and $F_n$-orbits of abstract lines there is  a unique minimal (with respect to the above partial order $\sqsubset$) free factor system that carries them all.  If that minimal free factor system is $\{[F_n]\}$ then we say that the collection {\em fills}.   In this context, we will treat a subgroup as the collection of conjugacy classes that it carries.  In particular, it makes sense to talk about a subgroup system filling.

A {\em filtration} of a marked graph $G$ is an increasing sequence of subgraphs $\filt$.   The $r^{th}$ {\em stratum} $H_r$ is the subgraph whose edges are contained in $G_r$ but not $G_{r-1}$.   A homotopy equivalence $\fG$ {\it preserves the filtration} if  $f(G_r) \subset G_r$ for each $G_r$.  Assuming this to be the case and that the edges in $H_r$ have been ordered, the {\em transition matrix} $M_r$ associated to $H_r$ is the square matrix with one row and column for each edge  of $H_r$ and whose $ij^{th}$ coordinate is the number of times that the $f$-image of the $i^{th}$ edge of $H_r$ crosses  (in either direction) the $j^{th}$ edge of $H_r$.   After enlarging the filtration if necessary, we may assume that each $M_r$ is either the zero matrix or  irreducible; we say that   $H_r$ is  {\em a zero stratum} or {\em an irreducible stratum} respectively. In the irreducible case, each $M_i$ has a {\em Perron-Frobenius eigenvalue} $\lambda_r \ge 1$.  The stratum $H_r$ is {\em  \eg}\ if $\lambda_r > 1$ and is {\em \noneg}\  if $\lambda_r = 1$. If $\fG$ is a topological representative of $\phi$ then we sometimes say that {\em $\fG$  and $\filt$ represent $\phi$}.

If $\sigma \subset G$ is a path   and $\fG$ is a homotopy equivalence then $f(\sigma) := f \circ \sigma :J \to G$ need not be immersed and so need not be a path.  If $\sigma$ is a finite path, we define $f_\#(\sigma)$ to be the unique path that is homotopic to $f(\sigma)$ rel endpoints.  For rays and lines, one defines $f_\#(\sigma)$ by  choosing a lift $\ti \sigma$, defining $\ti f_\#(\ti \sigma)$ to be the unique path that is homotopic to $\ti f(\ti \sigma)$ rel endpoints (including ideal endpoints) and then projecting $\ti f_\#(\ti \sigma)$ to a path $f_\#(\sigma) \subset G$.

If paths $\sigma_1$ and $\sigma_2$ can be concatenated then we denote the concatenation by  $\sigma_1 \sigma_2$.    A decomposition $\sigma = \ldots \sigma_1 \sigma_2\ldots \sigma_m\ldots$ of a path into subpaths is a {\em splitting} if $f_\#^k(\sigma) = \dots f_\#^k(\sigma_1)f_\#^k(\sigma_2) \ldots f_\#^k(\sigma_m)\ldots$ for all $k \ge 1$; in this case we usually write $\sigma = \ldots \cdot \sigma_1\cdot \sigma_2\cdot \ldots \cdot \sigma_m\cdot \ldots$.

If $f_\#^k(\sigma) = \sigma$ for some $k \ge 1$ and finite path $\sigma$ then $\sigma$ is a {\em periodic Nielsen path}.  If $k = 1$ then $\sigma$ is {\em a Nielsen path}.  If a (periodic) Nielsen path $\sigma$ can not be written as the concatenation of non-trivial (periodic) Nielsen subpaths then it is {\em indivisible}.   Two points in $\Fix(f)$ are in the same {\em Nielsen class} if they bound a Nielsen path.

An edge $E$ in a (necessarily \noneg) stratum $H_i$  is {\em linear} if $f(E)  = E u$ where $u\subset G_{i-1}$ is a non-trivial Nielsen path. If $u=w^d$ for some root-free $w$ and some $d \ne 0$ then the unoriented conjugacy class of $w$ is called the {\em axis} for $E$. All other edges in \noneg\ strata are {\em non-linear}.
       
A {\em direction} at a point $x \in G$  is a germ of non-trivial finite paths with initial vertex $x$.   If $x$ is not a vertex then there are two directions at $x$.  Otherwise there is one direction for each oriented edge based at $x$ and we identify the direction with the oriented edge.  A homotopy equivalence $\fG$ induces a map $Df$ from directions at $x$ to directions at $f(x)$.  A {\em turn} at $x$ is an unordered pair $(d_1,d_2)$ of directions based at $x$; it is {\it degenerate} if $d_1 = d_2$ and {\em non-degenerate} otherwise.  $Df$ induces a map on turns that we also denote by $Df$.  A turn is {\em illegal} if its  image under some iterate of $Df$ is degenerate and is {\em legal} otherwise.  If  $\sigma =\ldots E_i E_{i+1}\ldots$  and if each turn $(\bar E_i,E_{i+1})$ is legal then {\em $\sigma$ is legal}.  Here $\bar E_i$ denotes $E_i$ with the opposite orientation. We sometimes also use the exponent -1 to indicate the inverse of a path. If $\sigma$ has height $r$  then $\sigma$ is {\em $r$-legal}  if each  turn $(\bar E_i,E_{i+1})$ for which both $E_i$ and $E_{i+1}$ are edges in $H_r$  is legal.    
       
The homotopy equivalence $\fG$ is a {\em relative train track map}   \cite[page 38] {bh:tracks}  if it maps vertices to vertices and if the following conditions hold for every \eg\ stratum $H_r$.
\begin{description} 
\item [(RTT-i)]$Df$ maps directions in $H_r$ to directions in $H_r$.
\item  [(RTT-ii)]  If $\sigma \subset G_{r-1}$ is a non-trivial path with endpoints in $H_r \cap G_{r-1}$ then $f_\#(\sigma)$ is a non-trivial path with endpoints in $H_r \cap G_{r-1}$.
\item  [(RTT-iii)]  If $\sigma \subset H_r$ is legal then $f(\sigma)$ is an $r$-legal path.
\end{description}

\subsection{Algorithmic proofs} \label{subset:algpf} The following theorem is modeled on \cite[Theorem~5.12]{bh:tracks} which proves the existence of relative train track maps that satisfy an additional condition called {\it stability}.  This condition  is algorithmically built into \ct s  at a later stage of the argument  (see \cite[Step 1( \eg\ Nielsen Paths) in Section~4.5]{fh:recognition}) so we do not need it here. See also \cite{dv:endos}. The proof of  \cite[Theorem~5.12]{bh:tracks} is mostly algorithmic.  We have added  Lemmas~\ref{checking rtt} and \ref{core subdivision is algorithmic} to  
make the entire argument algorithmic.   
   
\begin{thm} \label{thm:rtt} There is an algorithm that produces for each $\phi \in \Out(F_n)$ a relative train track map $\fG$ representing $\phi$.
\end{thm}

The proof of Theorem~\ref{thm:rtt} appears at the end of this section after we recall some definitions and results from \cite{bh:tracks}. 

\begin{definition} \label{def:bounded} Suppose that $\fG$ and $\filt$ are a topological representative   and filtration  representing $\phi$. For each \eg\ stratum $H_r$,  let $\lambda_r \ge 1$ be the  Perron-Frobenious eigenvalue  of the transition matrix $M_r$ associated to $H_r$.  Let $\Lambda(f)$ be the set of (not necessarily distinct) $\lambda_r$'s  associated to \eg\ strata $H_r$ of $\fG$, listed in non-increasing order.     Say that  $\fG$    is {\em bounded} if there are at most $3n-3$ exponentially growing strata $H_r$   and if each  $\lambda_r$   is the Perron-Frobenious eigenvalue of some irreducible matrix with at most  $3n-3$ rows and columns.  Note that if each vertex of $G$ has  valence at least $3$   then $\fG$ is bounded because $G$ has at most $3n-3$ edges. Note also that if the set of all possible $\Lambda(f)$'s is ordered lexicographically, then any strictly decreasing sequence of $\Lambda(f)$'s associated to bounded $f$'s is finite; see \cite[page 37]{bh:tracks}. 
\end{definition}

\begin{lemma}\label{checking rtt}There is an algorithm that checks if a given   bounded topological representative $\fG$  of $\phi$ is a \rtt.
\end{lemma}

\proof Property (RTT-i) is obviously a finite property.  We  may therefore assume that each \eg\ stratum satisfies (RTT-i).  Suppose that $H_j$ is an \eg\ stratum and that $C$ is a component of $G_{j-1}$.  If $C$ is contractible then it contains only  finitely many paths with endpoints at vertices so checking (RTT-ii) for paths in $C$ is a finite process.  If $C$ is not contractible then there is a smallest $p \ge 1$ such that $f^p(C) \subset C$.  Since $H_j$ satisfies (RTT-i), $f^p(H_j \cap C) \subset H_j \cap C$.  If $f^p$ induces a bijection of $H_j \cap C$ then (RTT-ii) is satisfied for all $\sigma \subset C$.  Otherwise, there exist distinct $v,w \in H_j \cap C$ and a  smallest $1 \le q \le p$ such that  $f^q(v) = f^q(w)$.  After replacing $v$ and $w$ by $f^{q-1}(v)$ and $f^{q-1}(w)$, we may assume that $q =1$.  In this case $v$ and $w$ are connected by a unique path $\sigma \subset C$ whose $f_\#$-image is trivial and (RTT-ii) fails.    We have now proved that (RTT-ii) is a finite property.  Finally, (RTT-iii) is equivalent to the statement that $f(E)$ is $j$-legal for each edge $E \subset H_j$ and so is a finite property.
\endproof

\begin{definition}\cite[paragraph before Lemma 5.13]{bh:tracks} Suppose that $\fG$ is a topological representative and that $E$ is an edge in an \eg\ stratum $H_r$. The {\it core of $E$} is defined to be the smallest closed subinterval of $E$ such that each point in the complement of the core is eventually mapped into $G_{r-1}$.   The set of endpoints of the cores of all the edges in $H_j$  is mapped into itself by $f$.  Subdivision at this finite set (and enlarging the  original filtration to accommodate the new edges formed by the complements of the core) is called the {\em core subdivision of $H_r$}.
\end{definition}

\begin{lemma} \label{core subdivision} \cite[Lemma 5.13]{bh:tracks} Suppose that $\fG$   is a bounded topological representative    of $\phi$ and that $f' : G' \to G'$ is obtained from $\fG$ by a core subdivision of the \eg\ stratum $H_r$.  Then:
\begin{enumerate}
\item $\Lambda(f) = \Lambda(f')$
\item  There is a   bijection $H_j \leftrightarrow H'_{j'}$ between the \eg\ strata of $f$ and the \eg\ strata of $f'$ such that:
\begin{enumerate}
\item $H'_{j'}$ satisfies (RTT-i).
\item  relative height is preserved; i.e. $j < k $ if and only if $j' < k'$
\item  If $j \ne r$ and   $H_j$ satisfies (RTT-i)  or  (RTT-ii) then $H'_{j'}$ satisfies (RTT-i) or (RTT-ii), respectively. \qed
\end{enumerate}  
\end{enumerate}
\end{lemma}

\begin{lemma}  \label{core subdivision is algorithmic} Core subdivision is algorithmic.
\end{lemma}

\proof Suppose that $\fG$ is a topological representative and that $E$ is an edge in an \eg\ stratum $H_r$.  If the $Df$ orbit of $E$ (thought of as the initial direction of $E$) is contained in $H_r$   then the core of $E$ contains an initial segment of $E$ and $E$ does not contribute to the set of core subdivision points.   Suppose then that  some iterate of  $Df$ maps  $E$    into  $G_{r-1}$.  Define $D_rf(E)$  to be the first $H_r$-edge in the edge path $f(E)$.  If  $E$ is $D_rf$-periodic of minimal period $p$, let $E_i = D_r f^i(E)$  for $i = 0,\ldots,p-1$.  Thus $f(E_i) = u_i E_{i+1} v_i$ where $u_i $ is a possibly trivial path in $G_{r-1}$ and indices are taken mod $p$.   Consider the subintervals of  $E_i$ whose images  under $f^p$ are single edges.  The first such subinterval $e_i$  that is mapped into $H_r$  satisfies $f^p(e_i) = E_i$ and we choose the core subdivision point corresponding to $E_i$ to be the unique point in $e_i$ that is fixed by $f^p$.  If $E$ is not $D_rf$-periodic then choose the minimal $q$ such that  $D_{r}f^q(E)$ is $D_{r}f$-periodic and take the core subdivision point for $E$ to be the first point in $E$ that is mapped to the core subdivision point of $D_{r}f^q(E)$.
\endproof

\begin{lemma} \label{subdivide and fold} Given a bounded topological representative $\fG$  of $\phi$ with an \eg\ stratum that satisfies (RTT-i) but not (RTT-iii) there is an algorithm to construct a   bounded topological representative $f'' :G'' \to G''$  of $\phi$ such that $\Lambda(f'') < \Lambda(f)$.
\end{lemma}

\proof  The  proof of \cite[Lemma~5.9]{bh:tracks} contains an algorithm that modifies $\fG$ to produce a  not necessarily bounded, topological representative $f' :G ' \to G'$ such that $\Lambda(f') < \Lambda(f)$.  The  proof of \cite[Lemma~5.5]{bh:tracks} contains an algorithm that modifies $f' :G ' \to G'$ to produce a   bounded topological representative $f'' :G'' \to G''$  of $\phi$ such that $\Lambda(f'') \le \Lambda(f')< \Lambda(f)$.
\endproof
 
We next recall  \cite[Lemma 5.14]{bh:tracks}, making the algorithm used in its proof part of the statement of the lemma.

\begin{lemma} \label{collapsing} \cite[Lemma 5.14]{bh:tracks} Suppose that $\fG$ is a bounded topological representative   of $\phi$ and that $H_s$ is an \eg\ stratum that does not satisfy (RTT-ii).  Then there is an algorithm to construct a bounded topological representative $f' :G' \to G'$  of $\phi$ such that
\begin{enumerate}
\item $\Lambda(f) = \Lambda(f')$
\item  There is a   bijection $H_j \leftrightarrow H'_{j'}$ between the \eg\ strata of $f$ and the \eg\ strata of $f'$ such that:
\begin{enumerate}
\item  relative height is preserved; i.e. $j < k $ if and only if $j' < k'$.
\item   $|H'_{s'} \cap G'_{s'-1}|  < |H_s \cap G_{s-1}|$

\item  if $k > s$ and $H_k$ satisfies (RTT-i) and (RTT-ii) then $H'_{k'}$ satisfies (RTT-i) and (RTT-ii). 
\item  if $k \ge s$ and $H_k$ satisfies (RTT-i)  then $H'_{k'}$  (RTT-i).  \qed
\end{enumerate}  
\end{enumerate}
\end{lemma}

\medspace
\noindent{\em Proof of Theorem~\ref{thm:rtt}:}\ \ \    Start with any bounded topological representative $\fG$  of $\phi$.  For example, one can choose any homotopy equivalence of the rose that represents $\phi$.   If there are no \eg\ strata then $\fG$ is a \rtt\ and we are done.   Otherwise, apply Lemma~\ref{core subdivision}  to produce a    bounded topological representative (still called $\fG$)  of $\phi$ whose top \eg\ stratum satisfies (RTT-i).   If the top \eg\ stratum does not also satisfy (RTT-ii), apply Lemma~\ref{collapsing}   to produce a  new  bounded topological representative (still called $\fG$) of $\phi$ whose top \eg\ stratum still satisfies (RTT-i).  If the top \eg\ stratum of the current $\fG$ does not satisfy (RTT-ii) apply Lemma~\ref{collapsing} again.  Item (b) of that lemma guarantees that after finitely many applications of  Lemma~\ref{collapsing}, we arrive at $\fG$ whose top \eg\ stratum satisfies    (RTT-i) and (RTT-ii). 

Repeat this procedure on the second highest \eg\ stratum to produce a    bounded topological representative   of $\phi$ whose top two \eg\ stratum satisfies   (RTT-i) and (RTT-ii).   After finitely many iterations, we have a bounded topological representative  (still called $\fG$)    of $\phi$, all of whose \eg\ strata  satisfy   (RTT-i) and (RTT-ii). 

Apply Lemma~\ref{checking rtt} to check if $\fG$ is a \rtt.  If yes, we are done.  Otherwise  apply Lemma~\ref{subdivide and fold} to produce a bounded topological representative  $f':G' \to G' $ of $\phi$ with $\Lambda(f') < \Lambda(f)$.  Then start over again with $f':G' \to G'$ replacing the original $\fG$.   Since every decreasing sequence $\Lambda(f) > \Lambda(f') > \ldots$ is finite, this process produces a \rtt\ in finite time. \qed 
\medspace 
       
\begin{corollary}  \label{cor:rtt} There is an algorithm that takes   $\phi \in \Out(F_n)$ and a nested sequence $\cal C$ $= \F_1 \sqsubset \F_2 \sqsubset \cdots \sqsubset \F_m$ of $\phi$-invariant \ffs s as input and produces a    \rtt\ $\fG$ and filtration $\filt$ representing $\phi$ and such that for each $\F_i$ there exists $G_j$ satisfying $\F_i = [G_j]$.
\end{corollary}

\proof  The proof of this corollary is explicitly contained in the proof of \cite[Lemma 2.6.7]{bfh:tits1} (even though the statement of  that lemma is weaker in that it assumes that $\cal C$ is a single free factor system).  The first step of the proof of the lemma is to inductively construct a bounded topological representative $\fG$ and filtration $\filt$ representing $\phi$  such that for each $\F_i$ there exists $G_j$ satisfying $\F_i = [G_j]$.  Then one applies the relative train track algorithm of Theorem~\ref{thm:rtt}, checking that $\cal C$ is preserved,   to promote $\fG$ to a \rtt.
\endproof

The third existence theorem that needs discussion is \cite[Theorem 2.19]{fh:recognition}.  We first recall some notation that is used in its statement.

\begin{notn} \label{n:zero} If $u < r$ and 
  \begin{enumerate}
  \item $H_u$ is irreducible;
  \item $H_r$ is \eg\ and  each component of $G_r$ is non-contractible; and
  \item   for each $u < i < r$,   $H_i$ is a zero stratum that is a component of $G_{r-1}$  and each vertex of $H_i$ has valence at least two in $G_r$
  \end{enumerate}
then we say that each $H_i$ is {\em enveloped by $H_r$} and write $H_r^z = \cup_{k=u+1}^r H_k$.  
  \end{notn}

\begin{thm}[{\cite[Theorem~2.19]{fh:recognition}}] \label{2.19}
For each  $\oone \in \Out(F_n)$    there is a \rtt\ $\fG$ and a  filtration $\filt$ representing $\oone$  and satisfying the following properties.  
 \begin{description}
  \item [(V)] The endpoints of all indivisible periodic Nielsen paths are
vertices.
\item [(P)] If  a stratum $ H_m \subset \Per(f)$ is a forest then there exists a filtration element $G_j$ such  $[G_j] \ne [G_l \cup H_m]$ for any $G_l$.  
 \item [(Z)]  Each zero stratum $H_i$ is enveloped by an  \eg\ stratum $H_r$.    Each vertex in $H_i$ is contained in $H_r$ and has link  contained in $H_i \cup H_r$.
 \item [(\noneg)] The terminal endpoint   of an edge in a non-periodic NEG stratum  $H_i$  is periodic and is contained in a filtration element of height less than $i$ that is its own core.
  \item [(F)]   The core of each filtration element is a filtration element.
  \end{description}
  Moreover, if  $\cal C$ is a   nested sequence 
of $\oone$-invariant free factor systems then we may choose $\fG$ so that for each $\F_i \in \cal C$  there exists  $G_j$ satisfying  $\F = [G_j]$.
\end{thm}

The proof that $\fG$ as in Theorem~\ref{2.19} can be constructed algorithmically is contained in Section~\ref{proving 2.19}

\section{Rotationless iterates} \label{sec:rotationless iterates}  Every element of $\Out(F_n)$ has an iterate that is  rotationless (Definition~\ref{def:principal auto}).  Corollary~\ref{c:kn} below gives an explicit bound on the size of the iterate; see also   \cite[Lemma~4.42]{fh:recognition} for a proof that such a bound exists.

\subsection{More on markings}\label{s:markings}
 In this subsection we discuss markings in more detail and recall definitions and results from \cite[Section~3]{fh:recognition}.  We assume throughout this subsection that $\fG$ is a \rtt\ representing $\phi \in \Out(F_n)$.  
 
 Markings are used to translate the geometric properties of $\fG$  into algebraic properties of $\phi$. In this paper, we will focus on the geometric properties of the homotopy equivalences and only bring in markings at the last minute when necessary. Further  details on the material presented in this section can be found in  \cite[Section~2.3]{fh:recognition}.

Recall that the rose $R_n$ denotes the rose with vertex $*$ and that we have once and for all identified $\pi_1(R_n,*)$ with $F_n$.    A lift  $\ti * \in \ti R_n$ of $*$ to the universal cover $\ti R_n$ determines an isomorphism $J_{\ti *}$ from $F_n =  \pi_1(R_n,*)$ to the group $\cT(\ti R_n)$ of covering translations of $\ti R_n$   given by $[\circuit]$ 
maps to the covering translation $\trans$ of $\ti R_n$ that takes $\ti *$ to the terminal endpoint  
 of the lift   of $\circuit$ with initial endpoint $\ti *$.

Let $G$ be a finite graph equipped with a marking $\marking: R_n\to G$. Denoting $\mu(*)$ by $\star$, \ $\mu:(R_n,*) \to (G,\star)$ induces an isomorphism $\mu_\# : \pi_1(R_n,*) \to \pi_1(G,\star)$ that identifies $F_n$ with $\pi_1(G,\star)$.  
Fix a lift $\ti\star$ of $\star$ to $\ti G$. The lift $\ti \marking:(\ti R_n,\ti *)\to (\ti G, \ti\star)$ determines a homeomorphism $\partial\ti\marking:\partial\ti R_n\to\partial\ti G$ of Gromov boundaries. In this way $\partial F_n$, $\partial\ti R_n$, and $\partial\ti G$ are all identified.  Since covering translations are determined by their action on Gromov boundaries, there is an induced identification of   $\cT(\ti R_n)$  with $\cT(\ti G)$.  For any $v \in G$ and lift $  \ti v \in \ti G$, there is an induced isomorphism $J_{\ti v} : \pi_1(G,v) \to \cT(\ti G)$ defined exactly as $J_{\ti *}$.    It is straightforward to check that   $\mu_\# = J_{\ti \star}^{-1} J_{\ti *}  : \pi_1(R_n,*) \to \pi_1(G,\star)$.  

We also have an identification of automorphisms representing $\phi\in\Out(F_n)$ with lifts $\ti f:\ti G\to \ti G$ of $f:G\to G$ given by $\Phi\leftrightarrow\ti f$ if the actions of $\Phi$ and $\ti f$ on $\partial F_n$ agree, i.e.\ if $\partial\Phi=\partial\tilde f$. We usually specify $\tilde f$ by specifying $\ti f(\ti\star)$ or equivalently by specifying the path $\ti\rho=[\ti\star,\ti f(\ti\star)]$ or its image $\rho$ in $G$. We say that $\Phi$ or $\ti f$ is {\it determined by $\ti f(\ti\star)$, $\ti \rho$, or $\rho$.} The action of $\Phi$ on $\pi_1(G,\star)$ is given by $\gamma\mapsto f(\gamma)^\rho:=\rho f(\gamma)\overline\rho$. If $\ti f$ is determined by $\rho$ and $\ti f'$ is determined by $\rho'$ then $\Phi'=i_{\gamma}\Phi$ where $\Phi\leftrightarrow\ti f$, $\Phi'\leftrightarrow \ti f'$, and $\gamma\in F_n$ is represented by the loop $\rho'\overline\rho$.  Working in the universal cover $\ti G$ is algorithmic in the sense that we can always compute the action of $\ti f$ on arbitrarily large balls (in the graph metric) around $\ti \star$.   In particular, given $\Phi$ we may algorithmically find $\ti f$ with $\Phi\leftrightarrow \ti f$ and {\it vice versa}. If $\ti f$ fixes $\ti v\in\ti G$ then $\ti f\leftrightarrow \Phi$ for $\Phi$ determined by $\rho$ where $\ti \rho=\ti\sig\ti f(\ti\sig^{-1})$ and $\ti \sig=[\ti\star,\ti v]$.

\begin{definition}\label{d:zt}
For $\ti f : \ti G \to \ti G$  a lift of $f :G \to G$, we denote the subgroup of $\cT(\ti G)$ consisting of covering translations that commute with $\ti f$ by  $\zt(\ti f)$.
\end{definition}
We state the following well known fact (see for example \cite[Lemma~2.1]{fh:recognition}) as  a lemma for easy reference.

\begin{lemma} \label{identify Fix}  If $\ti f$ corresponds to $\Phi$ as above then  $\zt(\ti f)$ and $\Fix(\Phi)$ are equal when viewed as subgroups of $\cT(\ti G)$.
\end{lemma}

Automorphisms $\Phi_1, \Phi_2\in\Aut(F_n)$ are {\em isogredient} if $\Phi_1 = i_a \Phi_2 i_a^{-1}$ for some inner automorphism $i_a$.   Lifts $\ti f_1$ and $\ti f_2$ of $f$ are {\it isogredient} if the corresponding automorphisms are isogredient. That is, $\ti f_1$ and $\ti f_2$ are isogredient if there exists a covering translation $T$ of $\ti G$ such that $\ti f_2=T\ti f_1 T^{-1}$. The set of {\em attracting laminations} for $\phi \in \Out(F_n)$ is denoted $\L(\phi)$; see \cite[Section~3]{bfh:tits1}.

\subsection{Principal automorphisms and principal points}  \label{s:principal} 
Recall from \cite{gjll:index} that  for each $\Theta \in \Aut(F_n)$,  
  $$\Fix(\partial\Theta) = \Fix_-(\partial \Theta) \cup \Fix_+(\partial  \Theta) \cup \partial \Fix(\Theta)$$ 
 where  $\Fix(\Theta)$ is the fixed subgroup for $\Phi$, $\Fix_-(\partial  \Theta) \subset \partial F_n$ is a finite union of $\Fix(\Theta)$-orbits of isolated repellers and  $\Fix_+(\partial  \Theta) \subset \partial F_n$ is a finite union of $ \Fix(\Theta)$-orbits of isolated attractors.
 
 Associated to each $\phi \in \Out(F_n)$ is a finite set $\L(\phi)$ of  {\em attracting laminations}, each a closed subset of abstract lines which we refer to as {\em leaves} of the lamination.  A  leaf $\gamma$ of $\Lambda \in \L(\phi)$  is {\em generic} in $\Lambda$ if both of its ends are dense in $\Lambda$.  See   \cite[Section 3.1]{bfh:tits1}.

\begin{definition}\cite[Definition 3.1]{fh:recognition}  \label{def:principal auto} For $\Phi \in \Aut(F_n)$ representing $\phi $, denote the set of non-repelling fixed points of $\partial \Phi$ by  $\Fix_N(\partial\Phi)$.  We say that $\Phi$ is a {\em principal automorphism} and write $\Phi \in \PA(\phi)$ if either of the following hold.
\begin{itemize}
\item $\Fix_N(\partial \Phi)$ contains at least three points.
\item $\Fix_N(\partial \Phi)$ is a two point set that is neither the set of fixed points for the action of some non-trivial $a \in F_n$ on $\partial F_n$ nor the set of endpoints of a lift of a generic leaf of an element of $\L(\phi)$.
\end{itemize}
If $\fG$ is a topological representative of $\phi$ and   $\ti f : \ti G \to \ti G$ is the lift corresponding to principal $\Phi$  then $\ti f$ is a {\em principal lift}.
\end{definition}

If $\Phi \in \PA(\phi)$ and $k > 1$ then $\Fix_N(\partial\Phi) \subset \Fix_N(\partial\Phi^k)$ and $\Phi^k \in \PA(\phi^k)$.  It may be that the  injection  $\Phi \mapsto \Phi^k$ of  $\PA(\phi)$ into $\PA(\phi^k)$ is not surjective.     It may also be that $\Fix_N(\partial\Phi^k)$ properly contains $\Fix_N(\partial\Phi)$ for some principal $\Phi$ and some $k > 1$. If neither of these happen then we say that $\phi$ is {\em forward rotationless}. For a formal definition, see \cite[Definition~3.13]{fh:recognition}.
 
\begin{remark}
It is becoming common usage to suppress the word ``forward" in ``forward rotationless" and we will follow that convention in this paper. So, when we say that $\phi\in\Out(F_n)$ is {\it rotationless}, we mean that $\phi$ is forward rotationless. This convention was followed in the recent work of Handel-Mosher \cite{hm:subgroups}. Be aware though that the term ``rotationless" has a slightly different meaning in \cite{fh:abeliansubgroups}.
 \end{remark}

Suppose that $\fG$ is a topological representative of $\phi$.  By \cite[Corollary~3.17]{fh:recognition}, \  $\Fix(\ti f) \ne \emptyset$ for each principal lift $\ti f$.   The projected  image of $ \Fix(\ti f)$ is exactly a Nielsen class in $\Fix(f)$ and a pair of principal lifts are isogredient if and only if they determine the same Nielsen class of $\Fix(f)$ \cite[Lemma~3.8]{fh:recognition}. 

\begin{definition} \label{d:principal vertex} We say that $x \in \Per(f)$ is {\em principal}   if neither of the following conditions are satisfied.
\begin{itemize}
\item $x$ is not an endpoint of a non-trivial periodic  Nielsen path and there are exactly two periodic directions at $x$, both of which are contained in the same \eg\ stratum.
\item $x$ is contained in a component  $C$ of $\Per(f)$ that is topologically a circle and each point in $C$ has exactly two periodic directions.  
\end{itemize} 
If each principal periodic vertex is fixed and if each periodic direction based at a principal periodic vertex is fixed then we say that $f$ is {\em rotationless}.
\end{definition}

\begin{remark} \label{r:stable} By definition, a  point is principal with respect to $f$ if and only if it is principal with respect to  $f^k$ for all  $k \ge 1$.
\end{remark}

\begin{remark}  Definition~\ref{d:principal vertex} is a corrected version of Definition~3.18 of \cite{fh:recognition} in which  \lq $x$ is not an endpoint of a non-trivial Nielsen path\rq\ in the first item of Definition~\ref{d:principal vertex} is replaced with the inequivalent condition \lq $x$ is the only point in its Nielsen class\rq.  Our thanks to Lee Mosher who pointed this out to us.  Fortunately, the definition we give here and not the one given in \cite{fh:recognition} is the one that is actually used in \cite{fh:recognition} so no further corrections to \cite{fh:recognition} are necessary.
\end{remark}

 We are mostly interested in the case of a \ct , where characterizations of principal points are simpler. The next lemma gives two.

\begin{lemma}\label{l:not principal}
Suppose $\fG$ is a \ct. 
\begin{enumerate}
\item\label{i:ct principal}
A point $x\in\Per(f)$ is principal iff $x\in\Fix(f)$ and the following condition is not satisfied.
\begin{itemize}
\item
 $x$ is not an endpoint of a  non-trivial Nielsen path and there are exactly two periodic directions at $x$, both of which are contained in the same $EG$-stratum.
\end{itemize}
\item\label{i:leaf}
The following are equivalent for a point $x\in\Fix(f)$. Let $\ti f:\ti G\to \ti G$ be a lift of $f$ fixing a lift $\ti x$ of $x$.
\begin{enumerate}
\item\label{i:x principal}
$x$ is principal.
\item\label{i:f tilde principal}
$\ti f$ is principal.
\item\label{i:leaf 2}
$\Fix_N(\partial\ti f^2)$ is not the set of endpoints of a generic leaf of an element of $\L(\phi)$. 
\end{enumerate}
\end{enumerate}
\end{lemma}

\begin{proof}
(1): Periodic Nielsen paths in a \ct\ are fixed \cite[Lemma~4.13]{fh:recognition} and so  the bulleted item in the lemma is equivalent to the first bulleted item of the Definition~\ref{d:principal vertex}. By definition, periodic edges of a \ct\ are fixed and the  endpoints of fixed edges are principal.  Therefore the second item in Definition~\ref{d:principal vertex} never holds.  To complete the proof it remains to show that all principal points of $f$ are fixed.    This holds for vertices because \ct s are rotationless.  If $x$ is a periodic but not fixed point in the interior of an edge then (by definition of a \ct) that edge must be in an \eg\ stratum and so $x$ is not principal.

(2): By \cite[Corollaries~3.22 and 3.27]{fh:recognition}, (\ref{i:x principal}) and (\ref{i:f tilde principal}) are equivalent. If $\ti f$ is principal, then by definition of rotationless and principal, $\Fix_N(\partial\ti f)=\Fix_N(\partial\ti f^2)$ is not contained in the set of endpoints of a generic leaf. We see (\ref{i:f tilde principal}) implies (\ref{i:leaf 2}). If $x$ is not principal for $f$ then the bulleted item in (1) holds. In particular there are exactly two periodic directions at $x$, both of which are in the same $EG$-stratum. By \cite[Lemma~2.13]{fh:recognition}, $\Fix_N(\partial\ti f^2)$ contains the set of endpoints  of a generic leaf of an element of $\L(\phi)$. By Remark~\ref{r:stable}, $x$ is not principal for $f^2$, and so $\ti f^2$ is not a principal lift. Hence $|\Fix_N(\partial\ti f^2)|<3$. We conclude (\ref{i:leaf 2}) implies (\ref{i:x principal}).
\end{proof}

\subsection{A sufficient condition to be rotationless and a uniform bound} \label{s:uniform bound}
Before turning to Lemma~\ref{rotationless}, which gives a sufficient condition for an outer automorphism to be rotationless, we recall the connection between edges in a \ct\ and elements of $\Fix_+(\Phi)$.

\begin{definition}\label{def:slash E}  Given a \ct\ $\fG$ representing $\phi$, let $\E$ (or $\E_f$) be the set of oriented, non-fixed, and non-linear edges in $G$ whose initial vertex is principal and whose initial direction is fixed by $Df$.  For each $E \in \E$,  there is a path $u$ such that $f^k_\#(E) = E\cdot u\cdot f_\#(u) \cdot \ldots \cdot f^{k-1}_\#(u)$ for all $k \ge 1$ and such that $|f^k_\#(u)| \to \infty$ with $k$.  The union of the increasing sequence  $$E \subset f(E) \subset  f^2_\#(E) \subset \ldots  $$  of paths in $G$ is a ray $R_E$.  Each lift $\tilde R_E$ of $R_E$ to the universal cover of $G$ has a well-defined terminal endpoint $\partial\tilde R_E\in\partial F_n$ and so $R_E$ determines an $F_n$-orbit $\partial R_E$ in $\partial F_n$. 
\end{definition}
 
\begin{lemma} \label{Fix+}  \label{l:identifying Fix+}  Suppose that $\fG$  is a \ct\ and that $E \in \E$.  If $\ti E$ is a lift of $E$ and $\ti f$ is the lift of $f$ that fixes the initial endpoint of $\ti E$
then the lift $\ti R_{\ti E}$ of $R_E$ that begins with $\ti E$ converges to a point in $\Fix_+(\partial\ti f)$. Moreover, $E \mapsto \partial R_E$ defines a surjection $\E\to\big(\cup_{\Phi\in\PA(\phi)}\Fix_+(\partial\Phi)\big)/\f$.
\end{lemma}

\proof   Suppose that  $x$ is the initial endpoint of $E \in \E$, that  $\ti x$ is a lift of $x$, that $R_{\ti E}$ is the lift  of $R_E$ that begins at $\ti x$ and that $\ti f : \ti G\to \ti G$ is the lift of $f$ that fixes  $\ti x$. Lemma~\ref{l:not principal}\pref{i:leaf}  implies that $\ti f$ is a principal lift and \cite[Lemma 4.36(1)]{fh:recognition} implies that $\ti R_{\ti E}$ converges to a point $\partial\ti R_E \in \Fix_N(\partial\Phi)$ where $\Phi$ is the principal automorphism corresponding to $\ti f$.  Since $|f^k_\#(u)| \to \infty$, it follows \cite[Proposition~I.1]{gjll:index} that  $\partial\tilde R_{\ti E} \in \Fix_+(\partial\Phi)$. \cite[Lemma 4.36(2)]{fh:recognition} implies that  $E \mapsto \partial R_E$ is surjective.
\endproof

\begin{remark} \label{unique lift} By \cite[Proposition I.1]{gjll:index},  $P \in  \Fix_+(\partial\Phi)$ is not fixed by any $i_a$ and so is not fixed by $\partial \Phi'$ for any   $\Phi' \ne \Phi$ representing $\phi$.  Thus $\cup_{\Phi\in\PA(\phi)}\Fix_+(\partial\Phi)$ is a disjoint union.
\end{remark}

\begin{lemma}  \label{rotationless} Suppose that $\theta \in \Out(F_n)$ acts trivially on $H_1(\f;\Z/3\Z)$ and induces the trivial permutation on $\big(\cup_{\Phi\in\PA(\phi)}\Fix_+(\partial\Phi)\big)/\f$
for some (any) rotationless iterate $\oone = \theta^L$ of $\theta$.  Then $\theta$ is rotationless.
\end{lemma}

\begin{proof} We show below that 
\begin{description}
\item($\ast$)  For any $\aone\in\PA(\oone)$, there is $\Theta\in\PA(\theta)$ with the property that $\Fix_N(\partial\aone)\subset\Fix_N(\partial\Theta)$.
\end{description}
To see why this is sufficient to prove the lemma, let $\Theta_k\in\PA(\theta^k)$ for some $k\ge 1$.  Since $\theta^{kL}=\oone^k$, $\Theta_k^L\in \PA(\oone^k)$. Since $\oone$ is rotationless, there exists $\aone \in \PA(\oone)$  such that $\Theta_k^L = \aone^k$ and $\Fix_N(\partial\aone)=\Fix_N(\partial\aone^k)= \Fix_N(\partial\Theta_k^L)$. By $(\ast)$, there is $\Theta\in\PA(\theta)$ such that 
$$\Fix_N(\partial\Theta_k)\subset\Fix_N(\partial\Theta^L_k)=\Fix_N(\partial\aone)\subset\Fix_N(\partial\Theta) \subset \Fix_N(\partial\Theta^k)$$
 It follows that $\Theta_k=\Theta^k$. We have now seen that $\PA(\theta)\to\PA(\theta^k)$ given by $\Theta\mapsto \Theta^k$ is surjective.
By \cite[Definition~3.13 and Remark~3.14]{fh:recognition}, to show that $\theta$ is rotationless it remains to show that $\Fix_N(\partial\Theta^k)=\Fix_N(\partial\Theta)$ for all $\Theta\in\PA(\theta)$ and $k\ge 1$. This follows from the above displayed sequence of inclusions by taking $\Theta_k:=\Theta^k$.

We now turn to the proof of ($\ast$).   Set $\mathbb F:=\Fix(\aone)$. We claim that    there exists $\Theta$ representing $\theta$ such that $\mathbb F \subset \Fix(\Theta)$.    If the rank of $\mathbb F$ is $< 2$ then this follows from \cite[Part \urn{2} Theorem~4.1]{hm:subgroups}, which implies that $\theta$ fixes each conjugacy class that is fixed by $\oone$ and in particular fixes each conjugacy class  represented by an element of $\Fix(\aone)$.

Suppose then that $\mathbb F$ has rank $\ge 2$. We recall two facts.
\begin{itemize}
\item Each element  of $F_n$ is fixed by only finitely many elements of $\PA(\oone)$ and the root-free ones that are fixed by at least two such automorphisms determine only finitely many conjugacy classes; see  \cite[Lemma~4.40]{fh:recognition}.
\item $\mathbb F$ is its own normalizer in $F_n$. Proof:  Since $\mathbb F$ is finitely generated and has rank $> 1$, we can choose   $x \in \partial \mathbb F\subset \partial\f$ that is not fixed by any $\partial i_a$, $a \in\f\setminus \{1\},$ and so is not fixed by $\partial \Phi'$ for any automorphism $\Phi' \ne \Phi$ representing $\phi$.     If   $y$ normalizes $\mathbb F$ then $\partial i_y(x)  \in \partial \mathbb F$ and   $x$ is fixed by $\partial \Phi'$ where $\Phi' =  i_{y^{-1}}\aone i_y =   i_{y^{-1} \Phi(y)}\aone$.  It follows that $\aone(y) = y$ and hence $y \in \mathbb F$.
\end{itemize}
By the first bullet, we may choose a basis $\{b_j\}$ for $\mathbb F$ consisting of elements that are not fixed by any other element of $ \PA(\oone)$.  Applying  \cite[Part \urn{2} Theorem~4.1]{hm:subgroups} again, choose an automorphism $\Theta_j$ representing $\theta$ and fixing $b_j$.    The automorphism $\Theta_j \aone \Theta_j^{-1}$ fixes $b_j$ by construction and belongs to $\PA(\oone)$ by  \cite[Lemma~2.6]{fh:abeliansubgroups} and the fact that $\theta$ and $\oone$ commute.  By uniqueness, $\Theta_j \aone \Theta_j^{-1} = \aone$ and so $\Theta_j$ commutes with $\aone$.  In particular, $\Theta_j$ preserves $\mathbb F$.  Since $\mathbb F$ is its own normalizer, the outer automorphism $\theta | \mathbb F$ of $\mathbb F$ determined by $\Theta_j$   is  independent of $j$.   It follows that     $\theta | \mathbb F$ acts trivially on $H_1(\mathbb F;\Z/3\Z)$.  Since $\theta^L | \mathbb F$ is the identity and since the kernel of natural map $\Out(F_n)\to H_1(F_n;\mathbb Z/3\mathbb Z)$ is torsion-free\footnote{This follows from the standard fact that the kernel of the natural map $GL_n(\mathbb Z)\to GL_n(\mathbb Z/3\mathbb Z)$ is torsion-free and the result of Baumslag-Taylor \cite{bt:tf} that the  kernel of the natural map $\Out(F_n)\to GL_n(\mathbb Z)$ is torsion-free.}, $\theta |\mathbb F$ is the identity and the claim is proved.

Since $\theta$ acts trivially on $\big(\cup_{\Phi\in\PA(\phi)}\Fix_+(\partial\Phi)\big)/\f$, each $Q \in \Fix_+(\partial\aone)$ is fixed by some $\Theta_Q$ representing $\theta$. Since $\Theta_Q^L $ and $ \aone$  both fix $Q$ and represent $\oone$ we have $\Theta_Q^L= \aone$.  As above, $\Theta_Q$ commutes with $\aone$ and so preserves $\mathbb F$ and $\Fix_+(\partial\aone)$.  For any other $Q' \in \Fix_+(\partial\aone)$ we have $\Theta_{Q} = i_a \Theta_{Q'}$   for some $a \in F_n$.  Since both $\Theta_Q$ and $\Theta_{Q'}$  preserve  $\mathbb F$, $i_a$ does as well and so $a \in \mathbb F$.      

It suffices to show that $\Theta_Q$ is independent of $Q$ and $\mathbb F \subset \Fix(\Theta_Q)$.  This is obvious if $\mathbb F$ is trivial. If $\mathbb F$ has rank one then $\mathbb  F= \Fix(\Theta_Q)$  and 
$$Q' =\aone(Q') = \Theta_Q^L(Q') =  (i_a \Theta_{Q'})^L(Q') =  i_{a}^L\Theta_{Q'}^L(Q') =   i_{a}^L (Q')$$
 which implies that $a$ must be trivial and we are done. If $\mathbb F$ has rank $\ge 2$ then there is a unique $\Theta$ such that $\mathbb F \subset \Fix(\Theta)$ and $\aone$ is the only automorphism representing $\oone$ such that $\mathbb F \subset \Fix(\aone)$.  Thus  $\Theta^L = \aone$.  There exists $b \in \mathbb F$ such that $\Theta_Q = i_b \Theta$.    We have $\aone = \Theta_Q^L  = i_b^L \Theta^L =i_b^L \aone$ so $b$ is trivial and the proof is complete.  \end{proof}
 
To apply Lemma~\ref{rotationless} we will need a bound on the cardinality of $$\big(\cup_{\Phi\in\PA(\phi)}\Fix_+(\partial\Phi)\big)/\f$$

\begin{lemma}\label{l:bound on Fix+} If   $\oone\in\Out(\f)$ is rotationless then $$
\big(\cup_{\Phi\in\PA(\phi)}\Fix_+(\partial\Phi)\big)/\f\le15(n-1)$$
\end{lemma}

\begin{proof}   Choose a \ct\ $\fG$ representing $\phi$ and assume the notation of Definition~\ref{def:slash E}. 
By Lemma~\ref{Fix+} it suffices to show that the cardinality of the image of $E \mapsto [R_E]$  is bounded by $15(n-1)$.  By construction, the initial vertex of $E$ is principal. If  it has valence at least three then we say that it is {\em natural}.

There are at most $6(n-1)$ oriented edges based at natural vertices.  Some of these are not fixed and so do not contribute to $\E$.  For example, if $E$ is \noneg\  then (Lemma~4.21 of \cite{fh:recognition}) the terminal vertex of $E$  is natural and the direction determined by $\bar E$ is not fixed.   Similarly,  if $(E_1,E_2)$ is an illegal turn of \eg\ height then the basepoint for this turn is natural and  either $E_1$ or $E_2$ determines a non-fixed direction.   It follows that $6(n-1)$  is  an upper bound for the sum of  the number of edges in $\E$ that are based at natural vertices,  the number of  non-fixed  \noneg\ edges   and  the number of \eg\ stratum $H_r$ with an illegal turn of height $r$.   

It remains to account for those \eg\ edges $E \in \E$ that are based at  valence two vertices $v$.   By our previous estimate there are at most $6(n-1)$ such $v$ with the other edge incident to $v$ being non-fixed \noneg. The only other possibility is that   both edges incident to $v$ are \eg.   If the edges were in different strata, say $H_r$ and $H_{r'}$ with $r <r'$ then  $v$ would have valence $\ge 2$ in $G_r$ (because $G_r$ is a core subgraph), a contradiction.  Thus  both edges belong to the same stratum $H_r$. Since $v$ is a principal vertex, it must be an endpoint of a Nielsen path $\rho$ of height $r$.   There are at most four edges incident to valence two vertices at the endpoints of $\rho$ and these determine at most three points in 
$\big(\cup_{\Phi\in\PA(\phi)}\Fix_+(\partial\Phi)\big)/\f$ because the two directions pointing into $\rho$ determine the same point. There is at most one such $\rho$ for each \eg\ stratum $H_r$  and $\rho$ has an illegal turn of height $r$ so our initial bound of $6(n-1)$  counted each $\rho$ once; we now have to count it  two more times.    In passing from the highest core $G_s$ with $s<r$ to $G_r$, at least two natural edges are added. It follows that the number of $EG$ strata is $\le \frac{3}{2}(n-1)$.  The total count then is $6(n-1) + 6(n-1) + 3(n-1) =  15 (n-1)$. 
\end{proof}

\begin{corollary}\label{c:kn} Let $h(n)=|GL(\Z/3/Z, n)|=3^{(n^2-1)}$, let $g(m)$ be Landau's function, the maximum order of an element in the symmetric group $S_m$, and let $K_n = g(15(n-1))!\cdot h(n)$. If $\theta\in\Out(\f)$ then $\theta^{K_n}$ is rotationless.
\end{corollary}
 
\begin{proof}   Lemma~\ref{l:bound on Fix+} implies that $\theta^{g(15(n-1))!}$ induces the trivial permutation  on 
$$\big(\cup_{\Phi\in\PA(\phi)}\Fix_+(\partial\Phi)\big)/\f$$ for any rotationless iterate $\phi$ of $\theta$. Hence $\theta^{K_n} = (\theta^{g(15(n-1))!})^{h(n)} $ satisfies the hypotheses of Lemma~\ref{rotationless}.  
\end{proof}

\begin{remark}\label{r:better count}
In Corollary~\ref{c:improved ends count}, we will see that $$|
\big(\cup_{\Phi\in\PA(\phi)}\Fix_+(\partial\Phi)\big)/\f
|\le 6(n-1)$$ and so we could take $K_n = g(6(n-1))!\cdot h(n)$ in Corollary~\ref{c:kn}.
\end{remark}

\subsection{Algorithmic Proof of Theorem~\ref{2.19}} \label{proving 2.19}
We review the proof of the existence of $\fG$ for $\theta$ as given in \cite[Theorem~2.19]{fh:recognition}, altering  it slightly to make it algorithmic.   

If $\cal C$ is not specified, take it to be the single \ffs\  $[F_n]$.   Apply   Corollary~\ref{cor:rtt} to construct a \rtt\ $\fG$ and $\filt$ such that for each $\F_i \in \cal C$  there exists  $G_j$ satisfying  $\F = [G_j]$.    The modifications necessary to arrange all the properties but (V) are explicitly described in the original proof.  These steps come after (V) has been established in that proof but make no use of (V) so there is no harm in our switching the order in which properties are arranged. For notational simplicity we continue to refer to the \rtt\ as $\fG$ even though it has been modified to satisfy all the properties except possibly (V).   

For $K_n$ as in Corollary~\ref{c:kn}, $\theta^{K_n}$ is rotationless. Subdivide $\fG$ at the (finite) set $S$ of isolated points
 in $\Fix(f^{K_n})$ that are not already vertices; these occur only in \eg\ edges $E$ and are in one to one correspondence  with the occurrences of $E$ or $\bar E$ in the edge path $f^{K_n}_\#(E)$. We claim that property (V) is satisfied.  If not, then perform a further finite (\cite[Lemma 2.12]{fh:recognition}) subdivision so that  (V) is satisfied. \cite[Proposition 3.29]{fh:recognition} and \cite[Lemma 3.28]{fh:recognition} imply that every periodic Nielsen path of $\fG$ (after the further subdivision) has period at most $K_n$.  But then $S$ contains the endpoints of all  indivisible periodic Nielsen paths after all and no further subdivision was necessary.  Since subdivision at $S$ is algorithmic, we are done. \qed

\section{Reducibility}\label{section:reducibility}
Given a \rtt\ $\fG$ and filtration $\filt$ representing $\phi \in \Out(F_n)$, let $\emptyset = \F_0 \sqsubset \F_1\sqsubset \ldots \sqsubset \F_K$ be the increasing sequence of distinct $\phi$-invariant free factor systems that are realized by the $G_i$'s. Assuming that $\fG$ satisfies property (F) of Theorem~\ref{2.19}, $\F_i$ is realized by a unique core filtration element for each $i \ge 1$ and $\F_0$ is realized by $G_0$. If $\F$ is a free factor system that is invariant by some iterate $\phi^k$ of $\phi$ and that is  properly contained between $\F_i$ and $\F_{i+1}$ then we say that $\F$ is a {\em reduction for $\F_i \sqsubset \F_{i+1}$ with respect to $\phi$}; if there is no such $\F$ then {\em $\F_i \sqsubset \F_{i+1}$} is   {\em reduced with respect to $\phi$}.   If each $\F_i \sqsubset \F_{i+1}$ is reduced with respect to $\phi$ then we say that $\fG$ {\em is  reduced}.

We assume for the rest of the section that $\phi$ is rotationless.  In particular,  a free factor system that is invariant by some iterate of $\phi$ is $\phi$-invariant   \cite[Lemma 3.30]{fh:recognition}.

The main results of this section are Proposition~\ref{find reduction} and Lemma~\ref{reducible noneg}.  The former, which assumes   that $\fG$ satisfies the conclusions of Theorem~\ref{2.19},  provides an    algorithm in the \eg\   case   for deciding if  $\F_i \sqsubset \F_{i+1}$ is reduced and for finding a reduction if there is one.   The latter has stronger requirements and easily leads to an algorithm that handles the \noneg\ case. We save the final details of that algorithm for Section~\ref{s:extend or reduce proof}.  

\subsection{The \eg\ case}
Recall (\cite[Section~2]{bfh:tits1} or \cite[Part \urn{1} Fact 1.3]{hm:subgroups}) that a pair of free factor systems $\F^1$ and $\F^2$ has a well-defined {\em meet} $\F^1 \wedge \F^2$ characterized by $[A] \in \F^1 \wedge \F^2$ if and only if there there exist subgroups $A^1, A^2$ such that $[A^i] \in \F^i$ and $A^1 \cap A^2  =  A$.

Let $\mathcal B$ be the basis of $F_n$ corresponding to the edges of $R_n$ (see Section~\ref{standard}). If $A$ is a finitely generated subgroup of $F_n$ then the {\it Stallings graph} $R_A$ of the conjugacy class $[A]$ of $A$ is the core of the cover of $R_n$ corresponding to $A$. There is an immersion $R_A\to R_n$ and if we subdivide $R_A$ at the pre-image of the vertex of $R_n$ then we view the edges of $R_A$ as {\it labeled} by their image edges in $R_n$ and hence by elements of $\mathcal B$. The {\it complexity} of $A$ is the number of edges in (subdivided) $R_A$. Stallings graphs are generalized in Section~\ref{s:stallings} and more discussion can be found there.

\begin{lemma} \label{compute wedge}   Given free factor systems $\F^1$ and $\F^2$ one can algorithmically construct $\F^1 \wedge \F^2$.
\end{lemma}

\proof   We may assume without loss that $\F^1 =\{[A]\}$ and $\F^2 = \{[B]\}$ for given subgroups $A,B$. According to Stallings \cite[Theorem~5.5 and Section~5.7(b)]{st:folding}, the conjugacy classes of the intersections of $A$ with conjugates of $B$ are all represented by components of the pullback of the diagram $R_A\to R\leftarrow R_B$.
\endproof  

\begin{lemma} \label{free factor support}  Given a  finite set $\{a_i\}$ of elements of $F_n$ and a finite set $\{A_j\}$ of finitely generated subgroups of $F_n$ there is an algorithm that finds the unique minimal free factor system that carries each $[a_i]$ and each conjugacy class carried by some $[A_j]$. 
 \end{lemma}
 
\begin{proof}

By replacing $[a_i]$ by $[\langle a_i\rangle]$, we may assume that finite set $\{a_i\}$ is empty. The {\it complexity} of $\mathcal A=\{[A_j]\}$ is the sum of the complexities of the $[A_j]$. By Gersten \cite{sg:whitehead}, there is an algorithm to find $\athree\in\Aut(F_n)$ so that $\athree(\mathcal A)=\{[\athree(A_j)]\}$ has minimal complexity in the orbit of $\mathcal A$ under the action of $\Aut(F_n)$. Let $\mathcal P$ be the finest partition of $\mathcal B$ such that the labels of each Stallings graph $R_{\athree(A_j)}$ are contained in some element of $\mathcal P$. The free factor system $\mathcal{F(P)}$ determined by $\mathcal P$ is the minimal free factor system carrying $\athree(\mathcal A)$; see \cite[Lemma~9.19]{df:models}. Hence $\athree^{-1}(\mathcal{F(P)})$ is the minimal free factor system carrying $\mathcal A$.
\end{proof}

\begin{corollary}  \label{checking on ffs}  Suppose that $\phi \in \Out(F_n)$, that  $\F^1$ is a proper \ffs\  and that $\F^0 \sqsubset \F^1$ is a (possibly trivial) $\phi$-invariant \ffs.  
Then there is an algorithm that decides if there is a $\phi$-invariant \ffs\ $\F \sqsubset  \F^1$  that  properly contains  $\F^0$ and  that finds such an $\F$ if one exists. 
\end{corollary}

\proof    First check if $\F^1$ is $\phi$-invariant or if $\F^1= \F^0$.  If the latter is true the output of the algorithm is NO.  If the latter is false and the former is true the output is YES.   If neither is true apply Lemma~\ref{compute wedge} to compute $\F^1 \wedge \phi(\F^1)$  which is properly contained in $\F^1$, contains $\F^0$  and contains every $\phi$-invariant \ffs\ that is contained in $\F^1$.   Repeat these steps    with $\F^1 \wedge \phi(\F^1)$ replacing $\F^1$.  Since  there is a uniform bound to the length of a strictly decreasing sequence of \ffs s  \cite[Part \urn{1} Fact~1.3]{hm:subgroups} the process stops after finitely many steps.   
\endproof

The following lemma is used in Step 1 of the proof of Proposition~\ref{find reduction}.    Recall from \cite[Remark 3.20]{fh:recognition} that if $\fG$ represents a rotationless $\phi$ and satisfies   the conclusions of  Theorem~\ref{2.19}, then for each $\Lambda \in \L(\phi)$ there is an \eg\ stratum $H_r$ such that $\Lambda$ has height $r$ and  this defines a bijection between $\L(\phi)$ and the set of \eg\ strata.  The definition of a  path being weakly attracted to $\Lambda^+$ appears as   \cite[Definition 4.2.3]{bfh:tits1}.

\begin{lemma}  \label{critical constant} Suppose that $\fG$ and $\filt$ are  a \rtt\ and filtration representing a rotationless $\phi$ and satisfying the conclusions of  Theorem~\ref{2.19} and that $H_r$ is an \eg\ stratum  with associated attracting lamination $\Lambda^+$.  Then there is a computable constant $C$ such that if  $\sigma_0 \subset G_r$ is an $r$-legal path   that crosses at least $C$ edges in $H_r$   then every path $\sigma \subset G_r$ that contains $\sigma_0$ as a subpath is weakly attracted to $\Lambda^+$.
\end{lemma}    

\proof Choose $l$ so that the $f^l$ image of each edge in $H_r$ crosses at least two edges in $H_r$.  It is shown in the proof of \cite[Lemma 4.2.2]{bfh:tits1} (see also \cite[Corollary 4.2.4]{bfh:tits1})  that $C =4 l C_0+1$ satisfies the conclusions of our lemma for any constant  $C_0$ that is greater than or equal to the     bounded cancellation constant for $f$.    Since the latter can be computed \cite[Lemma 3.1]{bfh:tits0}  using only the transition matrix for $f$, we are done.
\endproof

\begin{lemma}  \label{finding iNps}Suppose that $\fG$ and $\filt$ are  a \rtt\ and filtration representing a rotationless $\phi$ and satisfying the conclusions of  Theorem~\ref{2.19}.  Suppose further that $H_r$ is an \eg\ stratum.  Then every \ipNp\ with height $r$  has period one and there is an algorithm that finds them all.  \end{lemma}

\begin{remark} For a more efficient method than the one described in the proof see \cite[Section 3.4]{hm:axes}.
\end{remark}

\proof  Suppose that $\rho$ is an \ipNp\ of height $r$.    Proposition~3.29   and ~Lemma 3.28 of \cite{fh:recognition} imply that $\rho$ has period $1$ and property (V) of Theorem~\ref{2.19} implies that the endpoints of $\sigma$ are vertices.   By \cite[Lemma 5.11]{bh:tracks}, $\rho$ decomposes as a concatenation $\rho = \alpha \beta^{-1}$ of $r$-legal edge  paths  whose initial and terminal edges are in $H_r$. Let $\alpha_0$ and $\beta_0$ be the initial edges of $\alpha$ and $\beta$   respectively.   By Lemma~\ref{critical constant} we can bound the number of $H_r$ edges crossed by $\alpha$ and $\beta$ by some positive constant $C$.  Since $H_r$ is an \eg\ stratum we can choose $k$ so that $f^k_\#(E)$ crosses more than $C$  edges in $H_r$ for each edge $E$ in $H_r$.    Since $\rho = f^k_\#(\rho)$ is obtained from $f^k_\#(\alpha) f^k_\#(\beta^{-1})$ by canceling edges at the  juncture point  and since no edges in $H_r$ are cancelled when $f^k(\alpha)$  and $f^k(\beta)$ are tightened to $f^k_\#(\alpha)$  and $f^k_\#(\beta)$,  $\alpha \subset f^k_\#(\alpha_0)$ and  $\beta \subset f^k_\#(\beta_0)$.     In particular, we can compute an   upper bound for the number of edges crossed by $\rho$, reducing us to testing a finite set of paths to decide which are \iNp s. 
\endproof

A subgroup system $\A$ is a {\em vertex group system} if there exists   a real $F_n$-tree with trivial arc stabilizers such that $\A$ is the set of non-trivial vertex stabilizers.

\begin{lemma}  \label{ANA}Suppose that $\fG$ and $\filt$ are  a \rtt\ and filtration representing a rotationless $\phi$ and satisfying the conclusions of  Theorem~\ref{2.19}.  Suppose also that $H_{N}$ is an \eg\ stratum with attracting lamination $\Lambda^+$, and that   $[G_u] \cup \Lambda^+$ fills   where  $G_u =  G \setminus H^z_{N}$.  Then the following are satisfied.  
\begin{enumerate}
\item  \label{item:subgroup system} There is a unique vertex  group system $\A$ such that a conjugacy class $[a]$ is not weakly attracted to $\Lambda^+$ if and only if $[a]$ is carried by an element of $\A$.
\item \label{item:ANA splitting}A circuit $\sigma \subset G$ represents an element of $\A$ if and only if  $\sigma$ splits as a concatenation of subpaths each of which is either contained in $G_{u}$ or is an \iNp\ of height $N$.
\item \label{item:characterize irreducibility} There is a proper free factor system $\F' \sqsupset [G_{u}] $ such that one of the following holds.
\begin{enumerate}
\item $\A = \F'$
\item $\A$ fills and  $\A= \F'\cup \{[A]\}$ where $[A]$ has rank one.
\end{enumerate}
Moreover, $[G_u] \sqsubset [G_N]$ is reduced with respect to $\phi$   if and only if $\F' = [G_{u}]$.
\end{enumerate}
\end{lemma}

\begin{remark}\label{r:core}  $G_u$ is not necessarily a core subgraph.  It deformation retracts to a core subgraph $G_s$ and is obtained from $G_s$ by adding  \noneg\ edges with terminal endpoints in $G_s$.  We use $G_u$ in this lemma rather than $G_s$ because \iNp s of height $N$ can have endpoints at the valence one vertices of $G_u$.
\end{remark}

\proof   The existence of a vertex group  system $\A$ as in \pref{item:subgroup system} follows from \cite[Theorem~6.1]{bfh:tits1} and  \cite[Part \urn{3} Proposition~1.4(1)]{hm:subgroups}.     Uniqueness of $\A$  follows from the fact  \cite[Part \urn{1} Lemma~3.1]{hm:subgroups} that a vertex group system is determined by the conjugacy classes that it carries.  
  For the rest of this proof we take \pref{item:subgroup system}  to be the defining property of $\A$.  Note that $\A$ depends only on $\Lambda^+$ and $\phi$ and not on the choice of $\fG$.  In particular, $\A$ is $\phi$-invariant.

If a circuit $\sigma \subset G$ splits into subpaths that are either contained in $G_u$ or are Nielsen paths of height $N$ then the number of $H_N$ edges in $f^k_\#(\sigma)$ is independent of $k$ and $\sigma$ is not weakly attracted to $\Lambda^+$.  This proves the if direction of  \pref{item:ANA splitting}.  
 
The only if direction of  \pref{item:ANA splitting} is more work. \cite[Lemmas~4.2.6 and 2.5.1]{bfh:tits1}  and Lemma~\ref{finding iNps} imply that  there exists  $k\ge 1$ such that $f^k_\#(\sigma)$ splits into subpaths that are either contained in   $G_{N-1}$,  are  indivisible Nielsen paths of height $N$  or are edges of height $N$.  Assuming that $\sigma$, and hence   $f^k_\#(\sigma)$,  is not weakly attracted to $\Lambda^+$,     \cite[Corollary~4.2.4]{bfh:tits1} implies that no term in this splitting is an edge of height $N$.     If  $f^k_\#(\sigma) \subset G_u$ then $\sigma \subset G_u$ and we are done.   If  $f^k_\#(\sigma)$ is a closed Nielsen path of height $N$ then $\sigma$ and $f_\#^k(\sigma)$ have the same $f^k_\#$-image and so are equal.  In particular, $\sigma$ is a Nielsen path of height $N$. In the remaining case,  there is a splitting  $$f^k_\#(\sigma) = \mu_1 \cdot \nu_1\cdot \mu_2 \cdot \nu_2\cdot \ldots \cdot \mu_m \cdot \nu_m$$  into subpaths $\mu_i \subset G_u$ and Nielsen paths $\nu_i$ of height $N$.    Since the endpoints of each $\mu_i$ are fixed by $f$ and since the restriction of $f$ to each $f$-invariant component of $G_u$ is a homotopy equivalence, there exist paths $\mu_i' \subset G_u$ with the same endpoints as $\mu_i$ such that $f^k_\#(\mu_i') = \mu_i$.  Letting $$\sigma' = \mu'_1 \cdot \nu_1\cdot \mu'_2 \cdot \nu_2\cdot \ldots \cdot \mu'_m \cdot \nu_m$$  we have $f^k_\#(\sigma') = f^k_\#(\sigma)$   and hence $\sigma' =  \sigma$.  In particular, $\sigma$ splits into subpaths of $G_u$ and indivisible Nielsen paths of height $N$.   This completes the proof of \pref{item:ANA splitting}.
 
The main statement of \pref{item:characterize irreducibility} follows from \cite[Proposition 6.0.1 and Remark 6.0.2]{bfh:tits1} (which applies because there is a \ct\ representing $\phi$ in which $\Lambda^+$ corresponds to the highest stratum and $[G_u]$ is realized by a core filtration element).  Since $\phi$ preserves $\A$, it acts periodically on the components of $\A$.  \cite[Lemma 3.30]{fh:recognition}  therefore implies that $\phi$ preserves each rank one component of $\A$ and so also preserves $\F'$.    If  $\F' \ne [G_u]$ then $\F'$ is a reduction for $[G_u] \sqsubset [G_N]$.  This proves the only if direction of the moreover statement. 

Suppose then that $\F' = [G_u] $.  If either (a) or (b) holds then any free factor system $\F$ that properly contains $[G_u]$ carries a conjugacy class not carried by $\A$.  Item \pref{item:subgroup system} implies that $\F$ carries a conjugacy classs that is weakly attracted to $\Lambda^+$.  If $\F$ is $\phi$-invariant then $\F$  carries $\Lambda^+$ in addition to containing $G_u$ and so is improper.  This completes the proof of the if direction of the moreover statement.
\endproof

The next proposition shows how to reduce a relative train track map satisfying the conclusions of Theorem~\ref{2.19}. One way to create an unreduced example is to identify a pair of distinct fixed points in a stratum $H_r$ of a \ct\ where $H_r$ is both highest and \eg.  

\begin{prop}  \label{find reduction} Suppose that $\fG$ and $\filt$ are  a \rtt\ and filtration representing a rotationless $\phi$ and satisfying the conclusions of  Theorem~\ref{2.19}.  Suppose also that $H_{r}$ is an \eg\ stratum and that  $G_s$ is the highest core filtration element  below $G_r$. Then there is an algorithm to decide if $[G_s] \sqsubset [G_r]$   is reduced and if it is not to find a reduction. 
\end{prop}

 \begin{proof}   By \cite[Lemma 3.30]{fh:recognition}, the non-contractible components of $G_r$ are $f$-invariant;   in particular   $H_r$ is contained in a single component of $G_r$. By restricting to this component we may assume  that $H_r$ is the top stratum and hence that $r=N$.

 Let $\Lambda^+ \in \L(\phi) $ be the lamination    associated to $H_N$.  By \cite[Lemma 3.2.4]{bfh:tits1}  there exists $\Lambda^- \in \L(\phi^{-1}) $ such that the smallest free factor system that carries $\Lambda^+$ is the same as the  smallest free factor system that carries $\Lambda^-$ and we denote this by $\F_\Lambda$. It follows that the realizations of $\Lambda^+$ and $\Lambda^-$ in any marked graph cross the same set of edges in that graph.  It also  follows that the smallest \ffs\ that carries $[G_s]$ and $\Lambda^+$  is the same as the smallest free factor system that carries $[G_s]$ and $\Lambda^-$ and we denote this by  $\F_{s,\Lambda}$.   Since both  $[G_s]$ and $\Lambda^\pm$ are $\phi$-invariant, \cite[Corollary~2.6.5]{bfh:tits1} implies that $\F_{s,\Lambda}$ is $\phi$-invariant.  
    
Choose constants as follows.
\begin{itemize}
\item $C_E = 6(n-1)$   is the  maximal number of oriented natural edges in a marked core graph of rank $n$. 
\item  $M =  2n$;   if $\F_0 \sqsubset \F_1 \sqsubset \ldots \sqsubset \F_P$ is an  increasing nested sequence of free factor systems of $F_n$ then  $P \le M-1$.  (This follows by induction on the rank $n$ and the observation that  if $\F_P=\{F_n\}$ with $n>1$ then there exists $\F$ such that $\F_{P-1} \sqsubset \F\sqsubset\F_P$ and $\F$ consists of a pair of conjugacy classes whose ranks add to $n$.)
\item $C_0= C +2$  where $C$ satisfies the conclusion  of Lemma~\ref{critical constant}.  
\end{itemize}    

\vspace{.1in}   
\noindent{\bf Step 1: An existence result when $\F_{s,\Lambda}$  is proper} : 
     Consider the set $\cal K$ of marked graphs $K$ containing a (possibly empty) core subgraph $K_0$ and equipped with a marking preserving homotopy equivalence $p :K \to G$ taking vertices to vertices such that 
\begin{description}
\item [(a)] $p \restrict K_0 :K_0 \to G_s$ is a homeomorphism  (If $s =0$ then $G_s  = K_0 = \emptyset$). 
\item  [(b)] the restriction of $p$ to each natural edge is either an immersion or constant.
\item [(c)]  there is at most one natural edge on which $p$ is constant.  
\end{description}
Each natural edge $E$  of $K\in\cal K$ is {\it labeled} by $p(E)$, thought of as a (possibly trivial) edge path in $G$.  We do not distinguish between two elements of $\cal K$ if there is  a label preserving homeomorphism between them.  The length $|E|$ of a natural edge $E$ is the number of edges in  $p(E)$ and the total length $|K|$ of $K$ is the sum of the length of its natural edges.         The number of $H_N$ edges in $p(E)$ is denoted $|E|_N$.  Let $\Lambda^+_K$  and $\Lambda^-_K$ be the realizations of $\Lambda^+$ and $\Lambda^-$ in $K$.  As observed above, the set of edges crossed by  $\Lambda^+$ is the same as  the set of edges crossed by  $\Lambda^-$.  We denote this common core subgraph by $K_\Lambda$.
       
We claim that if $\F_{s,\Lambda}$  is proper then there exists an element $K \in \cal K$ with the following properties.
 \begin{enumerate}
 \item  $|E|_N \le  C_0$ for all natural edges $E$ of $K$.
 \item The  restriction of $p$ to a leaf of $\Lambda_K^+$ is an immersion.
 \item $K_0 \cup K_\Lambda$ is a proper core subgraph.
 \end{enumerate}
Note that  $K_0 \cup K_\Lambda$ is a core subgraph because it is a union of core subgraphs so the content of (3) is that $K_0 \cup K_\Lambda$ is proper. Our proof of the claim makes use of an idea from the proof of Proposition~3.4 in Part \urn{4}  of \cite{hm:subgroups}.
  
Assuming that $\F_{s,\Lambda}$ is proper, there exists a marked graph  $K$ with proper core subgraphs $K_0 \subset K_1 \subset K$   and  a marking preserving homotopy equivalence $p:K \to G$ satisfying (a) and (b) and $[K_1]=\F_{s,\Lambda} $.  The last property implies that $K_1 = K_0 \cup K_\Lambda$ so (3) is satisfied. If there is at least one natural edge in   $K\setminus K_1$  on which $p$ is an immersion then  collapse each natural edge in  $K$  on which $p$ is constant to a point. Otherwise, collapse each natural edge in $K_1$ on which $p$ is constant and all but  one natural edge  in $K \setminus K_1$  on which $p$ is constant to a point.    The resulting marked graph, which we continue to denote $K$, still satisfies (a) and (b) because $p|K_0$ is injective and so  $K_0$ is unaffected by the collapsing and (c) is now satisfied so $K \in \K$.  Replacing $K_1$ by its image under the collapse, it is still true that  $K_1 = K_0 \cup K_\Lambda$ is a proper core subgraph and now the restriction of $p$ to each natural edge of $K_1$ is an immersion.  (It may now be that $[K_1] $ properly contains $\F_{s,\Lambda}$.)
  
If $p\restrict K_1$ is not an immersion then there is a pair of natural edges $E_1,E_2$ in $K_1$ with the same initial vertex  and such that the edge paths $p(E_1)$ and $p(E_2)$ have the same first edge, say $e$.     Folding the initial segments of $E_1$ and $E_2$ that map to $e$ produces an  element  $K' \in \cal K$ with subgraph $K_0'$ satisfying (a) and such that  $|K'| < |K|$.   Note that (3) is still satisfied    because $K'_0 \cup K'_\Lambda $ is  contained in the image  $K_1' \subset K'$    and  $[K_1] = [K_1']$.  Replacing $K$ with $K'$ and repeating this finitely many times, 
 we may assume that $p \restrict K_1$ is an immersion and hence that $p$ restricts to an immersion on   leaves of $\Lambda_K^+$  and  $\Lambda_K^-$.  In particular, (2) is satisfied.  If  (1) is satisfied then   the proof of the claim is complete.  
 
 Suppose then that there  is a  natural edge $E$ of $K$ such that   $|E|_r >  C_0  = C +2$.   By \cite[Theorem~6.0.1]{bfh:tits1},  a leaf of $\Lambda^-$ is not weakly attracted to a generic leaf of $\Lambda^+$.     Lemma~\ref{critical constant} therefore  implies that   at least one of the two laminations $\Lambda_K^+$ and $\Lambda_K^-$  does not cross $E$.    It follows that    $\F_\Lambda  \sqsubset [K \setminus E]$ and hence that  $E$ is  contained in the complement of $K_\Lambda$.    Since $E$ is obviously in the complement of $K_0$, we have that $E$  is contained in the complement of $K_0 \cup K_\Lambda$.   If there is a natural edge on which $p$ is constant, it must be in the complement of $K_0 \cup K_\Lambda \cup E$ and we collapse it to a point.  As above, the resulting marked graph  is still in $\K$ and (2) and (3) are still satisfied.   We may now assume that $p$ is an immersion on each natural edge of $K$. The map $p: K \to G$ is not a homeomorphism so there is at least one pair of edges that can be folded.  Perform the fold and carry the $K_0 \subset  K$ notation to the new marked graph.  It is still true that $K \in \K$ and that  (2) holds.  Folding reduces $|E|_r$ by at most $2$ so there is still an edge with $|E|_r  > C$.  Arguing as above we see that   $E$  is contained in the complement of $K_0 \cup K_\Lambda$ so (3) is still satisfied.  If  (1) is satisfied then   the proof of the claim is complete. Otherwise perform another fold.  Conditions   (2) and (3)    are satisfied so check condition (1) again.    Folding reduces $|K|$ so after finitely many folds, (1) is satisfied and the claim is proved.

\vspace{.1in}   
\noindent{\bf Step 2:   Part 1 of the algorithm} :   In this step we present an algorithm that     either finds  a reduction for $[G_s] \sqsubset [G_r]$ or   concludes that $\F_{s,\Lambda}$ is improper, i.e.  $\F_{s,\Lambda} = \{[F_n]\}$.  In the former case we are done.  In the latter case we move on to the second part of the algorithm in which we  either find  a reduction for $[G_s] \sqsubset [G_r]$  or we conclude that $[G_s] \sqsubset [G_r]$ is irreducible.    
\begin{description}
\item  [(A1)]    Choose a generic leaf $\gamma \subset G$ of $\Lambda^+$.  One way to do this is to  choose an edge $e$ of $H_r$ and $k > 0$ so that at least one occurrence of $e$ in the edge path  $f^k_\#(e)$ is neither the first nor last $H_r$ edge.     Then $e \subset f^k_\#(e) \subset f^{2k}_\#(e) \subset \ldots$ and the union of these paths is a generic leaf of $\Lambda^+$ by  \cite[Corollary 3.1.11 and Lemma 3.1.15]{bfh:tits1}. 
\item  [(A2)]
Let $C_1= 1$.  Choose  a subpath $\gamma_1$ of $\gamma$  that  crosses at least $$(C_E+1)C_0+ C_E(C_1 +C_0) +2C_0 $$ edges in $H_r$ and let $L_1$ be the number of edges in $\gamma_1$.  (The choice of these constants will be clarified in Step 3.)

\item  [(A3)] \label{item:first time through} Enumerate all   core graphs $J$ of rank   $<n$ satisfying: 
there is a core subgraph $J_0 \subset J$ and  a map $p :J \to G$ taking vertices to vertices such that the restriction of $p$ to each natural edge   is  an immersion onto a path of length at most $L_1$ and      such that $p \restrict J_0 :J_0 \to G_s$ is a homeomorphism. Label the natural edges of $J$  by their $p$ images. We    do not distinguish between    labeled graphs that differ by a label preserving homeomorphism so there are only finitely many  $J$  and we consider them one at a time. If $\sigma \subset J$ is a path with endpoints at  natural vertices and if $p$ restricts to an immersion on $\sigma$ then we let $|\sigma|$ be  the number of edges crossed by $p(\sigma)$ and $|\sigma|_r$ the number of $H_r$ edges crossed by $p(\sigma)$.  
      
       Let $p_\#$ be the homomorphism induced by $p:J\to G$ on fundamental groups.  By {Lemma~\ref{free factor support}} we can decide if $p_\#$ is an isomorphism to a \ffs\ $[J]$  of $F_n$.  If not, then move on to the next candidate.  If yes, then apply Lemma~\ref{checking on ffs} with $\F_0 = [J_0] = [G_s]$ and $\F_1 = [J]$.  If this produces a $\phi$-invariant   $\F \sqsubset [J]$ that properly contains $[G_s]$ then we have found a reduction and the algorithm stops.  Otherwise we know that  $[J]$ does not contain such an $\F$. In  particular $[J]$ does not contain  $\F_{s,\Lambda}$ and so does not carry $\Lambda^+$. Choose  a finite subpath $\gamma_{1,J} \subset \gamma \subset G$ that does not lift to $J$.
       
           By \cite[Lemma 3.1.10(4)]{bfh:tits1} there exists an edge $E$ in $H_r$ and $k \ge 1$ such that $\gamma_{1,J}$ is a subpath of $\ti f^k_\#(E)$. By  \cite[Lemma 3.1.8(3)]{bfh:tits1} there is a computable $k_0$ so that for any edge $E'$ in $H_r$ and any $l \ge k + k_0$, \  $\ti f^k_\#(E)$, and hence $\gamma_{1,J}$, is a subpath of $f^l_\#(E')$.  Finally, by \cite[Lemma 3.1.10(3)]{bfh:tits1} there exists computable $C_{2,J} > 0$ so that if $\sigma$ is a subpath of $\gamma$ that crosses  $\ge C_{2,J}$ edges of $H_r$  then $\sigma$ contains some $f^l_\#(E')$, and hence contains $\gamma_{1,J}$,  as a subpath  and so does not lift into $J$.  Now move on to the next candidate.  
       
At the end of the  process we have either found a reduction and the algorithm stops or we have found  a constant $C_2  = \max\{C_{2,J}\}$ so that if $\sigma$ is a subpath of $\gamma$ that crosses at least $C_2$ edges of $H_r$    then $\sigma$ does not lift into any $J$.           Choose a subpath $\gamma_2$  of  $\gamma$  that crosses at least $$(C_E+1)C_0+ C_E(C_2 +C_0) +2C_0  $$ edges of $H_r$         and let $L_2$ be the number  of edges crossed by $\gamma_2$.
       
\item     [(A4)]  \label{item:second time}   Repeat  (A3)  replacing $L_1$ with $L_2$.      At the end of the  process we have either found a reduction and the algorithm stops or we have found  a constant $C_3 $ so that if $\sigma$ is a subpath of $\gamma$ that crosses at least $C_3$ edges of $H_r$   then $\sigma$ does not lift into any $J$. Choose a subpath $\gamma_3$  of  $\gamma$  that crosses at least $$(C_E+1)C_0+ C_E(C_3 +C_0) +2C_0  $$ edges of $H_r$ and let $L_3$ be the number of edges crossed by $\gamma_3$.

\item  [(A5)]  Iterate this process up to $M$ times.   If after $M$ iterations the algorithm has not found a reduction and stopped, then stop and conclude that $\F_{s,\Lambda} = \{[F_n]\}$.  
  \end{description}
  
       \vspace{.1in}   
\noindent{\bf Step 3:   Justifying part 1 of the algorithm}: In this step    we   verify that if $\F_{s,\Lambda}$ is proper then  the above algorithm finds a reduction in at most $M$ steps.

Suppose $K\in \K$ satisfies (1) - (3) and let $K(L_1)$ be the subgraph of $K$ consisting of natural edges $E$ with $|E| \le L_1$.  Lift the path $\gamma_1$ into $K_\Lambda$.  After removing initial and terminal subpaths contained in single natural edges of $K_\Lambda$, we have a natural (in $K$) edge path $\gamma_{1,K} \subset \Lambda^+_K \subset K_\Lambda$ that projects onto all of   $\gamma_1$ except perhaps initial and terminal segments that cross at most $C_0$ edges of $H_r$.  Thus $\gamma_{1,K}  \subset K(L_1)$ and 
      $$|\gamma_{1,K}|_r \ge (C_E+1)C_0+ C_E(C_1 +C_0) $$
 Combining this inequality with    (1), we see that  $\gamma_{1,K}$ has a subpath that decomposes as 
 $$\alpha_1 \beta_1 \alpha_2\ldots \beta_{C_E} \alpha_{C_E+1}$$ where each $\alpha_i$ is a single natural edge and each $\beta_i$ is a natural subpath whose image under $p$ crosses at least $C_1$ edges in $H_r$.   Since there are at most $C_E$ oriented natural edges  in $K$, it follows that $\gamma_{1,K}$  contains a natural subpath that begins and ends with the same oriented natural edge and whose $p$-image crosses at least  $  C_1 =1$ edges in $H_r$. This proves that that there is a circuit in $ K(L_1)$ that is not in $K_0$ and hence  that $[K(L_1)]$ properly contains $[K_0]$.  By (3), $K_0\cup K_\Lambda$ is proper and so if $K_0 \cup K_\Lambda \subset K(L_1)$ then    $K_0 \cup K_\Lambda$ occurs as a $J$ in the first iteration of the process and    Lemma~\ref{checking on ffs} finds a reduction in the first iteration of  (A3) because   $\F_{s,\Lambda} \sqsubset [K_0 \cup K_\Lambda]$.

If no reduction is found in the first iteration then proceed to the second iteration as described in (A4).  By the same reasoning, there is a natural edge path $\gamma_{2,K} \subset K(L_2)$  such that $$|\gamma_{2,K}|_r \ge (C_E+1)C_0+ C_E(C_2 +C_0) $$ and there is a subpath of $\gamma_{1,K}$  that decomposes as 
 $$\alpha_1 \beta_1 \alpha_2\ldots \beta_{C_E} \alpha_{C_E+1}$$ where each $\alpha_i$ is a single natural edge and each $\beta_i$ is a natural subpath whose image under $p$ crosses at least $C_2$ edges in $H_r$.  It follows that there is a circuit in $ K(L_2)$ that is not in $K(L_1)$ and hence that   $[K(L_2)]$ properly contains $[K(L_1)]$.  If $K_0 \cup K_\Lambda \subset K(L_2)$ then    $K_0 \cup K_\Lambda$ occurs as a $J$ in the second iteration of the process and  our algorithm finds a reduction.
 
Continuing on, the iteration either produces a reduction within $M$ steps or produces a properly nested sequence $$[K_0] \sqsubset [K(L_1)] \sqsubset [K(L_2)] \sqsubset \ldots \sqsubset [K(L_M)]$$ of free factors.  Since the latter contradicts the definition of $M$, a reduction must have been found.
\vspace{.1in}

\vspace{.1in}   
\noindent{\bf Step 4:   Part 2 of the algorithm}: In this part of the algorithm we assume that $ \F_{s,\Lambda}$ is improper.   The filtration element $G_u = G_N \setminus H_N^z$ deformation retracts to $G_{s}$, see Remark~\ref{r:core}.  In particular, $[G_u] = [G_{s}]$.
    
Given a circuit $\sigma \subset G$ [resp.\ a  subgroup $A$] let $\F_{u, \sigma}$ [resp.\ $\F_{u,A}$] be the smallest free factor system that carries $[\sigma]$ [resp.\ every conjugacy class in $[A]$] and every conjugacy class in $[G_u]$. By Lemma~\ref{free factor support},  $\F_{u,\sigma}$ and  $\F_{u,A}$ can be algorithmically determined. 
   
By Lemma~\ref{finding iNps}, the set $P $ of indivisible periodic Nielsen paths with height $r$ is finite, can be determined algorithmically  and  each element of $P$  has period one.  For notational convenience we assume that $P$ is closed under orientation reversal. Let $\Sigma$ be the set of circuits in $G$ that split into a  concatenation of paths  in $G_u$  and elements of $P$.   Lemma~\ref{ANA} implies that $\Sigma$ is the set of circuits that are not weakly attracted to $\Lambda^+$ and that the conjugacy classes determined by $\Sigma$ are exactly those carried by the subgroup system  $\A$ of that lemma. 

We consider several cases, each of which can be checked by inspection of $G_u$ and the elements of $P$.  If  every circuit in $\Sigma$ is contained in $G_u$ then $\A = [G_u]$ and $H_{N}$ is reduced by  Lemma~\ref{ANA}\pref{item:characterize irreducibility}.
       
As a second  case, suppose that  there is a non-trivial path $\mu \subset G$ with endpoints   $x,y \in G_u$ such that $\mu$ is  homotopic rel endpoints to a concatenation of elements of $P$ and so   is a Nielsen path of height $N$.  Let $[B]$ be the conjugacy class of the subgroup of $F_n$ represented by closed paths based at $x$ that decompose as a concatenation of subpaths, each of which is either $\mu, \mu^{-1}$ or a path in $G_u$ with endpoints in $\{x,y\}$. Then $[B]$ is $\phi$-invariant, has  rank $\ge 2$,  and each conjugacy class in $[B]$ is represented by a circuit $\sigma \in \Sigma$.  By  Lemma~\ref{ANA}\pref{item:characterize irreducibility} there is  a proper free factor system that carries $[G_u]$ and the conjugacy class of each element of $B$.   It follows that  $\F_{u,B}$ is proper.  Since both $[G_u]$ and $[B]$ are $\phi$-invariant,    $\F_{u,B}$ is $\phi$-invariant by \cite[Corollary~2.6.5]{bfh:tits1}   and we have found   a reduction of $H_N$.
  
The  final case is  that there are no paths $\mu \subset G$ as in the second case and there is at least one  element $\sigma \in \Sigma$ that is not contained in $G_u$.  Each such $\sigma$ is homotopic to a concatenation of elements of $P$.  In particular,  the conjugacy class determined by $\sigma$ is $\phi$-invariant and so   $\F_{u,\sigma}$   is $\phi$-invariant. Choose one such $\sigma$ and check if $\F_{u,\sigma}$  is proper.  If it is then we have found a reduction of $H_N$  and we  are done   so suppose that it is not.   Lemma~\ref{ANA}\pref{item:characterize irreducibility} implies that $[\sigma]$ is carried by a rank one component $[A]$ of $\A$ and that $\A = \F' \cup [A]$ for some $\phi$-invariant free factor system $\F'$. If  there exists an element $\sigma' \in \Sigma$ that is not carried by $[G_u]$ and such that $\sigma$ and $\sigma'$ are not multiples of the same root-free circuit then $[\sigma']$ is not carried by $[A]$ so $\F_{u,\sigma'}$ is a reduction. Otherwise,   $\F' = [G_u]$ and   $H_r$ is reduced.  This completes the second step in the algorithm and so also the proof of the proposition.
\end{proof}

\subsection{The \noneg\ case}
We now consider reducibility for \noneg\ strata, beginning with a pair of examples. 

\begin{ex} \label{finding a loop}  Suppose that $\fG$ is a homotopy equivalence with filtration $\filt$ representing $\phi \in \Out(F_n)$ and that  $G_{r+2} = G_r \cup E_{r+1} \cup E_{r+2}$   where $G_r$ is a  connected core subgraph and where $H_{r+1} = E_{r+1}$ and $H_{r+2}=E_{r+2}$ are oriented edges with a common initial vertex not in $G_r$ and a common terminal endpoint in $G_r$.  Suppose also that   $f(E_{r+1}) = E_{r+1}u$ and $f(E_{r+2}) = E_{r+2}u$ for some closed  non-trivial path $u \subset G_r$    and  that $f\restrict G_r$ is a \ct.     Then the (NEG Nielsen Paths) property of  $f \restrict G_{r+2}$ fails.   The \ct\ algorithm corrects this (see Section~\ref{section one edge}) by discovering that $E_{r+2} \bar E_{r+1}$ is a Nielsen path and then sliding the terminal endpoint of $E_{r+2}$ along $\bar E_{r+1}$.  In other words, $E_{r+2}$ is replaced by a fixed loop $E'_{r+2}$ based at the initial  endpoint of $E_{r+1}$.  Note that while establishing (NEG Nielsen Paths) for $f \restrict G_r$, we have discovered a reduction of $[G_r] \sqsubset [G_{r+2}]$. Namely, the  $\phi$-invariant \ffs\   $\{[G_r], [E'_{r+2}]\}$ is properly contained between $[G_{r+1}] = [G_r]$ and $[G_{r+2}]$.    
\end{ex}  

\begin{ex} \label{fixed edge}  Suppose that $\fG$ is a homotopy equivalence with filtration $\filt$  representing $\phi \in \Out(F_n)$, that  $G_{r+1} = G_r \cup E_{r+1}$   where $G_r$ is a connected  core subgraph such that $f\restrict G_r$   is a \ct\  and where $H_{r+1} = E_{r+1}$ is a fixed edge whose initial and terminal endpoints, $x$ and $y$, belong to the same Nielsen class of $f \restrict G_r$.      Then  $f \restrict G_{r+1}$ satisfies all the properties of a \ct\ except that $[G_r] \sqsubset [G_{r+1}]$ is not reduced.     The \ct\ algorithm corrects this in stages.  First,  a Nielsen path $\sigma \subset G_r$ connecting $y $ to $x$ is found.  Then the terminal end of $E_{r+1}$ is slid along $\sigma$ so that its new terminal endpoint is $x$.  Thus $E_{r+1}$ is replaced by a fixed edge $E'_{r+1}$ with both endpoints at $x$.  Finally, $x$ is blown up to a fixed edge   $E'_{r+2}$ with $E'_{r+1}$ attached at the \lq new\rq\  endpoint of $E'_{r+2}$ and the  remaining edges in the link of $x$ still attached at $x$.    Both $[G_r] \sqsubset [G_{r+1}]$ and $[G_{r+1}] \sqsubset [G_{r+2}]$  are reduced.
\end{ex}  

\begin{lemma}\label{reducible noneg}    Suppose that $\fG$ and $\filt$ are a \rtt\ and filtration  satisfying the conclusions of Theorem~\ref{2.19} and representing a rotationless $\phi \in \Out(F_n)$. Suppose further  that $G_s$ is  core, that $H_s = \{E_s\}$ is an \noneg\ stratum, that $C$ is the component of $G_s$ that contains $H_s$  and that $f \restrict C$ satisfies  (\noneg\ Nielsen Paths). If $[C \setminus E_s] \sqsubset [C]$   is  reducible then   $E_s$ is fixed and its terminal endpoint is connected to its initial endpoint by a Nielsen path $\beta \subset C\setminus E_s$.  In particular, $E_s\beta$ is a basis element and    $\{[G_{s-1}],[E_s\beta]\}$ is a $\phi$-invariant \ffs\ that is properly contained between $[G_{s-1}]$ and $[G_s]$.  \end{lemma}

\proof  By restricting to $C$,   we may assume that $G = C = G_s$ and hence that $G_{s-1}$ has either a single component of rank $n-1$ or two components whose ranks add to $n$.   If $[C \setminus E_s] \sqsubset [C]$   is  reducible,  there is a marked graph $K$  with distinct proper  core subgraphs $K_1 \subset K_2   \subset K$ such that $[K_1] = [G_{s-1}]$ and such that $[K_2]$ is $\phi$-invariant.     From  $$-\chi(K_1) \le -\chi(K_2) < -\chi(K) = -\chi(K_1) +1$$  it  follows that $K_2$ is obtained from $K_1$ by adding a disjoint loop $\alpha$ and that $K \setminus K_2$ is an edge $E$.  In particular,   $K_1$ is connected.

We claim that the circuit $\sigma \subset G$ representing  $[\alpha]$ crosses $E_s$ exactly once.    The marked graph $K'$ obtained from $K$ by collapsing the components of  a maximal forest is a rose with $\alpha$ as one of its edges.  
The marked graph $G'$ obtained from $G$ by collapsing the components of  a maximal forest in  $G_{s-1}$ is a rose with  $E_s$ as one of its edges.  Moreover $[K'\setminus \alpha] = [K_1] = [G_{s-1}] = [G' \setminus E_s]$. Let   $h : K' \to G'$ be a homotopy equivalence that respects markings and that restricts to an immersion on each edge.  Then $h
(K'\setminus \alpha) = G' \setminus E_s$ and it suffices to show that $h(\alpha)$  crosses $E_s$ exactly once.  This follows from \cite[Corollary 3.2.2]{bfh:tits1}.

 Having verified the claim, we can now complete the proof of the lemma.  Since $[K_2]$ is $\phi$-invariant and $\alpha$ is a component of $K_2$, $\alpha$ determines a $\phi$-invariant conjugacy class.  It follows that $\sigma$ decomposes as a concatenation of \iNp s and fixed edges.  Since $\sigma$ crosses $E_s$ exactly once, the (NEG Nielsen Paths)  property of $\fG$ implies that $E_s$  is a fixed edge.  Thus $\sigma$ decomposes as a circuit into $E_s \beta$ where $\beta$ is a Nielsen path in $G_{s-1}$.  
 \endproof
 
\section{Definition of a \ct}    \label{the definition}
The definition of a \ct\ evolved over many years, growing out of improved relative trains \cite{bfh:tits1} which grew out of relative train tracks \cite{bh:tracks}.  In addition to the nine items that make up the formal definition,   there are numerous auxiliary consequences of these definitions that are used repeatedly.   We have included the complete definition  here  for the reader's convenience but recommend consulting  \cite{fh:recognition} (see also \cite[ Part \urn{1} Section 1.5]{hm:subgroups})  for discussion and elaboration. 

We have already discussed some of the technical terms in the definition of \ct.   
\begin{itemize}
\item  For basics of relative train track theory, including the definitions of \rtt s, filtrations,  \eg\ and \noneg\ strata, linear edges and  Nielsen paths   see Section~\ref{standard}.
\item  For zero strata enveloped by \eg\ strata see Notation~\ref{n:zero}.
\item  For principal vertices and rotationless $\fG$ see Definition~\ref{def:principal auto}. 
\item  For the definition of a filtration being reduced see the beginning of Section~\ref{section:reducibility}.
\end{itemize}

There are a few more  terms that need defining. 

\begin{definition} \cite[Definition 4.1]{fh:recognition}
If $w$ is a closed root-free Nielsen path and $E_i, E_j$ are linear edges satisfying $f(E_i) = E_i w^{d_i}$ and $f(E_j) = E_i w^{d_j}$ for distinct $d_i, d_j  > 0$ then a path of the form $E_i w^* \bar E_j$  is called an {\em exceptional path}.   
\end{definition}

 \begin{definition}  \cite[Definition 4.3]{fh:recognition}    If $H_r$ is \eg\ and $\alpha \subset G_{r-1}$  is a non-trivial path 
with endpoints in $H_r \cap G_{r-1}$ then we say that $\alpha$ is a   {\em connecting path  for $H_r$}.  If $E$ is an edge in an irreducible 
stratum $H_r$  and  $k >0$ then a maximal subpath $\sigma$ of $f^k_\#(E)$ in a zero stratum $H_i$  is said to be {\em $r$-taken}  or just {\em taken}  if $r$ is irrelevant. A non-trivial path or circuit $\sigma$ is {\em completely split} if it has a splitting, called a {\it complete splitting}, into subpaths, each of which is either a single edge in an irreducible stratum, an \iNp, an exceptional path  or    a connecting path in a zero stratum $H_i$ that is both maximal (meaning that it is not   contained in a larger subpath of $\sigma$ in $H_i$)  and taken.  Note that the endpoints, if any, of a completely split path are at vertices.
\end{definition}

\begin{definition}    [Definition 4.4 \cite{fh:recognition}] A \rtt\ is {\em completely split} if 
\begin{enumerate}
\item  $f(E)$ is completely split for each edge $E$ in each irreducible 
stratum.  
\item If $\sigma $ is a taken connecting path  in a zero stratum   then  $f_\#(\sigma)$ is completely split.   \end{enumerate}
\end{definition}

Proper extended folds, which are referred to in the (\eg\ Nielsen Paths) property of a \ct\ are defined in \cite[Definition 5.3.2]{bfh:tits1}.  We will not review that here, in part because the actual definition is never used in this paper and in part because in applications one almost always refers to a consequence of this property (for example \cite[Corollary 4.19]{fh:recognition}) rather than to the property itself.

\begin{definition} \label{def:ct} A   \rtt\ $\fG$ and filtration $\F $ given by $ \filt$\ is said to be a {\em \ct}\ (for completely split improved relative train track map) if it satisfies the following properties.   
\begin{enumerate}

\item{\bf(Rotationless)} $\fG$ is rotationless.(Definition~\ref{def:principal auto})
\item{\bf(Completely Split)} $\fG$ is completely split. 
\item {\bf (Filtration)}  $\F$ is reduced.   (Section~\ref{section:reducibility}) The core of each filtration element is a filtration element.   
\item {\bf (Vertices)} The endpoints of all indivisible periodic (necessarily fixed) Nielsen paths are (necessarily principal) vertices.   The terminal endpoint of each non-fixed \noneg\ edge is principal (and hence fixed).   
\item {\bf(Periodic Edges)} Each periodic edge is fixed and each endpoint of a fixed edge is principal.    If the unique edge $E_r$ in a fixed stratum $H_r$  is not a loop then $G_{r-1}$ is a core graph  and both ends of $E_r$ are contained in $G_{r-1}$. 
\item {\bf (Zero Strata)}  If $H_i$ is a zero stratum, then  $H_i$ is enveloped by an  \eg\ stratum $H_r$, each edge in $H_i$ is $r$-taken and each vertex in $H_i$ is contained in $H_r$ and has link  contained in $H_i \cup H_r$.    
\item {\bf(Linear Edges)}  For each linear $E_i$ there is a closed root-free Nielsen path $w_i$  such that $f(E_i) = E_i w_i^{d_i}$ for some $d_i \ne 0$.     If $E_i$ and $E_j$ are distinct linear edges  with the same axes  then $w_i = w_j$ and  $d_i \ne d_j$.   
\item {\bf (\noneg\ Nielsen Paths)} 
If the highest edges in an \iNp\ $\sigma$  belong to an \noneg\ stratum then there is a  linear edge $E_i$ with $w_i$  as in (Linear Edges) and there exists $k \ne 0$ such that $\sigma = E_i w_i^k \bar E_i$.   
\item {\bf (\eg\ Nielsen Paths)}  If $H_r$ is \eg\  and $\rho$ is an \iNp\ of height $r$, then   $f|G_r = \theta\circ f_{r-1}\circ f_{r}$  where :
\begin{enumerate}
\item $f_r : G_r \to G^1$ is a composition of proper extended folds  defined by iteratively folding $\rho$.
\item $f_{r-1} : G^1 \to  G^2$ is a composition of   folds   involving edges in $G_{r-1}$.
\item $\theta : G^2 \to G_r$ is a homeomorphism.
\end{enumerate}
\end{enumerate}
\end{definition}

We include the following for future reference.

\begin{lemma}\label{l:power of ct is ct}
 If $\fG$ is a \ct\ then $f^k : G \to G$  is a \ct\ for all $k \ge 1$. 
\end{lemma}

\proof  By  \cite[Lemma 4.13]{fh:recognition},    every periodic Nielsen path for $f$ has period one.   With this in hand, the first eight \ct\ properties for $f^k$ are easy to check.  The remaining property (EG Nielsen Paths) for $f^k$ follows from  \cite[Corollary 4.33]{fh:recognition}.
\endproof

The following example shows that the restriction of a \ct\ $f$ to a component of a core filtration element need not be a \ct. 

\begin{ex}\label{e:example}
We refer to Figure~\ref{f:fillgap} for notation. The map $f:G\to G$ given by $e\mapsto e$, $b\mapsto be$, $c\mapsto c$, and $d\mapsto de^2$ is a \ct\ with the filtration 
$$\emptyset\subset\{e\}\subset \{e,b\}\subset\{e,b,c\}\subset G$$
and the restriction of $f$ to each filtration element is a \ct. If we however consider the new filtration 
$$\emptyset\subset\{e\}\subset\{e,c\}\subset\{e,c,d\}\subset G$$
then $f$ is still a \ct\ with the new filtration but $f|\{e,c,d\}$ is not a \ct\ because it does not satisfy (Vertices).  
 \begin{figure}[h!]
 \centering
 \includegraphics[width=0.25\textwidth]{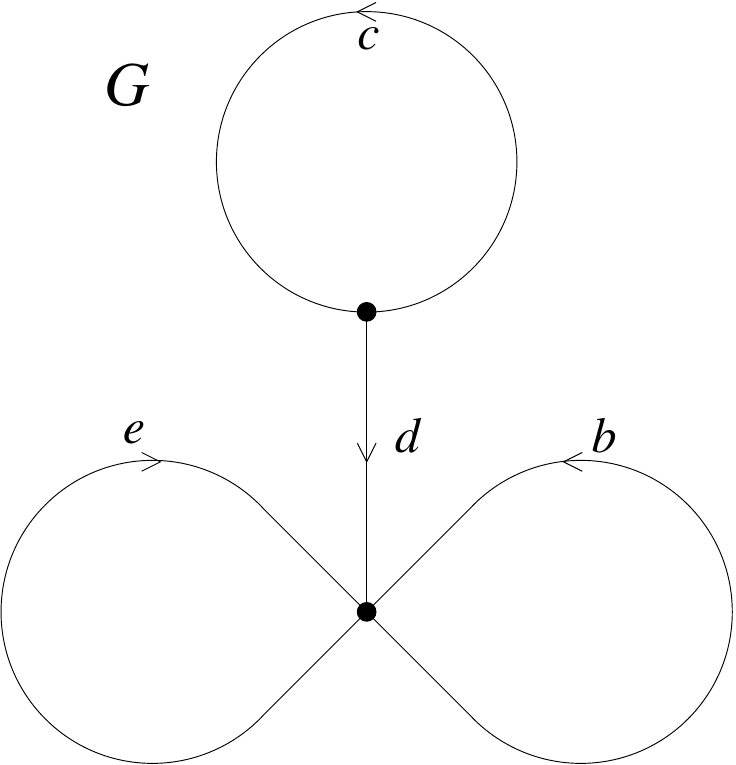}
 \caption{Example~\ref{e:example}}
 \label{f:fillgap}
  \end{figure}
\end{ex}

\section{Sliding \noneg\ edges}\label{section:sliding}
The key step for arranging that an \noneg\ edge has good properties under iteration is to slide the terminal endpoint of the edge into an optimal position in the lower filtration element that contains it.   This is carried out in \cite[Proposition 5.4.3]{bfh:tits1}.  In this section we make the algorithmic arguments needed to replace the non-algorithmic parts of the original proof. 
 
\subsection{Completely Split Rays}
Recall from \cite[Definition~4.4]{fh:recognition} that a splitting  $\sigma = \sigma_1 \cdot \sigma_2 \ldots $  is a {\em complete splitting}  if each $\sigma_i$  is either a single edge in an irreducible stratum, an \iNp, an exceptional path or a maximal subpath in a zero stratum (with some additional features that we will not recall here.) A finite path or circuit has at most one complete splitting by \cite[Lemma 4.11]{fh:recognition}. The first item of our next lemma states that the same is true for rays.
 
Recall from Section~\ref{standard} that  a ray $\sigma$ with initial point a vertex is thought of as an edge path $\sigma = E_0 E_1 \ldots$.   The initial and terminal endpoints, $w_i$ and $w_{i+1}$  of $E_i$ are  the {\em vertices of $\sigma$}.  We view the set $\W$ of vertices of a path as being ordered by their subscripts.  A decomposition of $\sigma$ into subpaths is specified by a subset of $\W$; if $w_i$ and $w_j$ are consecutive elements of the subset then $E_i \ldots E_{j-1}$ is a term in the decomposition.   If the decomposition is a splitting then we refer to these vertices as {\em splitting vertices}.  A similar definition holds for circuits.  
 
\begin{lemma}  \label{unique splittings} Suppose that $\fG$ is  a \ct, that $R \subset G$ is a completely split ray and that $R_0$ is a  subray of $R$ that has a complete splitting. Then
\begin{enumerate}  
\item The complete splitting of $R$ is unique.  
\item  Let $v$ be the first splitting vertex for $R$ that is contained in $R_0$ (when the edge path $R_0$ is viewed as a   subpath of the edge path  $R$).  Then each splitting vertex $w$ for $R_0$  that comes after $v$ (in the ordering of splitting vertices of $R_0$)  is a splitting vertex for $R$. 
\end{enumerate}
\end{lemma}

\proof   The second item implies the first by taking $R_0 =R$ so we need only prove the second.  Let $\mu_0$ be the term in the complete splitting of $R_0$ whose initial vertex is $w$. If $\mu_0$ is either an \iNp\ or an exceptional path then the interior of $\mu_0$ is an increasing union of pre-trivial paths by \cite[Remark 4.2 and Lemma 2.11(2)]{fh:recognition} and so by \cite[Lemma 4.11(2)]{fh:recognition} is contained in a single term $\mu$ of the complete splitting of $R$.  Obviously $\mu$ is not a single edge and is not contained in a zero stratum so it must be either an \iNp\ or an exceptional path.  Since $v$ is the initial endpoint of some term in the complete splitting of $R$  and $w$ comes after $v$,  it follows that $v$ is not contained in the interior of $\mu$ and so  $\mu \subset  R_0$.  The symmetric argument therefore applies to show that $\mu$ is contained in a term of the complete splitting of $R_0$ and hence that $\mu = \mu_0$ as desired.  If $\mu_0$ is either a single edge or is a maximal subpath in a zero stratum and $\mu_0$ is not a term in the complete splitting of $R$ then $\mu_0$   is properly contained in a term $\mu$ of $R$ that is an \iNp\ or an exceptional path.  But this violates the hard splitting property \cite[Lemma 4.11(2)]{fh:recognition} for the complete splitting of $R_0$ (applied to its finite completely split subpaths) and the fact that the interior of $\mu$ is the increasing union of pre-trivial paths.   Thus $\mu_0$ is a term in the complete splitting of $R$ and we are done.
\endproof

\begin{defns} \label{d:generates a ray} Suppose that $\fG$ is a \ct, that $x \in G$,  that  $\sigma \subset G$  is a non-trivial completely split path connecting $x$ to $f(x)$ and that  the turn at $f(x)$ determined by $\bar \sigma$ and $f_\#(\sigma)$ is legal.  The ray  $R = \sigma \cdot f_\#(\sigma) \cdot f^2_\#(\sigma)\cdot \ldots$  satisfies $f_\#(R) \subset R$ and the given splitting of $R$   has a refinement that is a complete splitting by  \cite[Lemma 4.11]{fh:recognition}; we say that {\em $R$ is generated by $\sigma$}.

If  $E$ is a non-fixed  edge of $G$ whose initial direction is fixed then $f(E) = E \cdot \sigma$ for some $\sigma$ as above.  The ray $R_E = E \cdot\sigma \cdot f_\#(\sigma) \cdot f^2_\#(\sigma)\cdot \ldots$ is the {\em eigenray} determined by $E$.  Note that we are not requiring that the initial vertex of $E$ be principal (as we did in Section~\ref{s:uniform bound}) or that $E$ is non-linear and \noneg\ so we are using the term eigenray a little more generally than is sometimes the case.   We will need this inclusiveness in the proof of Lemma~\ref{finding a fixed point}.  For the same reason we assume that each isolated fixed point for $f$ is a vertex.
\end{defns}

\begin{lemma}  \label{comparing rays} Suppose that $\fG$ is a \ct\ and that $\sigma$ and $\sigma'$ are completely split non-Nielsen paths generating rays $R$ and $R'$ respectively.  Then there is an algorithm to decide if  the rays $R$ and  $R'$   have a common terminal subray and if so to find initial subpaths $\tau \subset R$ and $\tau ' \subset R'$ that terminate at splitting vertices of $R$ and $R'$ respectively and whose complementary terminal subrays are equal.  Equivalently we find splitting vertices $v \in R$ and $v' \in R'$ such that terminal subrays of $R$ and $R'$ initiating at $v$ and $v'$ are equal (as edge paths).
\end{lemma}

\proof   
Let $\V=\{  v_0,  v_1,\ldots \}$ be the set of splitting vertices for $ R$ ordered so that $ v_{i-1}$ and $ v_i$ are the endpoints of the $i^{th}$ term in the complete splitting.    By construction,   $ f( \V) \subset  \V$.  For each $i \ge 0$, let $\sigma_i \subset  R$ be the path connecting $v_i$ to $ f( v_i)$ and let $\ell_i = |\sigma_i|$ be the number of edges in $\sigma_i$; in particular, $\sigma = \sigma_0$.   Note that $\sigma_i$ generates the terminal subray of $R$ that begins with $v_i$.   Define $ V',  \sigma_j'$ and $\ell_j'$  similarly using $\sigma'$ and $R'$ in place of $\sigma$ and $R$. It is obvious that  $R$ and $R'$ have a common terminal subray if $\sigma_i = \sigma'_j$ for some $i$ and $j$. (Namely,  the subrays of $R$ and $R'$  initiating at $v_i \in \V'$ and $v'_j \in \V'$ respectively.)     The converse follows from  Lemma~\ref{unique splittings}.   Our goal then is to either find $i$ and $j$ such that $\sigma_i = \sigma'_j$ or to conclude that no such $i$ and $j$ exist.

Let $r$ [resp $r']$ be  the maximal height of  a term in the complete splitting of $\sigma$  [resp. $\sigma'$] that is not a Nielsen path.  Since $f(\sigma)$ and $\sigma$ have a common endpoint,    $\sigma$ is not entirely contained in a zero stratum.  Thus any term $\tau$  in the complete splitting of $\sigma$ that is contained in a zero stratum is adjacent to a term that intersects the \eg\ stratum that envelops $\tau$ (Notation~\ref{n:zero}).  Since  this adjacent edge has at least one non-fixed  endpoint, it is neither an exceptional path nor   a Nielsen path so must be  a single edge in that \eg\ stratum. We conclude that $H_r$ is not a zero stratum.   It follows    that if $\mu$  is any height $r$ term   in the complete splitting of $\sigma$ and if $\mu$ is not a Nielsen path then  the length of $f^i_\#(\mu)$ goes to infinity with $i$;    we say that $\mu$ is  {\em growing}.      Note that $r$  is the maximal height of a growing term in the complete splitting of any $\sigma_i$ and similarly for $r'$ and $\sigma'_j$. Note also that for any given $L > 0$   one can find, by inspection,  $M>0$ such that $\ell_i, \ell'_j >L$ for all $i,j \ge M$.
 
 The first step in the algorithm is to check if $r = r'$.  If yes, then move on to step two.  If not then there do not exist $i$ and $j$ such that   $\sigma_i = \sigma'_j$  so the algorithm stops and outputs NO.   

 We may now assume that $r = r'$.  If $H_r$ is \noneg\ define $K=1$.  Otherwise $H_r$ is \eg\ and we choose $K$ 
 so that for each edge $E$ of $H_r$, $f^K_\#(E)$ contains at least $C$ edges of $H_r$  where $C$ is the constant of Lemma~\ref{critical constant}.   Now define $I$  to be the number of terms in the complete splitting of   $\sigma
  \cdot f_\#(\sigma) \cdot \ldots \cdot   f_\#^{K}(\sigma) $ or equivalently so that  $ i \le I$ if and only if $\ti v_i \in \ti \sigma \cdot \ti f_\#(\ti \sigma) \cdot \ldots \cdot  \ti  f_\#^{K}(\ti \sigma) $.  Define      $J$  to be the number of  terms in the complete splitting of   $\sigma' \cdot f_\#(\sigma') \cdot \ldots \cdot   f_\#^{K}(\sigma') $ or equivalently so that  $ j \le J$ if and only if $\ti v'_j \in \ti \sigma' \cdot \ti f_\#(\ti \sigma') \cdot \ldots \cdot  \ti  f_\#^{K}(\ti \sigma') $.    
 
Fix $j$.  To check if $\sigma'_j = \sigma_i$ for some $i$ we need only consider $i < M$ where   $\ell_i > \ell'_j$ for all $i \ge M$.   The symmetric argument implies that for any fixed $i$ we can check if $\sigma_i  = \sigma'_j$ for some $j$.  
The second and final step of the algorithm is to decide if there exists $i \le I$ such that  $\sigma_i = \sigma'_j$ for some $j$ or if there exists $j \le J$   such that $\sigma_i = \sigma'_j$ for some $i$. If not then the algorithm outputs NO. If yes then the algorithm outputs YES and $v_i,v'_{j}$.

It remains to prove that if  $R$ and $R'$ have a common subray then the algorithm outputs YES in the  second step.  Suppose that  $R = \rho R''$ and $R' = \rho' R''$ where  $\rho \subset R$ and $\rho' \subset R'$ are finite and  $R''$ is  a  maximal common subray. Lift   $R = \rho R'' \subset G$ to  $\ti R = \ti \rho \ti R'' \subset \ti G$ and the initial segment $\sigma \subset R$ to an initial segment $\ti \sigma \subset \ti R$.     Let $\ti f : \ti G \to \ti G$ be the lift of $f$ that takes the initial endpoint of $\ti \sigma$ to the terminal endpoint of $\ti \sigma$ and note that $\ti R =  \ti \sigma \cdot \ti f_\#(\ti \sigma) \cdot \ti f^2_\#(\ti \sigma)\cdot \ldots$.   Since $\sigma$ is not a Nielsen path,  $|f_\#^k(\sigma)| \to \infty$.  It follows that the terminal endpoint $P \in \partial F_n$ of $\ti R$    is an attractor for the action of $\partial \ti f$ and so is not fixed by any covering translation \cite[Proposition~I.1]{gjll:index}. In particular, $\ti f$ is the only lift of $f$ that fixes $P$.    Lift $R' = \rho' R''$  to  $\ti R' = \ti \rho' \ti R''$ and $\sigma'$ to an initial segment $\ti \sigma'$ of $\ti R'$.  The uniqueness of $\ti f$ implies  that  $\ti R' =  \ti \sigma' \cdot \ti f_\#(\ti \sigma') \cdot \ti f^2_\#(\ti \sigma')\cdot \ldots$.     Let   $\ti E''$ be the first height $r$ edge crossed by $\ti R''$.     There exist unique $k,k' \ge 0$ such that $\ti E''$ is crossed by $\ti f^k_\#(\ti \sigma)$ and by $ \ti f^{k'}_\#(\ti \sigma')$.   We may assume without loss that $k' \ge k$.  
By Lemma~\ref{unique splittings}, it suffices to show that $k \le K$ or equivalently that $\rho \subset \ti \sigma \cdot \ti f_\#(\ti \sigma) \cdot \ldots \cdot \ti f^{K}(\ti \sigma)$. 

If $H_r$ is \noneg\ then by the basic splitting property of \noneg\ edges  \cite[Lemma 4.1.4]{bfh:tits1} there is a unique height $r$ edge whose image under $\ti f^{k}$  crosses $\ti E''$.  Since both $\ti\sigma$ and $\ti f_\#^{k'-k}(\ti \sigma')$ cross such an edge, their intersection is non-empty.    It follows that $\ti \rho \subset \ti \sigma_1$ so we are done.       

If $H_r$ is \eg\ then there exist $\ti x \in \ti \sigma$ and  $\ti x' \in  \ti f^{k'-k}_\#(\ti \sigma')$ such that $\ti f^k(\ti x) = \ti f^k(\ti x')$.  
The   path $\ti \tau$ from $\ti x$ to $\ti x'$ decomposes as a concatenation of subpaths $\ti \alpha   \ti \beta^{-1}$ where   $\ti \alpha \subset \ti \rho$ and $\ti \beta \subset \ti \rho'$. By construction $\ti f^k_\#(\ti \tau)$ is trivial so in particular $\tau$ is not weakly attracted to $\Lambda^+$.   Lemma~\ref{critical constant} implies that $\ti \alpha$ does not cross $C$ edges  that project to $H_r$ and so does not contain $\ti f^K_\#(\ti E)$ for any edge $\ti E$ that projects to  $H_r$.     Since $\ti \sigma$ crosses such an  $\ti E$ it follows that $\ti \alpha$ and hence $\ti \rho$ is contained in $ \ti \sigma \cdot \ti f_\#(\ti \sigma) \cdot \ldots \cdot \ti f^{K}(\ti \sigma)$ as desired. 
 \endproof

\subsection{Finding a Fixed Point}
The proof of Theorem~\ref{algorithmic ct} begins with an arbitrary \rtt\ $\fG$ and filtration and  modifies  $\fG$ and the filtration so that it satisfies more and more of the \ct\ properties.   Certain steps in this process are  inductive and involve consideration of a,  not necessarily core, component of a filtration element. Lemma~\ref{finding a fixed point} below is applied in that context. The reader will note that essentially all of the arguments take place in a core filtration element.

We state our next result in terms of a lift $\ti f : \ti G \to \ti G$ of $ f:G \to G$ and fixed points for $\ti f$. We could just as easily have stated it in terms of finite paths  $\rho \subset G$ as described in Section~\ref{s:markings} but it seems more natural to work with lifts.

Following Definition~\ref{d:generates a ray} we say that a completely split path $\ti \sigma \subset \ti G$ {\em generates a completely split ray} $\ti R$ if $\ti R = \ti \sigma \cdot \ti f_\#(\ti \sigma) \cdot \ti f^2_\#(\ti \sigma)\cdot \ldots$.  As we have seen,  $\ti f_\#(\ti R) \subset \ti R$ and  $\ti f$ maps the set of splitting vertices for $\ti R$ into itself.

\begin{lemma}  \label{finding a fixed point}  Suppose that $\fG$ is a homotopy equivalence of a connected finite graph, that  $\emptyset =G_0 \subset G_1 \subset \ldots \subset G_K =G$ is an $f$-invariant filtration and that there is a connected core filtration element $G_r$ such that $f \restrict G_r$ is a \ct\ and such that for each $r < s \le K$, $H_{s}$ is a single non-fixed edge $E_s$ satisfying the following properties.
\begin{enumerate}
 \item The terminal endpoint of $E_s$ is contained in $G_r$ and the initial vertex of $E_s$ has valence one in $G$. 
 \item $f(E_s) =  E_s \cdot u_s$ where $u_s$ is a completely split closed path whose endpoint is principal for $f\restrict G_r$. 
 \item $\fG$ satisfies (Linear Edges) and (NEG Nielsen Paths). 
  \end{enumerate}
 \noindent Then  there is an algorithm that takes    a    lift   $\ti f : \ti G \to \ti G$  of  $ f:G \to G$ as input and determines if $\Fix(\ti f)$ is non-empty. If it is non-empty then the output of the algorithm is an element of $\Fix(\ti f)$.  If it is empty then the output of the algorithm is a completely split  path $\ti \sigma \subset \ti G_r$ that generates a completely split ray $\ti R \subset \ti G_r$.  Moreover, if  $\Fix(\ti f) = \emptyset$   and if the projection $\sigma \subset G_r$ of $\ti\sigma$ is not a Nielsen path then   $\Fix_N(\partial\ti f) = \Fix_N(\partial (\ti f \restrict \ti G_r))= \{P\}$ where $P$ is the endpoint of $\ti R$ and  $P$ is not the endpoint of an axis of a  covering translation. 
\end{lemma}

\proof    We  dispense with the moreover statement first.  Suppose that $\Fix(\ti f) = \emptyset$   and that   $\sigma$   is not a Nielsen path.    Then $|f^k_\#(\sigma)| \to \infty$ and   \cite[Proposition~I.1]{gjll:index} implies that the terminal endpoint of   $\ti R = \ti \sigma \cdot  \ti f_\#(\ti \sigma) \cdot \ti f^2_\#(\ti \sigma)\cdot \ldots$, which is evidently fixed by $\partial \ti f$, is an attractor for the action of $\partial \ti f$,  is contained in $\Fix_N(\partial\ti f)$ and  is not the endpoint of an axis of a  covering translation.  Since $\Fix(\ti f) = \emptyset$,   \cite[Corollary 3.16]{fh:recognition} implies that $P$ is the only attractor in  $\Fix_N(\partial\ti f)$.  If there were another point in  $\Fix_N(\partial\ti f)$ then it would be the endpoint of the axis of a covering translation that commuted with $\ti f$ and the translates of $P$ would be additional attractors in  $\Fix_N(\partial\ti f)$.  Thus $\Fix_N(\partial\ti f) =  \{P\}$. 

We now turn to the algorithm.    Following the proof of \cite[Proposition~5.4.3]{bfh:tits1}, we say that for each non-fixed vertex $\ti v \in \ti G_{r}$,  the initial edge of the path from $\ti v$ to $\ti f(\ti v)$ is    {\em preferred} by   $\ti v$.     If both  $\ti E$  and $\ti E^{-1}$ are preferred by their initial vertices   then some sub-interval of $\ti E$ is mapped over itself by $\ti f$ and so contains a fixed point.
 
Consider the following (possibly infinite) method for finding either   a fixed point in $\ti G$ or  a ray whose terminal endpoint is an element of $\Fix_N(\partial\ti f)$.    Choose any vertex $\ti v_0 \in \ti G_r$.  If $\ti v_0$ is not fixed, let $\ti E_0$ be the  edge  preferred by $\ti v$.  If $\ti E_0^{-1}$ is preferred by the terminal vertex of $\ti E_0$ then $\ti E_0$ contains a fixed point that we can find by inspection (Section~\ref{proving 2.19}).   Otherwise, let    $\ti E_1$ be the edge preferred by the terminal vertex of $\ti E_0$.  Repeat this to either  find a fixed point in $\ti E_1$ or define $\ti E_2$ and so on.  If this process does not terminate by finding a fixed point then the ray $\ti R_{\ti v_0} = \ti E_1 \ti E_2 \ldots$  that it produces converges to a point in $\partial \ti F_n  $ that is evidently fixed and not repelling so is contained in $\Fix_N(\partial\ti f)$.   For each $m \ge 0$, let $\ti \sigma_m$ be the path connecting the initial endpoint of $\ti E_m$ to its $\ti f$-image.   

\vspace{.1in}   
\noindent{\bf  Step 1 of the algorithm}: Modify the above process by stopping not only if $\tilde E_m$ contains a fixed point  but also if  $\ti \sigma_m$ is completely split and the turn between $\ti \sigma_m$ and $\ti f_\#(\ti \sigma_m)$ is legal.    

\vspace{.1in}

To see that this modified process  stops in finite time,  
it suffices to show that if the original process  produces a ray $\ti R_{\ti v_0}$  then at least one of the $\ti \sigma_m$'s has the desired  properties. We verify this by   following (and tweaking) the proof  of   \cite[Proposition~5.4.3]{bfh:tits1}.   

Consider the  subsequence $\{\ti v_i\}$  of the set of vertices of $\ti R_{\ti v_0}$ starting with  $\ti v_0$ and inductively defined by letting $p  \ge i$ be the largest integer such that the closest point  to $\ti f(\ti v_i)$ in $\tilde E_0 \tilde E_1 \ldots \ti E_{p}$ is the terminal endpoint of $\tilde E_{p}$ and then taking   $\tilde v_{i+1}$ to be the terminal endpoint  of $\tilde E_{p}$.  Equivalently, ${\tilde v}_{i+1}$ is the nearest point in $\ti R_{\ti v_0}$ to $\ti f(\ti v_i)$. 

Letting $[\ti v_i, \ti v_{i+1}]$ be the path   connecting $\ti v_i$ to $\ti v_{i+1}$, the key property of  the $\ti v_i$'s is $$ \ti f_\# ([\ti v_i,\ti v_{i+1}])\supset [\ti v_{i+1},\ti v_{i+2}] $$
For $m \ge 1$, define 
$$\ti Y_m = \{\ti y \in[ \ti v_0,\ti v_1] : \ti f^i(\ti y) \in[\ti v_i,\ti v_{i+1}] \   \forall  \ 1 \le i \le m\}$$
 The obvious induction argument shows that $\ti f(\ti Y_m) =     [\ti v_m,\ti v_{m+1}]$ and in particular that $\ti Y_m$ is non-empty.  The $Y_m$'s are a nested sequence of closed non-empty subsets of $[ \ti v_0,\ti v_1]$ and so their intersection $\cap_{m=0}^\infty  \ti Y_m$  is non-empty.  Each element of $\cap_{m=0}^\infty  \ti Y_m$ is contained in 
  $$\ti X = \{\ti x: \{\ti x, \ti f(\ti x), \ti f^2(\ti x),\ldots\} \text{ is an ordered sequence of }\ti R_{\ti v_0}\}$$

In the first two paragraphs on page 68 of \cite{bfh:tits1} it is shown that $\ti X \subset \ti R_{\ti v_0}$ contains a vertex $\ti v$ that is the initial vertex of an irreducible edge.      For sufficiently large $k$, the paths $\ti \mu := [\ti f^k(\ti v), \ti f^{k+1}(\ti v)]$ and $\ti \nu := [\ti f^{k+1}(\ti v), \ti f^{k+2}(\ti v)]$  are completely split by \cite[Lemma 4.25]{fh:recognition}.  It follows, after increasing $k$ if necessary,   that the initial directions of $\mu^{-1}$ and $ \nu$ are periodic by $Df$.      Since $\ti v \in \ti X$,  these directions are distinct and so the turn they define is legal. Letting $E_m$ be the edge in $\ti R_{\ti v_0}$ that begins with  
 $\ti f^k(\ti v)$,  we have found the desired $\ti\sigma_m$.  (The  proof of  \cite[Proposition~5.4.3]{bfh:tits1} allows the possibility of subdividing at an endpoint of a periodic Nielsen path; in our context, these points are already fixed vertices so no subdivision is required.)  This completes the proof  that the first part of our algorithm stops in finite time.  

\vspace{.1in}
If the first step of the algorithm produces a fixed point we are done and the algorithm stops. Suppose then that the first step produces a path $\ti \sigma= \ti \sigma_m$ as above.  Let $P \in \Fix_N(\partial\ti f)$ be the terminal endpoint   of the ray $\ti R = \ti \sigma \cdot \ti f_\#(\ti \sigma ) \cdot \ldots$   generated by $\ti \sigma $.   The hard splitting property \cite[Lemma~4.11(2)]{fh:recognition} implies that $\ti R$ is fixed point free.  If $\Fix(\ti f) \ne \emptyset$ then there exists a ray $\ti R'$ with initial endpoint $\ti z \in \Fix(\ti f)$, terminal endpoint $P$ and with interior disjoint from  $\Fix(\ti f)$. The initial edge $\ti E$ of $\ti R'$ determines a fixed direction;    this follows from  \cite[Lemma~3.16]{fh:recognition} if $E \subset G_r$ and by hypothesis if $E \subset G \setminus G_r$. Let $R_E$ be the eigenray determined by $E$ (Definition~\ref{d:generates a ray}),  let $\ti R_{\ti E} $ be the lift of $R_E$ with initial edge $\ti E$ and let  $P' \in \Fix_N(\partial\ti f)$ be the terminal endpoint of $\ti R_{\ti E}$. If $P \ne P'$ then the line connecting $P$ to $P'$ would be a fixed point free line in $\ti G_r$ with endpoints in $\Fix_N(\ti f)$  in contradiction to \cite[Lemma~3.16]{fh:recognition}. Thus $P' = P$ and the rays  $\ti R$ and $\ti R_{\ti E} = \ti R'$ have a common terminal subray.

If $\sigma$ is a Nielsen path then $E$ is a linear edge, $f(E) = E w^d$ for some root-free Nielsen path $w$ that forms a circuit,  $R_E = Ew^\infty$ and $\ti R $ is a ray in  a line $\ti L$ that projects to $w$.  It follows that that the terminal endpoint of $\ti E$ is contained in $\ti L$.  The root-free covering translation that preserves   $L$   commutes with $\ti f$.   After translating by some iterate of this covering translation we may assume that the terminal endpoint of $\ti E$ is contained in any chosen lift $\ti w$ of $w$ in $\ti R$.    This analysis justifies the next  two steps of the algorithm.  
  
\vspace{.1in}   
\noindent{\bf  Step 2 of the algorithm:} Check if $\sigma$ is a Nielsen path.  If it is not, go to Step 3.  If it is, then the algorithm ends as follows.  Consider the finite set of points that are the initial vertex of a linear edge $\ti E$ with terminal endpoint in $\ti \sigma$.  If an element of this set is contained in $\Fix(\ti f)$ then that point is the output of the algorithm.  Otherwise conclude that $\Fix(\ti f) = \emptyset$.  

\vspace{.1in}   
We may now assume that $\sigma$ is not a Nielsen path and hence that $P$ is not fixed by any covering translation.  It follows that $\ti f$ is the only lift of $f$ that fixes $P$ and hence that $\ti f$ fixes the initial endpoint of any eigenray that converges to $P$.   

\vspace{.1in}   
\noindent{\bf  Step 3 of the algorithm} :       Apply Lemma~\ref{comparing rays} to the  set of eigenrays  (Definition~\ref{d:generates a ray}) for $f$, one by one, to decide  if there is  an eigenray $R_E$  that shares a terminal end with $R$.  If there is no such   eigenray, then   $\Fix(\ti f) = \emptyset$.    Otherwise,  we have an edge $E$, its eigenray $R_E$ and decompositions $R = \tau R''$ and $R_E = \tau' R''$ for some ray $R''$.   Let $\tau$ be the lift of $\tau$ that begins with $\ti \sigma$ and let $\ti \tau'$ be the lift of $\tau'$ that shares a terminal endpoint with $\ti \tau$.  Equivalently, $\ti R = \ti \tau \ti R''$ and $\ti R_{\ti E} = \ti \tau' \ti R''$.   Then the initial endpoint of $\ti \tau'$ is fixed by $\ti f$.      
\endproof

\section{Upward induction and extension}\label{s:extension}
The original construction of \rtt s is by downward induction through the strata, making it difficult to prove extension statements.  In this paper, we construct \ct s using upward induction.

\begin{notn}
Let $\F=\{[F^i]\}$ be a free factor system in $\f$.   A core graph $K=\sqcup_i K_i$ is {\it $\F$-marked} if each $K_i$ is marked by $F^i$, i.e.\ there is a rose $R(F^i)$ whose fundamental group is identified with $F^i$ and a homotopy equivalence $R(F^i)\to K_i$.

Suppose $\phi\in\Out(\f)$, $K$ is $\F$-marked, $\F$ is $\phi$-invariant and $h:K\to K$ is a homotopy equivalence that preserves components. We say that $h$ is a {\it topological representative of $\phi|\F$} if each induced map $h\restrict K_i:K_i\to K_i$ is a topological representative of $\phi\restrict F^i$. If each $h\restrict K_i $ is a \ct\ [and satisfying (Inheritance)] then we say that $h : K \to K$ is {\em a \ct\ representing $\phi\restrict\F$}  [and satisfying (Inheritance)].  See Section~\ref{s:intro} for the definition of (Inheritance).

A topological representative $\fG$  of  $\phi$ is an {\em extension of $h :K \to K$} if there is an embedding $K\to G$ respecting markings such that the following diagram commutes.
$$
\begindc{0}[500]
\obj(1,2)[12]{$G$}[\south]
\obj(1,3)[13]{$K$}
\obj(2,2)[22]{$G$}[\south]
\obj(2,3)[23]{$K$}[\south]
\mor{13}{12}{}[\atright,\solidarrow]
\mor{23}{22}{}
\mor{12}{22}{$f$}
\mor{13}{23}{$h$}
\enddc
$$ 
In this situation we also say that {\it $f$ extends $h$} and that {\it $h$ extends to $\f$}.
\end{notn}

Interpreted in this language, the proof of Lemma~2.6.7 of \cite{bfh:tits1} shows that every restriction extends.

\begin{lemma}\label{l:restrictions extend}
Every topological representative for $\phi\restrict \F$ extends to $\f$. \qed
\end{lemma}

\begin{remark}
The proof of Lemma~2.6.7 assumes the property that each $h\restrict K_i$ fixes a point. This assumption is not necessary. However, we will only use Lemma~\ref{l:restrictions extend} when $h|K$ is a \ct\ and so this property holds.
\end{remark}

As mentioned, we want a construction of \ct s that proceeds by upward induction. Our main tool will be a relative version of Theorem~\ref{algorithmic ct}.

\begin{thm}[{ \bf Extension}]  \label{thm:extension}
Suppose that $\phi \in \Out(F_n)$ is rotationless, that ${\cal C}$ is a nested sequence $(\F=\F_0)\sqsubset \F_1 \sqsubset \F_2 \sqsubset \ldots \sqsubset (\F_m=\{[F_n]\})$ of $\phi$-invariant free factor systems and that  $h : K \to K$ is a \ct\ representing $\phi \restrict \F$. Then  there is an algorithm that produces a \ct\ $\fG$ and filtration  that  represents $\phi$,  extends $\hK$ and  such that each element of $\cal C$ is realized by a core filtration element;  if   $\hK$   satisfies (Inheritance) then   $\fG$ satisfies (Inheritance).
\end{thm}
  
Theorem~\ref{algorithmic ct} is the special case of Theorem~\ref{thm:extension} in which $\F = \emptyset$.  
  
The filtration  we start with might not be reduced so our algorithm will have to discover reductions, if they exist, as it proceeds. 
   
\begin{thm}[{\bf Extend or Reduce}]  \label{extend or reduce}
Suppose that $\phi \in \Out(F_n)$ is rotationless,   that $ \F$ is a   $\phi$-invariant free factor system and that $h : K \to K$ is a \ct\ that represents $\phi \restrict \F$  [and satisfies (Inheritance)].  Then there is an algorithm that either produces a \ct\ $\fG$ that represents $\phi$ [satisfies (Inheritance)]   and extends   $h : K \to K$ or finds a $\phi$-invariant proper free factor system $\F'$ that properly contains $\F$.
\end{thm}
 
\medspace
 
\noindent{\em Proof of Theorem~\ref{thm:extension}  given Theorem~\ref{extend or reduce}:}  
The {\it proper length of ${\cal C}$} is the number of inclusions that are proper. If ${\cal C}'$ is a sequence of inclusions of $\phi$-invariant free factor systems and if each element of ${\cal C}$ is an element of ${\cal C'}$ then we say that {\it ${\cal C'}$ is an extension of ${\cal C}$}. Define $L({\cal C})\ge 0$ to be the maximal proper length of some $\cal C'$ extending $\cal C$. The proof is an induction on $L({\cal C})$. If $L({\cal C})=0$ then $\F = \{[F_n]\}$ and the statement is vacuous.  If $L({\cal C})=1$ the statement follows from Theorem~\ref{extend or reduce}. Assume then that $L({\cal C})\ge 2$. We may assume that the inclusions $\F\sqsubset\F_{m-1}\sqsubset \{[\f]\}$ are proper.

{\bf Step 1} (Extend $h$ to $\F_{m-1})${\bf :} Each component $[F]$ of $\F_{m-1}$ induces the $\phi\restrict F$-invariant nested sequence ${\cal C}| F$ of free factor systems in $F$ given by 
$$(\F|F=\F_0|F)\sqsubset \F_1|F \sqsubset \F_2|F \sqsubset \ldots \sqsubset (\F_{m-1}|F=\{[F]\})$$ where $\F_i | F$  is the union of the components of $\F_i$ that conjugate into $F$. Clearly $L({\cal C}| F)<L({\cal C})$. Let $K(F)$ be the union of the components $C$ of $K$ such that $[C]$ is conjugate into $F$ and so $[K(F)] = \F | F$. By induction we can algorithmically produce a \ct\  and filtration that represents $\phi\restrict F$,  extends the restriction of $h$ to $K(F)$ and  such that each element of ${\cal C} | F$ is realized by a core filtration element;  if   $\hK$   satisfies (Inheritance) then so does this \ct. The disjoint union of these \ct s is a \ct\ [satisfying (Inheritance) if $h$ does] representing $\phi\restrict\F_{m-1}$ such that each element of $\F\sqsubset \F_1\sqsubset \dots \sqsubset \F_{m-1}$ is represented by a core filtration element.

{\bf Step 2} (Further extend to $\f$){\bf :} The sequence $\F_{m-1}\sqsubset \{[\f] \}$ also has proper length less than $L({\cal C})$. Hence  by induction we can algorithmically produce a \ct\  that represents $\phi$ and extends the restriction of $\phi$ to $\F_{m-1}$ found in Step~1;  if   $\hK$   satisfies (Inheritance) then so does this \ct. \qed

\medspace

To prove Theorem~\ref{algorithmic ct}, it remains to prove Theorem~\ref{extend or reduce}. The proof of Theorem~\ref{extend or reduce} is given in Section~\ref{s:extend or reduce proof}.

\section{Proof of Theorem~\ref{extend or reduce}} \label{s:extend or reduce proof}
The proof of Theorem~\ref{extend or reduce} is carried out in Sections~\ref{section one edge} and  \ref{s:multi-edge proof}. In both cases, $\phi, \F$ and $h:K \to K$ are as in the statement of the theorem.

\subsection {The one-edge extension case}\label{sec:one edge} \label{section one edge}
In this section we assume that $\F \sqsubset \{[F_n]\}$ is a one-edge extension, meaning that  either $\F = \{[F_1]\}$ where $F_1$ has rank $n-1$ or $\F = \{[F_1], [F_2]\}$ where the ranks of $F_1$ and $F_2$ add to $n$. In this case,  
 there is a marked graph $G$ obtained from $K$ by adding a single topological edge $E$ and there is a topological representative $g : G \to G$ of $\phi$  that agrees with $h$ on $K$ and satisfies $g(E) = \bar uEv$   or $g(E) = \bar u\bar Ev$    for some possibly trivial paths $u,v \subset K$ by \cite[Corollary 3.2.2]{bfh:tits1}.   We first complete the proof assuming that $g(E) = \bar uEv$ and then apply \ct\ theory to show that the $g(E) = \bar u\bar Ev$ case does not happen.

If $u$ and $v$ are both trivial then $g(E) = E$  and $g$ satisfies all the properties of a \ct\ except  that $K \subset G$ might not be reduced.   Check (using   Lemma~\ref{finding iNps} and (\noneg\ Nielsen Paths))     if there is a Nielsen path $\beta \subset K$ connecting the terminal endpoint of $E$ to the initial endpoint of $E$.    If not, then $K \subset G$  is reduced by Lemma~\ref{reducible noneg} and $g$ is a \ct.  If yes, then $E\beta$ is a closed Nielsen path and  the \ffs\ $\{[K],[E\beta]\}$ is properly contained between $[K]$ and $[G]$. 

Suppose  next  one of $u$ and $v$ is trivial and that the other is not. The cases are symmetric so we assume that $u$ is trivial and $v$ is non-trivial.    Let $C$ be the component of $K$ that contains $v$.   If  $C$ has rank one then $C$  has  a single fixed edge $e$ and single vertex $w$  and    $g(E) = E e^d$ for some $d \ne 0$.   There are two subcases.   If $w$ is  also the initial vertex  of $E$  then $n=2$ and  $G$ has only two edges,  $E$ and $e$. (\noneg\ Nielsen Paths)   follows from the basic splitting property for \noneg\ edges \cite[Lemma 4.1.4]{bfh:tits1}.    Lemma~\ref{reducible noneg} implies that $[K] \sqsubset F_n$ is reduced and so (Filtration) is satisfied.  The remaining \ct\ properties are clear so $g$ is a \ct.   If $w$ is not the initial vertex of $e$ then $g$ is not a \ct\ because $w$ is not principal.  In this case we redefine $g$ so that it fixes $E$ and is unchanged on $K$.  The resulting homotopy equivalence $\fG$ is homotopic to $g$ and so is still a topological representative of $\phi$.   We are now back in the case that $u$ and $v$ are trivial.  Lemma~\ref{reducible noneg} implies that $K \subset G$ is reduced so $\fG$ is a \ct.

 We assume now that $C$ has rank at least two and we apply   \cite[Step 5 pages 91--93]{fh:recognition}.    Choose a lift $\ti E$ of $E$, let $\ti g : \ti G \to \ti G$ be the lift of $g$ that fixes the initial endpoint of $\ti E$,    let $\Gamma \subset \ti G$ be the component of the full pre-image of $C$ that contains the terminal endpoint of $\ti E$ and let $\ti h =\ti g \restrict \Gamma : \Gamma \to \Gamma$.   
To make this step algorithmic, we apply Lemma~\ref{finding a fixed point} to either find a fixed point for $\ti h$ or to conclude that $\ti h$ is fixed point free and find a  completely split path $\ti \sigma$ that generates a completely split ray   $\ti R$.  No algorithm for this was given in the original proof.

The remainder of \cite[Step 5 pages 91--93]{fh:recognition}  can be applied as written in \cite{fh:recognition}. This is a sliding operation that changes the way that $E$ is attached to $C$.  It may be that $E$ becomes a fixed edge for the new modified $g:G \to G$.    In this case, we go back to the fixed edge case described above and proceed from there.  If $E$ is not fixed after the sliding operation  then all of the \ct\ properties are satisfied except   that $K \subset G$ might not be reduced. The proof now concludes by applying  Lemma~\ref{reducible noneg} and \cite[Step 5 page 91]{fh:recognition} to conclude that $K \subset G$ is reduced.

The final case is that both $u$ and $v$ are non-trivial.  Subdivide $E = \bar E_1 E_2$ where $g(E_1) = E_1 u$ and $g(E_2) = E_2 v$. Let $C_i$ be the component of $K$ that contains the terminal endpoint of $E_i$.  If  $C_1 =C_2$ has rank one, then $C_1=C_2$  has  a single fixed edge $e$ and single vertex $w$  and    $g(E_i) = E_i e^{d_i}$ for some $d_1,d_2 \ne 0$.   Define $f$ by $f(e) = e$ and $f(E) =  Ee^{d_2-d_1}$ and note that $f$ is homotopic to $g$ and so is still a topological representative of $\phi$.  If $d_1 \ne d_2$ then $f$ is a \ct.  If $d_1 = d_2$ then   $\{[e],[E]\}$ is a $\phi$-invariant free factor system that is properly contained between $[K]$ and $[F_n = F_2]$. 

If $C_1 \ne C_2$ both have rank one then the identity map of $G$ is homotopic to $g$ and is a \ct.   We may now assume, after interchanging $E_1$ and $E_2$ if necessary that $C_2$ has rank at least two.  If $C_1$ has rank one then $g$ is homotopic to $g'$  that agrees with $h$ on $K$ and that satisfies  $g'(E) = Ev$ so we are reduced to a previous case.

The remaining case is  that both $C_1$ and $C_2$ have rank at least two.    The edges $E_1$ and $E_2$ are modified in the same way that $E$ was modified in the case that $C$ had rank at least two. Namely,  follow \cite[Step 5 pages 91--93]{fh:recognition} and apply  Lemma~\ref{finding a fixed point} to make the construction algorithmic.  The edges $E_1$ and $E_2$ are considered one at a time so when $E_2$ is considered the subgraph it is being attached to is $K \cup E_1$ and not $K$ (This accounts for the hypotheses of Lemma~\ref{finding a fixed point} being what they are.) It may be that after sliding $E_2$, both of its endpoints are attached to the initial vertex of $E_1$; see Example~\ref{finding a loop}. 

This completes the proof in the case that $g(E) = \bar u Ev$.  Suppose now that $g(E) = \bar u\bar Ev$.  Applying the previous case to $h = g^2$ representing $\theta = \phi^2$, we see that there is a \ct\ representing $\theta$ and extending $h^2 :K \to K$.   As shown in the first two paragraphs of \cite[Section 5.1]{fh:recognition}, there is a line $\gamma \subset  G$ whose edge path contains one copy of $E$ and zero copies of $\bar E$  and such that some (and hence every) lift $\ti \gamma$ has endpoints in $\Fix_N(\Theta)$ for some $\Theta \in \PA(\theta)$.  Since $\phi$ is rotationless, there exists $\Phi \in \PA(\phi)$ such that  $\Fix_N(\Phi) = \Fix_N(\Theta)$.  Thus $\gamma$ is $\phi$-invariant which contradicts the fact that the edge path for $h(\gamma)$  crosses $\bar E$ once and $E$ zero times.  We conclude that this case does not happen. 
  \endproof

\subsection{The multi-edge extension case} \label{s:multi-edge proof}
We assume in this section  that $\F \sqsubset \{F_n\}$ is not a one-edge   extension.  If $\fG$ represents $\phi$ and extends $h :K \to K$ then the core filtration element that is identified with $K$ will be denoted $G_r$. In particular, $[G_r] = \F$.

\medspace

\noindent{\bf Step 1:} \ \  There is an algorithm that  either produces a \rtt\    extending $h : K \to K$ or finds a proper $\phi$-invariant free factor system  $\F'$   that properly contains $\F$.  Moreover, in the former case, either there are no \eg\ strata above the core filtration element $G_r$ that is identified with $K$ or  the top stratum $H_N$ of $G$ is the only \eg\ stratum above  $G_r$.

\medspace

We can not simply apply the \rtt\ algorithm of Theorem~\ref{thm:rtt}.  That algorithm was by downward induction through the filtration and so made no effort to leave lower filtration elements untouched. We say that an algorithm is {\it $K$-safe} if  it preserves the property of extending $h : K \to K$.  

\begin{lemma}\label{two safe moves} The algorithms of Lemma~\ref{core subdivision} and Lemma~\ref{subdivide and fold} are $K$-safe.
\end{lemma}

\proof    The algorithms use only subdivision, folding, valence one homotopies and valence two homotopies applied to edges above $K$ and so are $K$-safe.
\endproof

The algorithm of  Lemma~\ref{collapsing} (which is really \cite[Lemma 5.14]{bh:tracks})  is not entirely $K$-safe.  In order to isolate the part that is $K$-safe we introduce some notation and recall the steps in the algorithm.

If $H_s$ is an \eg\ stratum of a topological representative then a  non-trivial path $\sigma \subset G_{s-1}$    with endpoints in $H_s \cap G_{s-1}$ is called an {\em inessential connecting path for $H_s$} if  $f_\#(\sigma)$ is trivial.  If the \eg\ stratum $H_s$ satisfies   (RTT-i) then it  satisfies   (RTT-ii) if and only if there are no inessential connecting paths for  $H_s$.

An inessential connecting path $\sigma$ for $H_s$  is \lq collapsed\rq\  as follows.  Choose a turn in $\sigma$ whose  two directions have the same image under $Df$. This exists because $f_\#(\sigma)$ is trivial.   Then fold initial segments of these edges and tighten the new map if necessary to  \lq shorten\rq\   $\sigma$.  After finitely many such moves $\sigma$ is completely folded away and the endpoints of $\sigma$ are identified to a single vertex thereby reducing  the cardinality of $H_s \cap G_{s-1}$.

If $\sigma \subset G_{s-1}$ is disjoint from  $G_r$, for example if $\sigma$ is contained in a contractible component of $G_{s-1}$, then collapsing  $\sigma$ is $K$-safe.    We record this fragment of Lemma~\ref{collapsing}  as Lemma~\ref{partial collapsing} after adding one more piece of notation.
     
 \begin{notn} An \eg\ stratum $H_s$ satisfies {\em (partial RTT-ii)} if the contractible components of $G_{s-1}$ do not contain any  inessential connecting paths.
 \end{notn}  

\begin{lemma} \label{partial collapsing}  Suppose that $\fG$ is a bounded topological representative   of $\phi$  that extends $h : K \to K$ and that $H_s$ is an \eg\ stratum that does not satisfy (partial RTT-ii).  Then there is a $K$-safe algorithm to construct a bounded topological representative $f' :G' \to G'$  of $\phi$ such that 
\begin{enumerate}
\item $\Lambda(f) = \Lambda(f')$
\item  There is a   bijection $H_j \leftrightarrow H'_{j'}$ between the \eg\ strata of $f$ and the \eg\ strata of $f'$ such that:
\begin{enumerate}
\item  relative height is preserved; i.e.\ $j < k $ if and only if $j' < k'$.
\item   $|H'_{s'} \cap G'_{s'-1}|  < |H_s \cap G_{s-1}|$
\item  if  $H_s$ satisfies (RTT-i)  then $H'_{s'}$ satisfies (RTT-i). \qed
\end{enumerate}  
\end{enumerate}
\end{lemma}

\begin{remark}  
Note that Lemma~\ref{partial collapsing} is not exactly parallel to Lemma~\ref{collapsing}.   We will only apply Lemma~\ref{partial collapsing} with $H_s$ being the first \eg\ stratum above $K$ and we will not be concerned with preserving properties of higher \eg\ strata.
\end{remark}

The following corollary takes us as far as we can go using just the techniques of Theorem~\ref{thm:rtt}. 

\begin{corollary} \label{cor:partial} Given a bounded topological representative $\fG$  of $\phi$   that  extends $h : K \to K$, there is an algorithm that constructs a bounded topological representative  $f' : G' \to G'$  of $\phi$  that extends $h : K \to K$ such that $\Lambda(f') \le \Lambda(f)$ and such that either there are no \eg\ strata above $K$ or the first  \eg\ stratum above $K$ satisfies (RTT-i), (RTT-iii) and  (partial RTT-ii).
\end{corollary}

\proof If there are no \eg\ strata above $K$ then we are done.   Otherwise, apply Lemma~\ref{core subdivision} and Lemma~\ref{two safe moves}  to produce a    bounded topological representative (still called $\fG$)  of $\phi$ that extends $h : K \to K$ whose first \eg\ stratum above $K$ satisfies (RTT-i).   If that  stratum does not also satisfy (partial RTT-ii), apply Lemma~\ref{partial collapsing}   to produce a  new  bounded topological representative (still called $\fG$) of $\phi$ whose first  \eg\ stratum above $K$ still satisfies (RTT-i).    Item (b) of  Lemma~\ref{partial collapsing} guarantees that after finitely many applications of  Lemma~\ref{partial collapsing}, we arrive at $\fG$ whose first \eg\ stratum above $K$ satisfies    (RTT-i) and (partial RTT-ii).    Apply Lemma~\ref{checking rtt} to check if the  first \eg\ stratum above $K$ satisfies    (RTT-iii).  If yes, we are done.  Otherwise  apply Lemma~\ref{subdivide and fold}  and Lemma~\ref{two safe moves} to produce a bounded topological representative  $f':G' \to G' $ of $\phi$ that extends $h : K \to K$ with $\Lambda(f') < \Lambda(f)$.  Then start over again with $f':G' \to G'$ replacing the original $\fG$.   Since  every decreasing sequence $\Lambda(f) > \Lambda(f') > \ldots$ is finite (Definition~\ref{def:bounded}), this process produces the desired $\fG$ in finite time.
\endproof 

Before introducing the new move necessary to achieve (RTT-ii) in a $K$-safe manner, we prove a lemma that will simplify the situation in which the new move is needed.  

\begin{lemma} \label{proper extension}Suppose that $\fG$ is a   bounded topological representative    of $\phi$ that extends $K$ and that  $H_s$, $s > r$, is an \eg\ stratum that   satisfies (RTT-i),  (RTT-iii) and   (partial RTT-ii).  Then   $[G_r]= \F$ is properly contained in $[G_s]$.   In particular, if $G_s \ne G$ then $[G_s]$ is  a proper $\phi$-invariant free factor system that properly contains $\F$.
\end{lemma}

\proof  It suffices to show that there is a line in $G_s$ that crosses an edge in $H_s$ and for this it suffices to show that if a vertex $v \in H_s$ is either disjoint from $G_{s-1}$ or is  contained in a contractible component of $G_{s-1}$  then there is path $\mu =  \bar E_1 \tau E_2  \subset G_s$ where $E_1,E_2$ are edges in $H_s$ and $\tau$ is a possibly trivial path that contains $v$ and does not contain an edge in $H_s$. Since $H_s$ is an \eg\ stratum, the cardinality of $f^{-p}(v) \cap H_s$ goes to infinity with $p$. Choose $p \ge 1$ and  a point $x$ in the interior of an edge $E \subset H_s$   such that $f^p(x) = v$.   We may assume without loss   that either $f(x) \in G_{s-1}$ or that $f(x)$ is a vertex in $H_s$.  Subdivide $E$ into \lq edgelets\rq\ that are mapped by $f$   to single edges in $G$.   There is an edgelet subpath $e_1\ldots e_t$ such that $f(e_j)$  is an edge in $H_s$ if and only if $j=1$ or $j =t$ and such that either $t > 2$ and $x \in e_2\ldots e_{t-1}$ or $t = 2$ and $x$ is the common endpoint of $e_1$ and $e_2$.  If $t = 2$ then $(f(\bar e_1),f(e_2))$ is a legal turn in $H_s$ by (RTT-iii) and so    $Df^{p-1}(f(\bar e_1))$ and $Df^{p-1}(f(e_2))$   are distinct edges in $H_s$ based at $v$ and we are done (with $\tau$ being trivial).   If $t > 2$ then $\tau' = f(e_2)\ldots f(e_{t-1})$ is a non-trivial path in $G_{s-1}$ with endpoints in $H_s\cap G_{s-1}$.  In particular,  $f(x) \in G_{s-1}$.  It follows that  $v = f^{p}(x) \in G_{s-1}$ and hence (by hypothesis)    $v$ is contained in  a contractible component of $G_{s-1}$.  This in turn implies that each  of $\tau', f_\#(\tau'),\ldots, f^{p-1}_\#(\tau')$  is   contained in a contractible component of $G_{s-1}$. Since $\tau'$ is a non-trivial path in $G_{s-1}$ with endpoints in $H_s\cap G_{s-1}$, the same is true for $f_\#(\tau'),\ldots, f^{p-1}_\#(\tau')$   by (RTT-i) and the (partial RTT-ii) property.   Property (RTT-i)  implies that  $Df^{p-1}(f(\bar e_1))$ and $Df^{p-1}(f(e_t))$ are  directions in $H_s$ and again we are done.  This completes the proof of the lemma.
\endproof  

\begin{notn} Let $\cR$  be the set of bounded topological representatives $\fG$    of $\phi$  that extend $h : K \to K$ and such that the top stratum $H_N$ :
\begin{itemize}
\item is the only \eg\ stratum above $K$
\item   satisfies (RTT-i), (RTT-iii) and (partial RTT-ii).
\end{itemize}
The failure of $\fG$ in $\cR$ to be a \rtt\ is measured by the number $\delta(f)$ of directions in $H_N$ that are based at non-periodic vertices in non-contractible components of $G_{N-1}$.
\end{notn}

We will only define the new move in this context.
 
\begin{lemma} \label{K safe collapsing}  Suppose that $\fG$ is an element of  $ \cR$ and that $\delta(f) > 0$.  Then there is an algorithm to construct an element   $\hat f :\hat G \to \hat G$  of $\cR$ such that either:
\begin{enumerate}
\item \label{reduce Lambda}$\Lambda(\hat f) < \Lambda(f)$\ \  or
 \item  \label{reduce delta}  $\Lambda(\hat f) = \Lambda(f)$ and $\delta(\hat f) < \delta(f)$
\end{enumerate}
\end{lemma}

\proof   Let $A_p$ [resp. $A_{np}$]  be the set of periodic [resp. non-periodic] vertices of  $H_N$ that are contained in non-contractible components of $G_{N-1}$.  Thus   $f$ permutes the elements of $A_p$ and each element of $A_{np}$ is mapped by an iterate of $f$ into $A_p$.   By hypothesis, $A_{np}\ne \emptyset$.  Choose an element   $x \in A_{np}$ such that $f(x) \in A_p$.        There is a unique $y \in A_p$ such that $f(x) = f(y)$   and there is a unique inessential connecting path $\nu$ connecting $x$ to   $y$.    Choose an edge $E \subset H_N$ with initial endpoint  $x$ and slide its initial endpoint along $\nu$ to $y$ to produce a new   topological representative $f' :G' \to G'$.      The marked graph $G'$ is  obtained from $G$ by replacing $E$ with an edge $E'$ with  terminal endpoint equal to that of $E$ and with  initial endpoint $y$.  Thus $G \setminus E$ can be viewed as a subgraph of both $G$ and $G'$.  For each edge $e \subset G \setminus E$, the edge path $f'(e)$ is obtained from the edge path $f(e)$ by replacing each copy of $E$ with $\nu E'$ and each copy of $\bar E$ with $\bar E' \bar \nu $ and then tightening.  The edge path $f'(E')$ is obtained from the edge path  $f_\#( \bar \nu E) =f_\#(\bar \nu) f(E) = f(E)$ by making the same replacements as in the previous case and then tightening.   It is clear that $f \restrict G_{N-1} = f'\restrict G_{N-1}$ and that no edges in $H_N$ are cancelled during the tightening operation.  It follows that $f'$ extends $h : K \to K$, that the top stratum $H'_{N}$ of $G'$ is the only \eg\ stratum above $K$,   that $\Lambda(f) = \Lambda(f')$ and that  $f'$ satisfies (partial RTT-ii). By construction, $\delta(f') = \delta(f) -1$. If the initial direction determined by $E$  is not in the image of $Df$ then $H'_{N}$ satisfies (RTT-i).  If $H'_{N}$ also satisfies (RTT-iii) then $\hat f = f'$ satisfies  \pref{reduce delta}.  Otherwise, we apply Lemma~\ref{subdivide and fold} and Lemma~\ref{two safe moves} to produce $\hat f$ satisfying \pref{reduce Lambda}.

If the initial direction determined by $E$  is   in the image of $Df$ then $H'_{N}$ does not satisfy  (RTT-i).  (For example, if   $Df(e) = E$ then $f'(e)$ begins with  $\nu E'$.) In this case, we perform  a core subdivision producing a new map   $f'' :G'' \to G''$.  (Continuing with the  example, $e = e_1e_2$ where  $f''(e_1) =   \nu$ and $f''(e_2)$ begins with $E'$.) There is one subdivision point for each direction in $H_N$ that is mapped by some iterate of $Df$ to $E$.   If an edge $e$ is subdivided into $e_1e_2$ then $e_1$ is a  zero stratum for $f''$ and $e_2$ replaces $e$ as an edge in the top \eg\ stratum.  (It may be that the initial and terminal directions of $e$ are both eventually mapped to $E$ and so $e$ is ultimately subdivided into two zero strata and one edge in the top \eg\ stratum.)        The contribution of $e_2$ to $\delta(f'')$ balances the contribution of $e$ to $\delta(f')$ so $\delta(f'') = \delta(f') < \delta(f)$.   If (RTT-iii) is satisfied  then $\hat f = f''$ satisfies  \pref{reduce delta}.  Otherwise, we apply Lemma~\ref{subdivide and fold} and Lemma~\ref{two safe moves} to produce $\hat f$ satisfying \pref{reduce Lambda}.
   \endproof
   
   Step 1 can now be completed as follows.  Start with any bounded topological representation of $\phi$ that extends $h : K \to K$ (Lemma~\ref{l:restrictions extend}).      Apply Corollary~\ref{cor:partial} to produce a bounded topological representative  $f' : G' \to G'$  of $\phi$  that extends $h : K \to K$ such that $\Lambda(f') \le \Lambda(f)$ and such that either there are no \eg\ strata above $K$ or the first  \eg\ stratum $H'_{s'}$ above $K$ satisfies (RTT-i), (RTT-iii) and  (partial RTT-ii).  In the former case, $f'$ is  a \rtt\ and we are done so assume the latter holds.  If $H'_{s'}$ is   not the top stratum then (Lemma~\ref{proper extension}) we have found  a proper $\phi$-invariant free factor system that properly contains $\F$ and we are done.   We may therefore assume that $f':G' \to G'$ is an element of $\cR$. If $\delta(f') = 0$ then  $f':G' \to G'$ is a \rtt\ and we are done.  Otherwise, apply Lemma~\ref{K safe collapsing} to produce $f'' : G'' \to G'$ with either $\Lambda(f'') < \Lambda(f') \le \Lambda(f)$ or $\delta(f'') < \delta(f')$.  In the former case, we start all over. This can only happen a finite number of times so we may assume that we are in the case that $\delta(f'') < \delta(f')$.  After applying    Lemma~\ref{K safe collapsing} finitely many times, we arrive a \rtt\  and are done.   

\medspace

\noindent{\bf Step 2:} \ \   Modify   $\fG$ from Step 1 so that  the conclusions of Theorem~\ref{2.19} are satisfied. 

\medspace

The algorithm for Step 2 is explicitly described in Subsection~\ref{proving 2.19} and in \cite[pages 18--23]{fh:recognition}.  Each move effects only strata above $G_r$ and so is $K$-safe.  For future reference we note that the   number of edges in each  \eg\ stratum and the number of \iNp s of \eg\ height are not increased in Step 2.  This ends Step~2.

\medspace

Going forward, we may now also assume (with justification below) that:
\begin{description}
\item  [(i)]  $\fG$ is rotationless.
\item [(ii)] $\F \subset \{[F_n]\}$ is irreducible.
\item [(iii)] there are no periodic strata above $G_r$.
\item [(iv)] every \ipNp\ $\rho$ of \eg\ height has period one.
\item [(v)] $H_N$ is \eg\ and aperiodic, meaning that some iterate of its transition matrix is positive.
\end{description}

Let $\fG$ be as in Step 2.  Item (i) follows from \cite[Proposition 3.29]{fh:recognition}.   Applying the algorithm of    Proposition~\ref{find reduction},  we can either find a reduction of $\F \subset \{[F_n]\}$, in which case we are done, or conclude that $\F \subset \{[F_n]\}$ is irreducible.   
We may therefore assume that  (ii) is satisfied.    Item (iii) follows from (ii) and property (P) of Theorem~\ref{2.19}. 

  If there are no \eg\ strata above $G_r$ then each stratum $H_j$ above $G_r$ is non-periodic \noneg\ by  (iii) and property (Z) of Theorem~\ref{2.19}.  Item (i) and property (NEG) of Theorem~\ref{2.19} then imply that each $H_j$ is a single edge $E_j$ with terminal endpoint in a core filtration element of height less than $j$.   By  property (F) of Theorem~\ref{2.19}, the core of each filtration element is a filtration element.   Letting $G_t$ be the first core filtration element above $G_r$, it follows that  $[G_r] \sqsubset [G_t]$ is a one-edge extension.  By hypothesis, $[G_r] \sqsubset [G_N]$ is not a one-edge extension. We have therefore found a reduction in contradiction to (ii).  We conclude that $H_N$ is \eg\ and is the only irreducible stratum above $G_r$ (recall the output of Step~1).  
 
 We   need only check (iv)  for an    \ipNp\ $\rho$ of  height $N$.  Any such $\rho$  begins and ends with edges in $H_N$ by \cite[Lemma 5.11]{bh:tracks}.  It follows that the endpoints of $\rho$ are incident to  at least one periodic direction in $H_N$ and so are principal by   Definition~\ref{d:principal vertex}.    Item   (iv) therefore follows from (i) and \cite[Proposition 3.29]{fh:recognition}.     Item (i) and  \cite[Lemma 3.19]{fh:recognition} imply  that  there are   fixed directions in $H_N$ which in turn implies that $H_N$ is aperiodic so (v) is satisfied.

\medspace

\noindent{\bf Step 3:} \ \  Modify   $\fG$ from Step 2 to produce a \rtt\ that satisfies   (EG Nielsen Paths).

\medspace 

  An algorithm for Step 3 is given   in \cite[pages 87--89]{fh:recognition} but it is not entirely $K$-safe because it makes use of  the \lq collapsing inessential connecting paths\rq\ move in the \rtt\ algorithm. If we carry out this collapse as described in Step~1 above rather than as described in \cite{bh:tracks} then the algorithm becomes $K$-safe and we can use it  to complete Step 3.  
  For the readers convenience, we summarize this algorithm and point out the one place where the $K$-safe modification occurs.   
  
\begin{remark}  The algorithm of \cite{fh:recognition} applies results from \cite[Sections 3 and 5]{bh:tracks} and \cite[Sections 5.2 and 5.3]{bfh:tits1}.    Section~5.3 of \cite{bfh:tits1}   has the global hypothesis that $H_N$ is a geometric stratum.  That hypothesis is not used in any of the results  cited so our conclusions also hold in the non-geometric case.   Most of the results cited ultimately derive from Section~5 of \cite{bh:tracks} where there is no assumption that $H_N$ is geometric.
\end{remark} 

By \cite[Lemma 5.11]{bh:tracks}, every \iNp\ of height $N$ has the form $\rho= \bar \alpha \beta$ where $(\alpha,\beta)$ is the only illegal turn of height $N$ in $\rho$. Let $E_1$ and $E_2$ be the first edges of $\alpha$ and $\beta$ respectively.   Depending on the edge paths $f(E_1)$ and $f(E_2)$, there are three ways in which $\fG$ and $\rho$ can be modified to produce a new \rtt\ $f':G' \to G'$ and \iNp\ $\rho' \subset G'$.  In all three cases, the number of edges in   \eg\ strata   and the number of \iNp s of \eg\ height   do not increase.  The first and third are $K$-safe as described in   \cite{bh:tracks}  and \cite{bfh:tits1}. To make the second $K$-safe we  use Lemma~\ref{K safe collapsing} instead of Lemma~\ref{collapsing}.

If one of $f(E_1)$ and $f(E_2)$, say $f(E_2)$,  is properly contained in the other then the fold is said to be {\em proper}.  There is a maximal path $\delta \subset G_{N-1}$ such that $E_2 \delta$ is an initial segment of $\beta$.   In this case,   $f' : G' \to G'$ is defined by  folding an initial segment of $E_1$ with $E_2 \delta$; see  \cite[Definition 5.3.2 and Lemma 5.3.3]{bfh:tits1}. The \iNp\ $\rho'$ is the tightened  image of $\rho$ in $G'$ under the folding map.

In the {\em improper} case,   $f(E_1) = f(E_2)$ and one begins the process  \cite[Definition 5.3.4]{bfh:tits1} by folding $E_1$ and $E_2$ to form a single  new edge.  If the edge following $E_1$ in $\alpha$ or the edge following $E_2$ in $\beta$ belongs to $H_N$ then both of   those edges belong to $H_N$ and  nothing more is required.  Otherwise,   the resulting map is not a \rtt\ and one must perform core subdivisions and collapses of inessential connecting paths to restore the \rtt\ properties.  These should be done as in Step 1 so as to be $K$-safe. In this case, the the number of edges in the \eg\ stratum   decreases.   See \cite[Lemma 5.3.5]{bfh:tits1}.

The third possibility, a {\em partial fold},  is that the maximal common subpath of $f(E_1)$ and $f(E_2)$ is proper in both edge paths.  It follows \cite[page 25]{bh:tracks} that  $\alpha = E_1$ and $ \beta = E_2$. Items (ii) and (v) and \cite[Lemma 5.1.7]{bfh:tits1} imply that the endpoints of $\rho$ are distinct and that if both endpoints are contained in $G_{N-1}$ then at least one of them is  contained  in a contractible component of $G_{N-1}$ 
\cite[Lemma 5.1.7]{bfh:tits1}.  Item (iii) above implies that there are no invariant contractible components of $G_{N-1}$ and we conclude that at least one of the endpoints of $\rho$ is disjoint from  $G_{N-1}$.   In this case, do not perform a standard fold but rather entirely identify $E_1$ and $E_2$ to form a new graph $G'$ with an induced map $f':G'\to G'$.  Since   the turn $(\alpha, \beta)$ is not taken by $f(E)$ for any edge $E$, the image in  $G'$ of $f(E)$ is already tight and $f' :G' \to G'$ is a topological representative.  It is immediate from the construction that  (RTT-i) is preserved.  Since at least one of the endpoints of $\rho$ is disjoint from $G_{N-1}$,  (RTT-ii) is also preserved.  As argued in the proof of \cite[Lemma 5.2.4]{bfh:tits1}, $\Lambda' = \Lambda$ so Lemma~\ref{subdivide and fold} and Lemma~\ref{two safe moves} imply that (RTT-iii) is satisfied and  $f' :G' \to G'$ is a relative train track map.  By construction (see also the proof of \cite[Lemma~5.2.4]{bfh:tits1}) the number of \iNp s with \eg\ height has been decreased.

Since each of these operations begins and ends with a relative train track map and an \iNp, we can iteratively fold to produce a sequence of relative train track maps and \iNp s.    Let $C(f)$ be the sum of the entries in the transition matrix for $f$.  This gives an upper bound for the number of elementary folds in a  Stallings factorization of $\fG$.  If the first $C(f)$ folds encountered are proper then $f$ satisfies (EG Nielsen Paths); see  \cite[Lemma 4.33 and Remark 4.34]{fh:recognition} and the proof of \cite[Lemma 5.3.6]{bfh:tits1}.    Otherwise, one encounters either a partial or an improper fold   in which case either the number of edges in   \eg\ strata or the number of \iNp s  of \eg\ height  decreases and we start over.  Since these numbers never increase, this algorithm terminates in finite time.

\medspace

 At the end of Step 3,   $\fG$  satisfies (EG Nielsen Paths).   If the \rtt\ that was the input to Step 3 did not satisfy (EG Nielsen Paths) then either the number of edges in an \eg\ stratum decreased or the number of \iNp s of \eg\ height  decreased during Step 3.  Throughout the remaining steps, the number of edges in each  \eg\ stratum and the number of \iNp s of \eg\ height are never increased.  As a result, we need not verify the (EG Nielsen Paths) after each step.  If it fails at some point, go back to Step 3 and start again.  This can only happen finitely many times so does not prevent the process from terminating.   
 
\medspace

\noindent{\bf Step 4:} \ \   If the \rtt\ produced by Step 3 does not satisfy  the conclusions of Theorem~\ref{2.19} return to Step 2.  As just noted, changes are required in Step 3 only a finite number of times so Step 4 is a finite process.   

\medspace

\noindent{\bf Step 5:} \ \  (Rotationless), (Filtration), (Zero Strata) and (Periodic Edges).  

\medspace  The first two properties follow from \cite[Proposition 3.29]{fh:recognition} and Theorem~\ref{2.19} respectively.  (Zero Strata) is arranged by the $K$-safe  tree replacement moves  described in \cite[Step 3]{fh:recognition}.  (Periodic Edges) follows from the assumption that $f \restrict G_r$ is a \ct\ and item (c) above. 

\medspace

\noindent{\bf Step 6:}  \  Modify   $\fG$ from Step 5 so that it satisfies (Vertices), (Linear Edges) and (NEG Nielsen Paths).  Additionally, the modified $\fG$  satisfies (Completely Split) on all \noneg\ edges.

\medspace

By property (Z) of Theorem~\ref{2.19} there is an irreducible stratum $H_p$ such that  each stratum between $H_p$ and $H_N$ is a zero stratum that is a component of $G_{N-1}$.   The edges $E_1,\ldots, E_p$ of $G_p \setminus G_r$ are non-fixed \noneg\ edges with terminal endpoints in $G_r$ and  initial endpoints  that have valence one in $G_{N-1}$.  The proof of this assertion is essentially the same as the proof of (v) and is left to the reader. The three conditions to be achieved depend only on the \noneg\ edges $E_i$.     Since $\fG$ is rotationless  and satisfies Theorem~\ref{2.19},  $f(E_i) = E_i u_i$ for some  $u_i \subset G_r$.   By item (c) above, $u_i$ is non-trivial.

Suppose that $C$ is a rank one component of $G_r$ and that the unique vertex $w$ of $C$ is the terminal endpoint of $E^1,\ldots, E^q \subset \{E_1,\ldots, E_p\}$. Then $C$ has a single edge $e$   and $f(E^j) = E^je^{d_j}$ for some $d_j \ne 0$.    If $w$ is not the endpoint of an edge in $H_N$ then redefine $f$ on the $E^j$'s by $f(E^j) = E^j e^{d_j - d_1}$.   The new map  still represents $\phi$  and none of our established properties are lost.  The edge $E^1$ is now fixed and so can be collapsed.  After these moves, $w$ is the endpoint of at least one edge in $H_N$ and so is principal for $f$.  
We assume now that the unique vertex of each rank one component of $G_r$ is the endpoint of an edge in $H_N$.

With that special case out of the way, one just applies \cite[Step 5 pages 91--93]{fh:recognition},  applying Lemma~\ref{finding a fixed point}  as described in Section~\ref{section one edge}.   

\medspace

\noindent{\bf Step 7:}  \  Modify   $\fG$ from Step 6 so that it satisfies  (Completely Split).

\medspace

We can apply \cite[Step 6]{fh:recognition} without change.

This completes the proof of Theorem~\ref{extend or reduce}.
\qed

\section{Finding  $\Fix(\Phi)$ }  \label{s:Finding Fix}  The goal of this section is to give another proof of the result of Bogopolski-Maslakova \cite{bm:fix} that there is an algorithm that, given $\Phi\in\Aut(F_n)$, computes $\Fix(\Phi)$.

\subsection{The periodic case}\label{s:periodic fix}
In this section, we examine the special case that  $\Phi$ is periodic. The analysis in this section will parallel that of the general case.

Recall from Section~\ref{s:markings} that if $G$ is a marked graph and $\ti v \in \ti G$ is a lift to the universal cover of $v \in G$ then there is an isomorphism $J_{\ti v} :\pi_1(G,v) \to \cT(\ti G)$   given by mapping $[\circuit]$ to the covering translation $\trans$ of $\ti G$ that takes $\ti v$ to the terminal endpoint   of the lift $\tilde{\circuit}$ of $\circuit$ with initial endpoint $\ti v$.

\begin{lemma}\label{l:z periodic}  Suppose that $h : G \to G$ is a periodic homeomorphism of a marked graph $G$  and that $\ti h : \ti G \to \ti G$ is a periodic lift of $h$  to the universal cover $\ti G$.  Then 
\begin{enumerate}
\item \label{item:nonempty fix}$\Fix(\ti h) \ne \emptyset$ and a point   $\ti v \in \Fix(\ti f)$ can be found algorithmically.
\item \label{item:centralizer}  If    $\ti v \in \Fix(\ti h)$ projects to $v \in G$  then $J_{\ti v}(  \pi_1(\Fix(h),v)) =  \zt(\ti h)$. (See Definition~\ref{d:zt}.)
  \end{enumerate}
\end{lemma}

\proof After subdividing if necessary we may assume that $h$, and hence $\ti h$, pointwise fixes each edge that it preserves.  We are now in the setting of Bass-Serre theory and we use its language.  Note that $\ti h$ is not hyperbolic, for otherwise $\ti h$ has infinite order. Hence $\ti h$ is elliptic; equivalently $\Fix(\ti h) \ne \emptyset$.  It is algorithmic to find a fixed point $\ti v$. Indeed, if $\ti x\in\ti G$ then the midpoint of $[\ti x,\ti h(\ti x)]$ is fixed. This completes the proof of \pref{item:nonempty fix}.

 For \pref{item:centralizer}, let  $\Fix(\ti h)$ denote the subtree of $\ti G$ consisting of $\ti h$-fixed edges. Given $[\gamma] \in \pi_1(\Fix(h),v)$, let $\ti \gamma$ be the lift of $\gamma$ that begins at $\ti v$ and let $\ti w$ be the terminal endpoint of $\ti \gamma$.  Then $T = J_{\ti v}(\gamma)$  is the covering translation  that carries $\ti v$ to $\ti w$.  Since $\ti h$ fixes $\ti v$ and $\gamma \subset \Fix(h)$, \  $\ti h$ fixes  $\ti w$.   In particular, $\trans\circ\ti h(\ti v)=\trans(\ti v)=\ti w=\ti h(\ti w)=\ti h\circ\trans(\ti v)$ and so $\trans$ and $\ti h$ commute. We see that  $J_{\ti v}(  \pi_1(\Fix(h),v))$ is contained in $\zt(\ti h)$.  To see surjectivity, let $\trans\in\zt(\ti h)$. Then $\trans(\ti v)\in\Fix(\ti h)$. Since $\Fix(\ti h)$ is a tree, $[\ti v, \trans(\ti v)]$ descends to a closed path in $\Fix(h)$ based at $v$.
 \endproof
 
We could not find a reference for the following result so we have included a proof.

\begin{lemma}\label{l:nonrotationless fix}
There is an algorithm that, given periodic $\Phi\in \Aut(F_n)$, computes $\Fix(\Phi)$.  
\end{lemma}

\begin{proof}
Let $\phi\in\Out(F_n)$ be represented by $\Phi$. Since the only periodic automorphisms of $\Z$ are the identity and $x\mapsto -x$, we may assume that $n\ge 2$. The relative train track algorithm of \cite{bh:tracks} (see Theorem~\ref{thm:rtt}) finds a periodic homeomorphism $h: G \to G$ of a marked graph representing $\phi$. 

We recall some  notation from  Section~\ref{s:markings}.  The marking homotopy equivalence is $\mu : (R_n,*) \to (G,\star)$.  Via a lift of $\star$ to $\ti \star \in \ti G$ we have  an identification of   $\cT(\ti G)$  with $F_n$, an isomorphism $J_{\ti \star} : \pi_1(G,\star) \to \cT(\ti G)$ and   a lift $\ti h : \ti G \to \ti G$ that can be found algorithmically and that corresponds to $\Phi$ in a sense made precise in Section~\ref{s:markings}.  The key points for us are that $\ti h$ is periodic and (Lemma~\ref{identify Fix}) that  $\zt(\ti h)$ and $\Fix(\Phi)$ are equal when viewed as subgroups of $\cT(\ti G)$.    Since  $F_n$ has been identified via $\mu_\#$ with $\pi_1(G, \star)$,  our goal is to find $J_{\ti \star}^{-1}\zt(\ti h) < \pi_1(G,\star)$. By  Lemma~\ref{l:z periodic} we can algorithmically find an element $\ti v \in \Fix(\ti h)$.  Moreover, letting $v\in G$ be the projection of $\ti v$ and $H :=  \pi_1(\Fix(h), v) < \pi_1(G,v)$, we have  $J_{\ti v} (H) = \zt(\ti h) $.  Let $\ti \eta$ be the path in $\ti G$ from $\ti \star$ to $\ti v$ and let $\eta \subset G$ be its projection.  A quick chase through the definitions shows that $J_{\ti \star}^{-1}J_{\ti v} :  \pi_1(G,v) \to \pi_1(G,\star)$ is defined by $[\gamma] \to [\eta \gamma\eta^{-1}]$.   Thus $\Fix(\Phi)$ is identified with  $H^\eta < \pi_1(G,\star)$.

To recap, the algorithm is  $$\Phi \rightsquigarrow \ti h \rightsquigarrow \ti v \in \Fix(\ti h) \rightsquigarrow H = \pi_1(\Fix(h),v) < \pi_1(G,v) \rightsquigarrow H^\eta < \pi_1(G, \star)$$
 
\end{proof}

\subsection{A $G$-graph of Nielsen paths}\label{s:stallings}
For the remainder of Section~\ref{s:Finding Fix} we assume that $\fG$ is a \ct\ representing (a necessarily rotationless) $\phi\in\Out(\f)$.

Let $\Sigma$ be a (not necessarily connected) graph. It is often useful to work in the Stallings category of graphs labeled by $\Sigma$ or $\Sigma$-graphs, i.e.\ an object is a graph $H$ with a cellular immersion $H\to \Sigma$ and a morphism from $H\to \Sigma$ to $H'\to \Sigma$ is a cellular immersion $H\to H'$ making the following diagram commute:
$$
\begindc{0}[10]
\obj(30,10)[31]{$\Sigma$}
\obj(10,50)[15]{$H$}
\obj(50,50)[55]{$H'$}
\mor{15}{31}{}[\atright,\solidarrow]
\mor{55}{31}{}[\atright,\solidarrow]
\mor{15}{55}{}[\atright,\solidarrow]
\enddc
$$
The map to $\Sigma$ is often suppressed. In this section we assume that $H$ is finite but in Section~\ref{s:stallings_N} we   allow $H$ to be infinite. If we give $H$ the $CW$-structure whose vertex set is the preimage of the vertex set of $\Sigma$, then the resulting edges of $H$ (often called {\it edgelets}) are {\it labeled} by their image edges in $\Sigma$ and we refer to the oriented edges of $\Sigma$ as the set of {\it labels}.  An edge-path is {\it labeled} by its sequence of oriented edges. The {\it 
core of $H$} is the $\Sigma$-graph that is the union of all immersed circles in $H$. $H$ is {\it core} if it is its own core.

$\Sigma$-graphs are useful because, on each component of $H$, the immersion $H \to \Sigma$ induces an injection on the level of $\pi_1$ and so $H$ is a {\it geometric realization} of a collection of  conjugacy classes $\pi_1(\Sigma)$ of subgroups of $\f$ indexed by the components of $H$. Our goal in this section is to construct a $G$-graph that is the geometric realization of the collection of conjugacy classes $[\Fix(\Phi)]$ as $[\Phi]$ varies over isogredience classes $\PA(\phi)/\sim$ of principal automorphisms representing $\phi$. 
 
\begin{review}\label{r:Nps}
We precede the construction with a quick review of Nielsen paths in a \ct.  Since Nielsen paths with endpoints  at vertices split as products of fixed edges and \iNp s, we focus on \iNp s.  There are only two sources of indivisible Nielsen paths $\mu$. By the (Vertices) property of a \ct, the endpoints of \iNp s are always vertices. If $E\in\lin(f)$, the set of linear edges in $G$, then $f(E) =  Ew_E^d$ for some \twistpath\ $w_E$ and some $d \ne 0$ and $Ew_E^k\bar E$ is an indivisible Nielsen path for $k\not= 0$. By the (\noneg\ Nielsen Paths) property of a \ct, all \iNp s of \noneg -height have this form. To $E$ we associate a $G$-graph $Y_{\mu_E}$ which is a lollipop. Specifically, $Y_{\mu_E}$ is the union of an edge labeled $E$ and a circle labeled $w_E$ attached to the terminal end of $E$.  Note that  each $Ew_E^k\bar E \subset G$ lifts to a path in $Y_{\mu_E}$ with both endpoints at the initial vertex of the edge labeled $E$. The other possibility is an indivisible Nielsen path $\mu$ of $EG$-height, say $r$. In this case, $\mu$ and $\bar \mu$ are the only \iNp s of height $r$ and the initial edges of $\mu$ and $\bar \mu$ are distinct edges of $H_r$  \cite[Lemma~4.19]{fh:recognition}. Further, $\mu=\alpha\bar\beta$ where $\alpha$ and $\beta$ are $r$-legal and the turn $(\bar\alpha, \bar\beta)$ is illegal of height $r$ \cite[Lemma~2.11]{fh:recognition}. See \cite[Section~4]{fh:recognition} for details on Nielsen paths in \ct s.
\end{review}

 \begin{definition} \label{defn:S(f)} We construct a $G$-graph $\stallings(f)$ as follows. If $n=1$, then $G=R_1$ (a rank 1 rose) and $\stallings(f):=G$. Otherwise, start with the subgraph $\stallings_{1}(f)$ of $G$ consisting of all vertices in $\Fix(f)$  and all fixed edges.   For each $E \in \lin(f)$, attach the lollipop $Y_{\mu_E}$ to $\stallings_1(f)$ at the initial vertex of $E$ thought of as a vertex in $\stallings_1(f)$.  For each \eg\ stratum with an \iNp\ of that height, choose  one $\mu$ of that height (there are only two and they differ by orientation)    and   attach an  edge, say $E_\mu$, labeled by $\mu$ to   $\stallings_1(f)$  with initial and terminal endpoints equal to those of $\mu$.  This completes the construction of the graph $\stallings(f)$.  There is a natural map $h : \stallings(f) \to G$ given by inclusion on $\stallings_{1}(f)$ and by the $G$-graph structures on  each $Y_{\mu_E}$ and  $E_\mu$.   By construction,  $h$ is  an immersion away from the attaching points in $\stallings_{1}(f)$ and is a local homeomorphism at each vertex in $\stallings_{1}(f)$.  Thus $\stallings(f)$ is a $G$-graph.

Let $v\in\Fix(f)$. We abuse notation slightly by also denoting the unique lift of $v$ in $\stallings_1(f)$ by $v$. Define $\stallings(f,v)$ to be the component of $\stallings(f)$ that contains $v$. It is possible that $\stallings(f,v)$ is not core.
\end{definition}

\vspace{.1in}

\begin{figure}[h!] 
\centering
\includegraphics[width=0.3\textwidth]{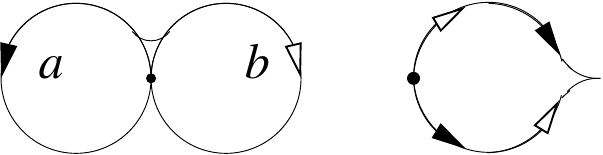}
\caption{A \ct\ $\fG$ given by $a\mapsto ab, b\mapsto bab$ and the graph $\stallings(f)$.  $\stallings_1(f)$ is the unique vertex of $G$. $\stallings(f)$ is the closed Nielsen path $\mu = ab\bar a \bar b$.}
\label{f:egstallings}
\end{figure}
 
\begin{figure}[h!]
\centering
\includegraphics[width=0.5\textwidth]{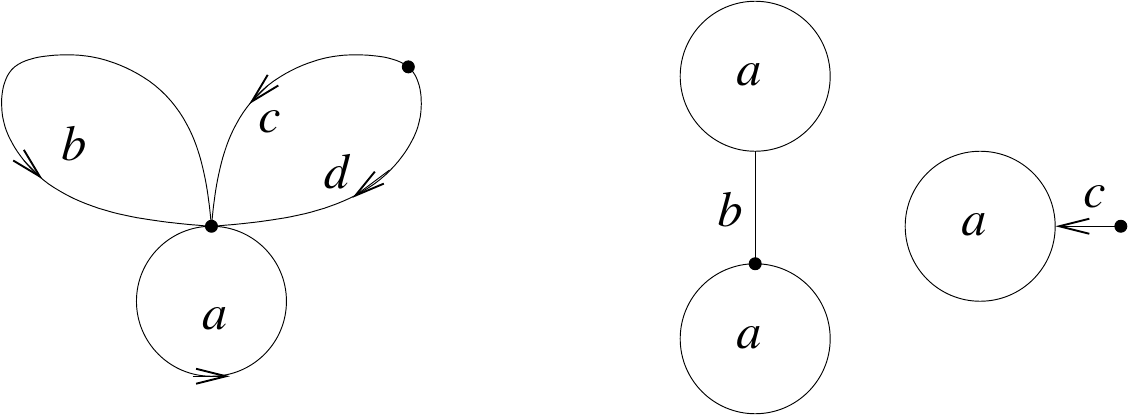}
\caption{A \ct\ $h:H\to H$ given by $a\mapsto a, b\mapsto ba^2, c\mapsto ca, d\mapsto db$ and the graph $\stallings(h)$. $\stallings_{1}(h)$ is the loop $a$ and the common initial vertex of $c$ and $d$.  Add the lollipop associated to $b$ to the former component and the lollipop associated to $c$ to the latter to make $\stallings(h)$.}
\label{f:stallingspg}
\end{figure}

\begin{remark}\label{r:stallings}
Since the construction of $\fG$ is algorithmic, finding the set of \iNp s of \eg\ height, and finding $\Fix(f)$  are all algorithmic, it follows that the constructions of $\stallings(f)$ and $\stallings(f,v)$ are algorithmic.
\end{remark}

\begin{lemma} \label{lifting to S(f)}  
\begin{enumerate}
\item 
A path $\sigma \subset G$ with endpoints at vertices lifts  to  a (necessarily unique) path $\hat \sigma \subset \stallings(f)$ with endpoints in $\stallings_1(f)$ if and only if $\sigma$ is a Nielsen path. Moreover, $\sigma$ is closed if and only if $\hat \sigma$ is closed.
\item
If $f$ represents $\phi\in\Out(F_n)$ then non-trivial $\phi$-periodic (equivalently $\phi$-fixed) conjugacy classes in $F_n$ are characterized as those classes represented by circuits in $\stallings(f)$. 
\end{enumerate}
\end{lemma}

\proof
 (1): The main statement is an immediate consequence of the  construction of $\stallings_1(f)$ and the fact that a Nielsen path in  a \ct\ is the concatenation of fixed edges and \iNp s with endpoints at vertices. The moreover part follows from the fact that each vertex in $G$ has a unique lift into $\stallings_1(f)$.

(2): Suppose that the circuit $\sigma \subset G$ represents a $\phi$-fixed conjugacy class.  \cite[Lemmas 4.11 and 4.25]{fh:recognition} imply that   $f^k_\#(\sigma)$ has a unique complete splitting for all sufficiently large $k$  and hence $\sigma$ has a unique complete splitting.  It follows that each term in the splitting is a periodic, and hence fixed, Nielsen path.   Viewing $\sigma$ as a closed path with endpoint at one of the splitting vertices, we can lift $\sigma$ to a path  $\hat \sigma \subset \stallings(f)$ with endpoints in $\stallings_1(f)$.  Since each vertex in $G$ has a unique lift into $\stallings_1(f)$, $\hat \sigma$ is a closed path and hence a circuit (because it projects to a circuit).  Conversely, every circuit in $\stallings(f)$ projects in $G$ to a concatenation of fixed edges and \iNp s.
\endproof

\begin{definition}\label{d:PS(f)}
We will also have need of the subgraph $\pstallings(f):=\cup_v \pstallings(f,v)$ of $\stallings(f)$ where the union is over principal vertices $v$ of $G$.  Equivalently (see Lemma~\ref{l:not principal}), $\pstallings(f)$ is obtained from $\stallings(f)$ by removing components consisting of a single non-principal vertex. The construction of $\pstallings(f)$ goes exactly as in Definition~\ref{defn:S(f)} except we start with the subgraph $\pstallings_1(f)$ consisting of all principal vertices and fixed edges in $\Fix(f)$.
\end{definition}

\subsection{$\Fix(\Phi)$ for rotationless $\phi$} 
In this section we compute $\Fix(\Phi)$ for  (not necessarily principal) $\Phi\in\Aut(F_n)$ representing rotationless $\phi\in\Out(\f)$.
We begin with the analog of Lemma~\ref{l:z periodic}(2). Recall that the unique lift of $v \in \Fix(f)$ in $\stallings_1(f,v)$ is denoted $v$.

\begin{lemma}  \label{lift to stallings(f)}   
For each vertex  $v \in \Fix(f) $ and lift $\ti v \in \ti G$    let $\hat J_{\ti v}  = J_{\ti v} h_\#:\pi_1(\stallings(f,v), v ) \to \cT(\ti G)$  where $h_\# :\pi_1(\stallings(f,v),v) \to  \pi_1( G,v)$ is induced by the immersion $h :\stallings(f,v) \to G$.     Let $\ti f : \ti G \to \ti G$ be the lift of $f$ that fixes $\ti v$.  Then $\hat J_{\ti v}$ is injective and has  image equal to $\zt(\ti f)$.
\end{lemma} 

\begin{proof}
$\hat J_{\ti v}$ is injective because $h_\#$ is injective and $J_{\ti v}$ is an isomorphism.  

If $\hat \gamma \subset \stallings(f,v)$ is a closed path based at $v$   then $\gamma:= h(\hat \gamma) \subset G$ is a closed Nielsen path based at $v$ by Lemma~\ref{lifting to S(f)}.     Lift $\gamma$ to a Nielsen path $\ti \gamma \subset \ti G$ for $\ti f$ with initial endpoint $\ti v$.  By definition,  $T :=\hat J_{\ti v}([\gamma])$  is the covering translation that maps $\ti v$ to the terminal endpoint $\ti w$ of $\ti \gamma$.  Since   $\ti w \in \Fix(\ti f)$, we have $\ti f\circ\trans(\tilde v)=\ti f(\ti w)=\ti w=\trans(\ti v)=\trans\circ\ti f(\ti v)$ and so $\ti f\circ\trans=\trans\circ\ti f$, i.e.\ $\trans\in\zt(\ti f)$. 

To see that  $\hat J_{\ti v}$ is surjective, let $T\in\zt 
(\ti f)$. Then $\trans(\ti v)$ is fixed by $\ti f$ and the path $\ti \gamma$ from $\ti v$ to $\ti w$ is a Nielsen path for $\ti f$ and so projects to a closed Nielsen path $\gamma  \subset G$ based at $v$ that lifts to a closed Nielsen path  $\hat \gamma \subset \stallings(f)$ based at $\hat v$.   By construction, $\hat J_{\ti v}[\hat \gamma] = T$.
\end{proof}

\begin{lemma} \label{l:rotationless fix}
There is an algorithm that, given rotationless $\phi\in \Out(F_n)$, computes $\Fix(\Phi)$ for $\Phi\in\phi$.
\end{lemma}

\begin{proof}    Let $\fG$ be a \ct\ representing the element $\phi \in \Out(F_n)$ determined by $\Phi$.  Subdivide if necessary so that every isolated fixed point of $f$ is a vertex. The setup is similar to that of Lemma~\ref{l:nonrotationless fix}. The marking homotopy equivalence is $\mu : (R_n,*) \to (G,\star)$.  Via a lift of $\star$ to $\ti \star \in \ti G$ we have  an identification of   $\cT(\ti G)$  with $F_n$, an isomorphism $J_{\ti \star} : \pi_1(G,\star) \to \cT(\ti G)$ and   a lift $\ti f : \ti G \to \ti G$ that can be found algorithmically and that corresponds to $\Phi$ in a sense made precise in Section~\ref{s:markings}.  The key point for us is  Lemma~\ref{identify Fix} which states that  $\zt(\ti f)$ and $\Fix(\Phi)$ are equal when viewed as subgroups of $\cT(\ti G)$.

Apply Lemma~\ref{finding a fixed point}  to decide if $\Fix(\ti f)  =  \emptyset$. 
 
If  $\Fix(\ti f) \ne \emptyset$ then Lemma~\ref{finding a fixed point} finds $\ti v \in \Fix(\ti f)$. Since isolated fixed points of $f$ are vertices, we may assume that $\ti v$ is a vertex. Indeed, by Convention~\ref{c:linear} non-isolated fixed points only occur in fixed edges.

Let $v\in G$ be the projection of $\ti v$, let   $\eta \subset G$ be the  projection of the path  $\ti \eta\subset \ti G$ from $\ti \star$ to $\ti v$, and let $h :\stallings(f,v) \to G$ be the immersion given by the $G$-structure. Define $H  =  h_\#(\pi_1(\stallings(f,v),v))< \pi_1(G,v)$. Arguing  as in Lemma~\ref{l:nonrotationless fix}, with Lemma~\ref{l:z periodic}(2) replaced by Lemma~\ref{lift to stallings(f)} we conclude that $\Fix(\Phi)$ is identified with  $H^\eta < \pi_1(G,\star)$. 
Since $F_n$ has been identified via $\mu_\#$ with $\pi_1(G, \star)$, we are done. In summary,
   $$\Phi \rightsquigarrow \ti f\rightsquigarrow \ti v \in \Fix(\ti f) \rightsquigarrow H = h_\#(\pi_1(\stallings(f,v),v)) < \pi_1(G,v) \rightsquigarrow H^\eta < \pi_1(G, \star)$$

If  $\Fix(\ti f) = \emptyset$  then $\ti f$ is not principal \cite[Corollary 3.17]{fh:recognition} and so $\rk(\Fix(\Phi)) \le 1$ \cite[Remark 3.3]{fh:recognition} and $\Fix_N(\partial\ti f) = \Fix_N(\partial\Phi)$ has at most two points.   Lemma~\ref{finding a fixed point}   finds  a completely split path $\ti \sigma$ that generates a ray $\ti R$.  There are two subcases.   If the projected image $ \sigma \subset G$   is not a Nielsen path then $|f^k_\#(\sigma)| \to \infty$ and   \cite[Proposition~I.1]{gjll:index} implies that the terminal endpoint $P$ of   $\ti R = \ti \sigma \cdot  \ti f_\#(\ti \sigma) \cdot \ti f^2_\#(\ti \sigma)\cdot \ldots$, which is evidently fixed by $\partial \ti f$, is an attractor for the action of $\partial \ti f$,  is contained in $\Fix_N(\partial\ti f)$ and  is not the endpoint of an axis of a  covering translation.   If $\zt(\ti f)$  contains a non-trivial element $T$  then the $T$-orbit of $P$ would be an infinite set in $\Fix_N(\partial\ti f)$.  This contradiction shows that $\zt(\ti f)$, and hence $\Fix(\Phi)$,  is trivial.  The remaining subcase is that $ \sigma$ is a Nielsen path. Let $\ti v$ and $\ti w$ be the initial and terminal endpoints of $\ti \sigma$ respectively and let   $T $ be the covering translation that carries $\ti v$ to $\ti w$.  Since  $\ti f_\#(\ti \sigma)$ is the unique lift of $\sigma$ with initial endpoint at $\ti w$,  we have $\ti f(\ti v) = T(\ti v) =\ti w$ and $\ti f(\ti w) = T(\ti w)$.   Thus $\ti f T(\ti v) = T \ti f(\ti v)$ and $T \in \zt(\ti f)$.   Equivalently $J_{\ti v}([\sigma]) \in \zt(\ti f)$.  Letting $H < \pi_1(G,v)$ be the maximal cyclic subgroup that contains $[\sigma]$,  we have $J_{\ti v}(H)  =\zt(\ti f)$ and  the usual argument shows that  $\Fix(\Phi)$ is identified with  $H^\eta < \pi_1(G,\star)$. In summary,
 $$\Phi \rightsquigarrow \ti f\rightsquigarrow \ti \sigma \rightsquigarrow   \langle [\sigma] \rangle  < H < \pi_1(G,v) \rightsquigarrow H^\eta < \pi_1(G, \star)$$ 
\end{proof}

\subsection{The general case}

\begin{prop}[Bolgopolski-Maslakova \cite{bm:fix}]\label{p:fix}
There is an algorithm that, given $\Phi\in\Aut(F_n)$, computes $\Fix(\Phi)$.
\end{prop}

\begin{proof}
The case of $\Phi\in\phi$ with $\phi$ rotationless was handled in Lemma~\ref{l:rotationless fix}.

Suppose then that $\phi$ is not rotationless. In Corollary~\ref{c:kn} we computed $M$ so that $\phi^M$ is rotationless. So quoting Lemma~\ref{l:rotationless fix} again, $\Fix(\Phi^M)$ can be computed algorithmically. We are reduced to finding the fixed subgroup of the periodic action of $\Phi$ on $\Fix(\Phi^M)$. More generally, we are reduced to finding $\Fix(\Phi)$ for a finite order $\Phi\in\Aut(F_n)$, and this is the content of Lemma~\ref{l:nonrotationless fix}.
\end{proof}

\section{$\pstallings(f)$ and $\Fix(\phi)$}\label{s:fix phi}
Let $\fG$ be a \ct\ representing $\phi\in\Out(F_n)$. In this section we characterize the components of $\pstallings(f)$ and show that $\pstallings(f)$ is the geometric realization of $\Fix(\phi)$ defined as follows.

\begin{definition}
For $\phi\in\Out(\f)$, $\Fix(\phi)$ is defined to be the collection of $[\Fix(\Phi)]$ indexed by isogredience classes $[\Phi]\in\PA(\phi)/\sim$.
\end{definition}

Recall two facts from the review in Section~\ref{s:principal}:
\begin{itemize}
\item
If $\ti f$ is a principal lift of $f$ then $\Fix(\ti f)$ is non-empty.
\item
If $\ti f_1, \ti f_2$ are principal lifts of $f$ then $\ti f_1$ and $\ti f_2$ are isogredient iff $\Fix(\ti f_1)$ and $\Fix(\ti f_2)$ have equal projections in $G$. 
\end{itemize}
 
By the definition of a \ct, endpoints of indivisible Nielsen paths are vertices and so Lemma~\ref{l:not principal}(\ref{i:ct principal}) implies that, for principal $\ti f$, $\Fix(\ti f)$ consists of vertices and edges. It follows that there are only finitely many equivalence classes of principal lifts of $f$ and   that there is a 1-1 correspondence between isogredience classes $[\ti f]$ of principal lifts $\ti f$ of $f$ and Nielsen classes $[v]$ of principal vertices $v$ in $G$ given by $[\ti f]\leftrightarrow [v]$ iff $\ti f$ fixes some lift of $v$. It is algorithmic to tell if a vertex of $G$ is principal and to find its Nielsen class (Lemmas~\ref{l:not principal}(\ref{i:ct principal}) and \ref{lifting to S(f)}).

By construction, for principal vertices $v$ and $v'$, $\pstallings(f,v)=\pstallings(f,v')$ iff $[v]=[v']$. We write $\pstallings(f,v)=\pstallings(f,[v])=\pstallings([\ti f])$ where $[v]\leftrightarrow[\ti f]$ and see that $\pstallings(f)=\sqcup_{[v]}\pstallings(f,[v])=\sqcup_{[\ti f]}\pstallings([\ti f])$ where $[v]$ runs over Nielsen classes of principal vertices in $G$ and $[\ti f]$ runs over isogredience classes of principal lifts of $f$.
 
If $\ti f\leftrightarrow\Phi$ and $\ti f'\leftrightarrow\Phi'$ are isogredient then $\Fix(\Phi)$ and $\Fix(\Phi')$ are conjugate and $\zt(\ti f)$ and $\zt(\ti f')$ are conjugate. Hence, the isogredience classes of $\Phi$ and $\ti f$ determine a conjugacy class $\Fix([\Phi])$ of subgroup of $F_n$ and a conjugacy class $\zt([\ti f])$ of subgroup of $\cT(\ti G)$. The sets $\Fix([\Phi])$ and $\zt([\ti f])$ correspond under the identifications of Section~\ref{s:markings}; we denote this by $\Fix([\Phi])\leftrightarrow\zt([\ti f])$.  If  $[\ti f]\leftrightarrow [v]$ and  $\ti f\leftrightarrow\Phi$ then $\pstallings(f,[v])$ is the geometric realization of  $\Fix([\Phi])$ and $\pstallings(f)$ is the geometric realization of $\Fix(\phi)$.

\section{Possibilites for $[\Fix(\Phi)]$}
In this section we algorithmically find the possibilities for $[\Fix(\Phi)]=\Fix([\Phi])$  for $\Phi\in\phi$ with $\phi$ rotationless (Corollary~\ref{c:possibilities}).

\begin{corollary}\label{c:possibilities}
Let $\phi\in\Out(\f)$ be rotationless.
\begin{enumerate}
\item
There are finitely many $\f$-conjugacy classes in $$\{\Fix(\Phi)\mid \Phi\in\Aut(\f) \mbox{ is principal and represents } \phi\}$$ These conjugacy classes are represented by the components of $\pstallings(f)$ where $\fG$ is a \ct\ for $\phi$. In particular, they can be computed algorithmically.  
\item
For all  $\Phi\in\phi$, $\Fix(\Phi)$ is root-free. Conversely, if $w\not= 1\in\f$ is root-free and $[w]$ is $\phi$-invariant then \begin{enumerate}
\item
there is $\Phi\in \PA(\phi)$ such that $w\in\Fix(\Phi)$ and
\item
there is non-principal $\Phi'\in\phi$ such that $\Fix(\Phi')=\langle w\rangle$.
\end{enumerate}
\end{enumerate}
\end{corollary}

\begin{proof}
(1): This is the content of Section~\ref{s:fix phi}.   

(2): $\Fix(\Phi)$ is always root-free.  The existence of $\Phi$ as in  (a) follows from \cite[Lemma 3.30 (1)]{fh:recognition}. 
The automorphism $\Phi_m:=i^m_w\circ\Phi$  fixes $w$ for all $m$ and is principal for only finitely many $m\in\Z$; see \cite[Lemma~4.40]{fh:recognition}. Take $\Phi'$ to be a non-principal $\Phi_m$. Since  $\Phi'$ is non-principal,    $\rk(\Fix(\Phi'))<2$ and we see that $\Fix(\Phi')=\langle w\rangle$.
\end{proof}

Lemma~\ref{trivial Fix} below seems obvious but we could not find a reference for it in the literature so we are including it here with a proof for completeness.  The following lemma is used in its proof.

\begin{lemma}  \label{filling circuit} For all non-trivial   $\phi$ there exists a filling conjugacy class that is not fixed by $\phi$.
\end{lemma}

\proof  After replacing $\phi$ by an iterate, we may assume that $\phi$ is rotationless.  Let $\fG$ be a topological representative of $\phi$ and let $\beta \subset G$ be a circuit representing a conjugacy class $[c]$ that  is not $\phi$-invariant.   Since $\phi$ is rotationless,  $[c]$ is not fixed by any iterate of $\phi$.  We may therefore choose $k$ so that the edge length of $f^k_\#(\beta)$ is greater than the edge length of $\beta$.   Choose a path $\alpha \subset G$ such that any circuit containing $\alpha$ is filling (any path $\alpha$ whose Whitehead graph does not have a cut-point will do \cite{rm:thesis}; see also \cite[Proof of Lemma~3.2]{bf:hyperbolicity}).  Choose a circuit $\tau$ that contains both $\alpha$ and $\beta$ as disjoint subpaths.  Decompose $\tau$ as a concatenation of subpaths  $\tau = \mu \beta$ and let $\tau_N = \mu \beta^N$.  The bounded cancellation lemma implies that  the edge length of $f^k_\#(\tau_N)$ is greater than the edge length of $\tau_N$ for all sufficiently large $N$.  Each such $\tau_N$ satisfies the conclusions of the lemma. 
\endproof    

\begin{lemma} \label{trivial Fix} Each infinite order $\phi \in \Out(F_n)$ is represented by $\Phi \in \Aut(F_n)$ with trivial fixed subgroup. 
\end{lemma}

\proof  Choose   (Lemma~\ref{filling circuit})  a filling conjugacy class $[a]$ that is not fixed by  $\psi = \phi^l$ where $l$ is chosen so that both $\psi$ and $\psi^{-1}$ are rotationless.  Let $\{a^\pm\}\subset \partial\f$ denote the endpoints of the axis of $a$.  Choose $\Phi$ representing $\phi$ such that $\partial\Phi(\{a^\pm\}) \cap \{a^\pm\} = \emptyset$  and let $\Phi_m = i_a^m \circ \Phi$.   We will prove that $\Fix(\Phi_m)$ is trivial for all sufficiently large $m$.

Assuming that $\Phi_m$  fixes some non-trivial $b_m$    for arbitrarily large $m$, we will argue to a contradiction.  For notational convenience we pass to a subsequence and   assume that   for all $m\ge 1$, $b_m$ is fixed by $\Phi_m$.  After passing to a further subsequence we may assume that $b^+_m \to P$ and $b^-_m \to Q$ for some $P,Q \in \partial F_n$.  From  $ \partial \Phi_m(b^+_m) =b^+_m$ we see that $\partial \Phi(b^+_m) = \partial i_a^{-m}(b^+_m)$  and hence that $P$ is either $a^+$ or $\partial \Phi^{-1}(a^-)$.   Similarly, $Q$  is either $a^+$ or $\partial \Phi^{-1}(a^-)$. 
 
 Let $\fG$ be a \ct\ representing $\psi$ and let  $\ti f_m$ be the lift of $\fG$ corresponding to $\Psi_m := \Phi_m^l$.   Let $T$ and $S_m$ be the covering translations corresponding to $a$ and $b_m$ respectively. Thus $T^\pm = a^\pm$ and $S_m^\pm = b_m^\pm$.  If $P \ne Q$ then the line connecting $a^+$ to $\partial\Phi^{-1}(a^-)$  is a weak limit of the  axes for $S_m$.  It follows that the axis of $T$ is a weak limit of a sequence of Nielsen axes for $\ti f_m$ and hence that  the axis of $T$ is an increasing union of arcs each lifting to $\pstallings(f)$. The axis of $T$ therefore  lifts to $\pstallings(f)$ and so represents a $\psi$-fixed conjugacy class by Lemma~\ref{lifting to S(f)}, contradiction.  

Suppose then that  $P = Q = \Phi^{-1}(T^-)$.  (The case that $P= Q = T^+$ is argued symmetrically.)   For all sufficiently large $m$, there is a neighborhood  $U_m^+$ of $T^+$ in $\partial F_n$ such that $\partial \Phi_m(U_m^+) \subset U_m^+$.   By \cite[Theorem I]{ll:dynamics},  $U_m^+$  contains a non-repelling periodic point      for the action of $\partial\Phi_m$.   Thus  $\Psi_m$  is a principal lift of $\psi$ and $U_m^+$ contains an element $B_m$ of $\Fix_N(\Psi_m)$.  As $m \to \infty$ we may assume that $B_m \to T^+$.

By Lemma~\ref{Fix+},  there exists a lift $\ti E_m$ of some $E_m \in \E$ such that the lift $\ti R_{E_m}$ of the eigenray generated by $E_m$ terminates at $B_m$.  We may assume without loss that $E = E_m$ is independent of $m$. The ray from the initial endpoint $\ti x_m$ of $\ti E_m$ to $b^+_m$  factors as a concatenation of Nielsen paths and so lifts into $\pstallings(f)$. After passing to yet another subsequence we may assume that  $\ti x_m$   converges to a point $X$.   If $X = T^+$ then the axis of $T$ is a weak limit of paths that lift into  $\pstallings(f)$ and so lifts to  $\pstallings(f)$.     As above, this gives the desired contradiction.   We may therefore assume that $X \ne T^+$.  In this case the axis of $T$ is a weak limit of lifts of $R_E$ and so the  periodic line that it projects to is a weak limit of $R_E$.   We will complete the proof by showing that each such periodic line $L$   is carried by $G_{r-1}$ where $H_r$ is the top stratum of $G$.   There is no loss in assuming that $E$ is an edge in $H_r$.  If $E$ is \noneg\ then \cite[Lemma 3.26(3)]{fh:recognition} completes the proof. If $E \subset H_r$ is \eg\ then $L$ is leaf in the attracting lamination $\Lambda_r$ associated to $H_r$ by \cite[Lemma 3.26(2)]{fh:recognition}. But every leaf   is  either contained in $G_{r-1}$ or is dense in $\Lambda_r$ by \cite[Lemma 3.1.15]{bfh:tits1} and the latter is impossible for a periodic line.    
\endproof

\section{A Stallings graph for $\Fix_N(\partial\Phi)$}\label{s:stallings_N}\label{s:S}
As usual, throughout Section~\ref{s:stallings_N} $\fG$ will denote a \ct\ for $\phi$. The goal of this section is to generalize Section~\ref{s:fix phi} by describing a (not necessarily finite) $G$-graph $\pstallings_N(f)$  with core $\pstallings(f)$ that in a sense described below represents the collection of $\Fix_N(\partial\Phi)$  (Definition~\ref{def:principal auto}) indexed by $\Phi\in\PA(\phi)$. This collection plays an important role in the main theorem of \cite{fh:recognition}. We will first define the graph $\pstallings_N(f)$  and  then relate it to the collection of $\Fix_N(\partial\Phi)$. $\Fix_N(\partial\Phi)$  has been of considerable interest; see for example \cite{jn:collected1, jn:unter2, ikt:fixN, co:bcc, gjll:index, bfh:tits0}.

\subsection{Definition of $\pstallings_N(f)$}\label{s:sn(f)}
The idea of the construction of $\pstallings_N(f)$ is to start with $\pstallings(f)$ and add a copy of $R_E$ for each $E\in\E_f$. This does not quite work since it is possible that $\partial R_E=\partial R_{E'}$ for $E\not= E'$. In the next paragraph we describe the ways that this can happen.

Suppose that $E \ne E' \in \E_f$ and that $\partial R_E =  \partial R_{E'}$.  Choose a lift $\ti E$ of $E$ to $\ti G$, let $\ti R_{E}$ be the lift of $R_E$ that begins at the initial endpoint $\ti v$ of $\ti E$  and let $P$ be  the terminal endpoint of $\ti R_E$.  By hypothesis, there is a lift $\ti R_{E'}$ of $R_{E'}$ that terminates at $P$. Let $\tilde E'$ be the initial edge of $\ti R_{E'}$ and let $\ti w$ be the initial vertex of $\ti E'$. If $\ti f$ is the lift that fixes $\ti v$ then $  \partial\ti f$ fixes $P$.  Since $\ti f$ is the only lift of $f$ that fixes $P$ (Remark~\ref{unique lift}), it follows that $\ti f$ also fixes $\ti w$.  The path $\ti \rho$ connecting $\ti v$ to $\ti w$ is therefore a Nielsen path for $\ti f$ and its image  $\rho \subset G$ is a Nielsen path for $f$. 
    The edges $E$ and $E'$ are of the same \eg\ height, say $r$, since the height of $E$ (resp.\ $E'$) is the height of $R_E$ (resp.\ $R_{E'}$) and edges of this height occur infinitely often in $R_E$ (resp.\ $R_{E'}$) and $R_E$ and $R_E'$ have a common tail. It follows from the facts in Review~\ref{r:Nps} that  $\rho$ is a concatenation of subpaths that are fixed edges or indivisible Nielsen paths, that $E$ (resp.\ $E'$) must be contained in an indivisible Nielsen subpath of height $r$, and that each of these indivisible Nielsen subpaths of height $r$ contributes an illegal turn of height $r$. By construction, $\rho$ has at most one illegal turn of height $r$ and so we conclude that $\rho$ is indivisible. Note that $\ti R_E$ and $\ti R_{E'}$ have a common terminal ray $\ti R_{E,E'}$ and that $\ti \rho = \ti \alpha \ti \beta^{-1}$ where  $\ti R_E = \ti \alpha  \ti R_{E,E'}$ and $\ti R_{E'} = \ti \beta  \ti R_{E,E'}$ giving $\rho$ the form $\alpha\beta^{-1}$ as in Review~\ref{r:Nps}.

\begin{definition}
We use the notation from the Definitions~\ref{defn:S(f)} and \ref{d:PS(f)}. Let $\CS_N(f)$ be the graph obtained from $\pstallings(f)$ as follows. For each \noneg\ $E \in \E_f$  and for each \eg\ $E\in \E_f$ that is not the initial edge of an \iNp, attach $R_E$ to $\pstallings(f)$ by identifying the initial endpoint of $R_E$ with the initial endpoint of $E$ thought of as a vertex in $\pstallings_{1}(f)$.  If $E\in \E_f$ belongs to an \eg\ stratum $H_r$ and $E$ is the initial edge of an \iNp\ $\rho$ of height $r$ then the initial edge $E'$ of $\rho^{-1}$ is also an edge of $\E_f \cap H_r$.    Subdivide the edge  labeled $\rho$ that was added  to $\pstallings(f)$ during stage two so that it is now two edges, one labeled $\alpha$ and the other $\beta^{-1}$.  Attach   $R_{E,E'}$ to $\pstallings(f)$ at the newly created vertex.
\end{definition}

\begin{remark}\label{r:cs algorithmic}
Given that the construction of $\pstallings(f)$ is algorithmic and that any initial segment of a ray $R_E$ of prescribed length can be explicitly computed, it follows that there is an algorithm that, given  $d>0$, constructs the $d$-neighborhood (in the graph metric) of $\pstallings(f)$ in $\CS_N(f)$.
\end{remark}

\begin{figure}[h!]
\centering
\includegraphics[width=0.6\textwidth]{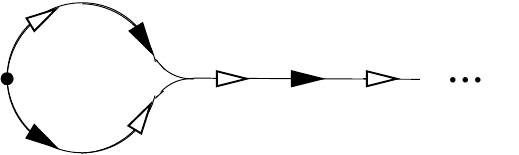}
\caption{$\CS_N(f)$ for $f$ as in Figure~\ref{f:egstallings}. $\CS_N(f)$ is obtained from $\pstallings(f)$ by adding the rays $R_{a,b}=babbab\dots$ and $R_{B}=BABBABAB\dots$.}
\label{f:sNeg}
\end{figure}
 
\begin{figure}[h!]
\centering
\includegraphics[width=0.4\textwidth]{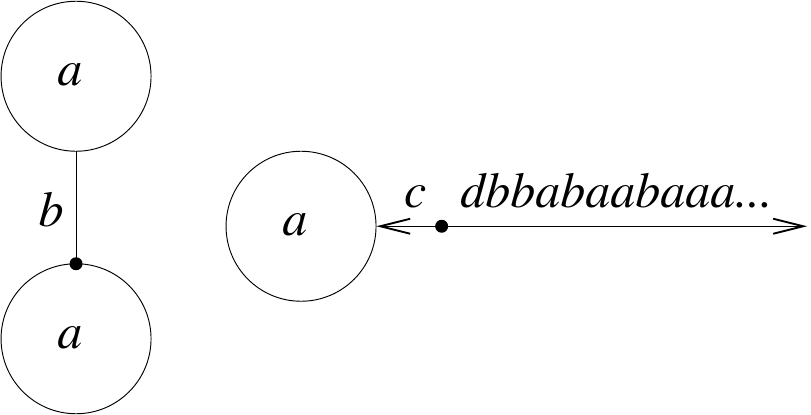}
\caption{$\CS_N(h)$ for $h$ as in Figure~\ref{f:stallingspg}. $\CS_N(h)$ is obtained from $\pstallings(h)$ by adding the eigenray $dbb\ldots$.}
\label{f:sNpg}
\end{figure}

By construction of $\CS_N(f)$, its components are in a 1-1 correspondence with the components of $\pstallings(f)$.  That is, components of $\CS_N(f)$ are in a bijective correspondence with Nielsen classes $[v]$ of principal vertices $v$ in $G$. Let $\CS_N(f,[v])$ denote the component of $\CS_N(f)$ containing $v$. We have $\CS_N(f)=\sqcup_{[v]} \CS_N(f,[v])$ where $[v]$ runs over Nielsen classes of principal vertices in $G$. We continue to identify the principal vertices of $G$ with the vertices of $\pstallings_1(f)$. We call these vertices the {\it principal vertices of $\CS_N(f)$}.

\subsection{Properties of $\AS_N(f)$}\label{s:properties of S}
In this section we record some properties of $\AS_N(f)$. We defined the term {\it core} of a $\Sigma$-graph $H$ in Section~\ref{s:stallings}. The {\it weak core of $H$} is the union of all properly immersed lines in $H$. $H$ is {\it weakly core} if it is its own weak core. We say that $H$ has {\it finite type} if it is the union of a finite graph and finitely many rays.

\begin{lemma}
$\CS_N(f)$ is a $G$-graph of finite type and all vertices have valence $\ge 2$. In particular, $\CS_N(f)$ is weakly core.
\end{lemma}

\begin{proof}
By construction, the labeling map $\CS_N(f)\to G$ is an immersion and $\pstallings_N(f)$ has finite type and its non-principal vertices have valence either two or three. By \cite[Lemma~4.14]{fh:recognition},  a principal vertex $v$ has at least two fixed directions in $G$, each corresponding to a fixed edge, an edge in $\lin(f)$, or an edge in $\E_f$.  By construction, each of these edges contributes a direction at $v$ in $\CS_N(f)$. Thus all principal vertices, and hence all vertices, have valence at least two.
\end{proof}

\begin{lemma}\label{l:boundary}
Let $\ti f:\ti G\to \ti G$ be principal and fix the vertex $\ti v\in\ti G$. The labelling map $\pstallings_N(f,v)\to G$ induces an embedding $\ti S\to\ti G$ of universal covers. The induced map $\partial \ti S\to\partial\ti G$ has image $\Fix_N(\partial\ti f)$.
\end{lemma}

\begin{proof}
The first conclusion follows since the labelling map is an immersion. To prove the second, note that every ray $R$ with initial vertex $v$ is either contained in $\pstallings(f,v)$ in which case the corresponding lift $\tilde R$ has endpoint in $\partial(\Fix(\ti f))$ or else is the concatenation of a Nielsen path and a ray $R_E$ in which case $\ti R$ has endpoint in $\Fix_+(\partial\ti f)$. Conversely, if $P\in \Fix_N(\partial\ti f)$ then the ray $[\ti v, P)$ is a concatenation of a Nielsen path and the lift of some $R_E$ by \cite[Lemma~4.36(2)]{fh:recognition}.
\end{proof}

\begin{corollary}\label{i:no circle or line}
In the notation of Lemma~\ref{l:boundary}, $\tilde S$ is the convex hull of $\Fix_N(\partial\tilde f)$, i.e.\ the union of all lines connected distinct points in $\Fix_N(\partial\tilde f)$. In particular, if $n>1$ then no component of $\AS_N(f)$ is a circle, an axis in $\ti G$, or a generic leaf of an attracting lamination of $f$.\qed
\end{corollary}

\begin{remark}\label{r:fix at infinity}
In the main theorem of \cite{fh:recognition}, a rotationless outer automorphism $\phi$ is characterized in terms of two invariants: one qualitative and the other quantitative. The collection of $\Fix_N(\partial\Phi)$ indexed by $\Phi\in\PA(\phi)$ is the qualitative invariant. From the results of this section, we see that $\pstallings_N(f)$ represents this qualitative invariant in the following sense. If $X$ is a component of $\pstallings_N(f)$ and $\tilde X\to\tilde G$ is a lift of the natural immersion $X\to G$, then the image of the induced map $\partial\tilde X\to\partial\f$ is $\Fix_N(\partial\Phi)$ for some $\Phi\in\PA(\phi)$. Conversely, if $\Phi\in\PA(\phi)$ then $\Fix_N(\partial\Phi)$ is the image of $\partial\tilde X\to\partial\f$ for some component $X$ of $\pstallings_N(f)$ and some lift $\tilde X\to\tilde G$.
\end{remark}

\section{Moving up through the filtration}\label{s:filtration}
Several of our applications are verified by working up through the filtration of a \ct\ $\fG$.   In this section we establish notation by recalling some notation and a lemma from \cite[Notation~8.2 and Lemma~8.3]{fh:abeliansubgroups}. The negative of the Euler characteristic will be important to us; we write $\chi^-$ for $-\chi$.

\begin{notn}
Recall from (Filtration) that the core of each filtration element is a
filtration element. The core filtration $$\emptyset=G_0=G_{l_0}\subset G_{l_1}\subset G_{l_2}\subset\dots\subset G_{l_K}=G_N = G$$
is defined to be the coarsening of the full filtration obtained by restricting to those
elements that are their own cores or equivalently have no valence one vertices. Note
that $l_1=1$ by (Periodic edges). For each $G_{l_i}$, let $H^c_{l_i}$ be the $i$Ðth stratum of the
core filtration. Namely $$H^c_{l_i}=\bigcup_{j=l_{i-1}+1}^{l_i} H_j$$ The change in negative Euler characteristic is denoted $\Delta_i\chi^-:=\chi^-(G_{l_i})-\chi^-(G_{l_{i-1}})$. Referring to Notation~\ref{n:zero}, if $H_{l_i}$ is $EG$ then $H_{u_i}$ denotes the highest irreducible stratum in $G_{l_i-1}$.
\end{notn}

\begin{lemma}[{\cite[Lemma 8.3]{fh:abeliansubgroups}}] \label{l:filtration}
\begin{enumerate}
\item
If $H^c_{l_i}$ does not contain any $EG$ strata then one of the following
holds.
\begin{enumerate}
\item
$l_i=l_{i-1}+1$ and the unique edge in $H^c_{l_i}$ is a fixed loop that is disjoint from $G_{l_{i-1}}$.
\item
$l_i=l_{i-1}+1$ and both endpoints of the unique edge in $H^c_{l_i}$ are contained in $G_{l_{i-1}}$.
\item
$l_i=l_{i_1}+2$ and the two edges in $H^c_{l_i}$ are nonfixed and have a common initial endpoint that is not in $H_{l_{i-1}}$ and terminal endpoints in $G_{l_{i-1}}$.
\end{enumerate}
In case (a), $\Delta_i\chi^-=0$; in cases (b) and (c), $\Delta_i\chi^-=1$.
\item
If $H^c_{l_i}$ contains an $EG$ stratum then $H_{l_i}$ is the unique $EG$ stratum in $H^c_{l_i}$ and there exists $l_{i-1}\le u_i< l_i$ such that both of the following hold.
\begin{enumerate}
\item
For $l_{i_1}< j\le u_i$, $H_j$ is a single nonfixed edge $E_j$ whose terminal vertex is in $G_{l_{i-1}}$ and whose initial vertex has valence one in $G_{u_i}$. In particular, $G_{u_i}$ deformation retracts to $G_{l_{i-1}}$   and $\chi(G_{u_i})=\chi(G_{l_{i-1}})$.
\item
For $u_i < j < l_i$, $H_j$ is a zero stratum. In other words, the closure of $G_{l_i}\setminus G_{u_i}$ is the extended $EG$ stratum $H^z_{l_i}$.
\end{enumerate}
If some component of $H^c_{l_i}$ is disjoint from $G_{u_i}$ then $H^c_{l_i}=H_{l_i}$ is a component of $G_{l_i}$ and $\Delta_i\chi^-\ge 1$; otherwise $\Delta_i\chi^-\ge 2$.
\end{enumerate}
\end{lemma}

\section{Primitively atoroidal outer automorphisms}
In this section we exhibit an algorithm to determine whether or not a $\phi\in\Out(\f)$ is primitively atoroidal (Corollary~\ref{c:primitively atoroidal}).

\begin{definition}  A conjugacy class in $F_n$ is {\em primitive} if it is represented by an  element in some basis of $F_n$.
An outer automorphism $\phi$ is {\it primitively atoroidal} if there does not exist a periodic conjugacy class for $\phi$ which is primitive.
\end{definition}

\begin{lemma}\label{l:h1}
Let $\fG$ be a \ct . Either some stratum is a fixed loop or all closed Nielsen paths are trivial in $H_1(G;\Z/2\Z)$.  
\end{lemma}

\begin{proof}
We use the notation of Section~\ref{s:filtration}. Suppose the statement holds for $G_{l_{l_i-1}}$. To get a contradiction, suppose the statement fails for $G_{l_i}$, that is suppose that $H_{l_i}^c$ is not a fixed loop and that some closed Nielsen path $\mu$ of height $l_i$ is non-trivial in $H_1(G_{l_i};\Z/2\Z)$. We go through the cases of Lemma~\ref{l:filtration}.

Recall that Nielsen paths are completely split and their complete splitting consists of fixed edges and \iNp s.  

In (1a),  $H_{l_i}$ is a fixed loop and we are assuming this is not the case.

In (1b), by the (NEG Nielsen Paths) property of a \ct\ $E_{l_i}:=H_{l_i}$ is either fixed or linear and if $E_{l_i}$ is linear then $\mu$ or $\mu^{-1}$ has a term of the form $E_{l_i}w_{l_i}^k\bar E_{l_i}$ with notation as in (NEG Nielsen Paths).

Suppose first that $E_{l_i}$ is fixed. 
  
\begin{itemize}
\item
If $\mu$ has only one occurrence of $E_{l_i}$ and no occurrences of its inverse (or the symmetric case) then $\mu$ is primitive. In this case, the union $\F$ of the  free factor system  $[G_{l_{i-1}}]$ with the conjugacy class of the cyclic subgroup generated by $\mu$ is a $\phi$-invariant free factor system   properly containing  $[G_{l_{i-1}}]$ and contained in $[G_{l_{i}}]$; see Example~\ref{fixed edge}.  The (Filtration) property of a \ct\ implies that  $\F =  [G_{l_{i}}]$ contradicting our  assumption that $H_{l_i}^c$ is not a fixed loop.
\item
 Consider the (cyclic) sequence of $E_{l_i}^{\pm 1}$'s that occurs in $\mu$ and suppose there are consecutive occurrences with the same orientation, say $\mu=\dots E_{l_i}xE_{l_i}\dots$.   Then $\mu' = E_{l_i}x$ is a closed Nielsen path as in the preceding bullet, contradiction. 

\item
The remaining case is that $\mu$ or $\mu^{-1}$ has the form $E_{l_i}x_1E_{l_i}^{-1}x_2E_{l_1}\dots E_{l_i}^{-1}x_N$, i.e.\ the orientations of the $E_{l_i}$'s alternate. Each $x_i$ is a closed Nielsen path and by induction is trivial in $H_1(G;\Z/2\Z)$. If follows that $\mu$ is also trivial in $H_1(G;\Z/2\Z)$
\end{itemize}

 Next suppose that $E_{l_i}$ is linear and $f(E_{l_i})=E_{l_i}\cdot w_{l_i}^{d_{l_i}}$ (with notation as in (Linear Edges)). In particular, $w_{l_i}$ is a closed Nielsen path of height $<l_i$ and so is trivial in $H_1(G;\Z/2\Z)$. By (NEG Nielsen Paths), $\mu$ a product of paths of the form $E_{l_i}u^kE_{l_k}^{-1}$ and closed Nielsen paths of height $<l_i$. In particular, $\mu$ is trivial in $H_1(G;\Z/2\Z)$, contradiction. This completes the proof of (1b).

In (1c), {by (NEG Nielsen Paths)} $H_{l_i}^c$ consists of two linear edges. This case is similar to the case that $H_{l_i}^c$ consists of one linear edge. We conclude $\mu$ is trivial in $H_1(G;\Z/2\Z)$, again a contradiction.

 In (2), by \cite[Corollary~4.19 and Remark~4.20]{fh:recognition} either there are no closed Nielsen paths of height $l_i$ or there is a Nielsen loop $\mu_{l_i}$ of height $l_i$ such that every Nielsen loop of this height is a power of $\mu_{l_i}$ and further $\mu_{l_i}$ is trivial in $H_1(G;\Z/2\Z)$. This is a contradiction.
 \end{proof}

\begin{corollary}\label{c:h1}
Let $\phi$ be rotationless. Either
\begin{itemize}
\item
for some $\Phi\in\PA(\phi)$, there is an element of $\Fix(\Phi)$ that is primitive in $\f$; or
\item
for all $\Phi\in\PA(\phi)$, every element of $\Fix(\Phi)$ is trivial in $H_1(\f;\Z/2\Z)$.
\end{itemize}
\end{corollary}

There are more general purpose algorithms that check whether a finitely generated subgroup of $\f$ contains a primitive element; see \cite{cg:primitive, wd:primitive}. We thank the referee for pointing out these references to us.

\begin{corollary}\label{c:primitively atoroidal}
There is an algorithm to tell whether or not a given $\phi\in\Out(\f)$ is primitively atoroidal.
\end{corollary}

\begin{proof}
Compute a \ct\ $\fG$ for a rotationless power of $\phi$. According to Lemma~\ref{l:h1}, $\phi$ is primitively atoroidal iff   some stratum of $G$ is a fixed circle. 
\end{proof}

\section{The index of an outer automorphism}\label{s:index}

\begin{definition}[\cite{gjll:index}]
For $\Phi\in\Aut(\f)$, $$i(\Phi):= \max\{0,\rk(\Fix(\Phi))+\frac{1}{2}a(\Phi)-1\}$$ where $a(\Phi)$ is the number of $\Fix(\Phi)$-orbits of attracting fixed points of $\partial\Phi$. For $\phi\in\Out(\f)$, $i(\phi):=\sum i(\Phi)$ where the sum ranges over representatives of isogredience classes of $\phi$.
\end{definition} 

It is clear that $i(\Phi)>0$ implies $\Phi$ is principal. In particular
$$i(\phi):= \sum \big(\rk(\Fix(\Phi))+\frac{1}{2}a(\Phi)-1\big)$$
where the sum $\sum i(\Phi)$ is over principal representatives $\Phi$ of $\phi$ only.

\begin{prop} \label{old index is algorithmic}
There is an algorithm to compute the index of rotationless $\phi\in\Out(\f)$.
\end{prop}

See also \cite[Section~6.1]{gjll:index}.

\begin{proof}
\def\graph{\Gamma}
Using Theorem~\ref{algorithmic ct}, construct a \ct\ $\fG$ for $\phi$. Let $\CS_N(f)=\sqcup_{[v]} \CS_N(f,[v])$ be the graphs constructed from $f$ in Section~\ref{s:S}. If $\Phi$ is in the isogredience class of the principal representative of $\phi$ determined by the principal vertex $v\in G$ then  $i(\Phi)=i(\CS_N(f,[v]))$ where the the index $i$ of a finite type connected graph $\graph$ is $$i(\graph):=\max\{0,\rk(\graph)+\frac{1}{2}a(\graph)-1\}$$ where $a(\graph)$ is the number of ends of $\graph$.
\end{proof}

The index satisfies a well-known inequality:

\begin{thm}[\cite{gjll:index}]
For $\phi\in\Out(\f)$, $i(\phi)\le n-1$.
\end{thm}

We take this opportunity to use \ct s to somewhat strengthen this inequality; see Proposition~\ref{p:index}. See Section~\ref{s:uniform bound} to recall notation.  Also see Lemma~\ref{l:identifying Fix+}.

From the \ct\ point of view, the set of attracting fixed points of $\partial \Phi$ can be partitioned into those coming from \eg\ strata and those coming from \noneg\ strata.  The following lemma states that this is independent of the choice of \ct\ representing $\phi$. 

\begin{lemma} \label{l:noneg ray}  \cite[Part \urn{1} Definitions 2.9 and 2.10 and Lemma 2.11] {hm:subgroups} Suppose that $\phi$ is rotationless, that  $\Phi\in\PA(\phi)$ and that $P \in \Fix_+(\partial\Phi)$.   Then the following are equivalent.\begin{enumerate}
\item
For some  \ct\ $\fG$ representing $\phi$ there is a non-linear  $NEG$ edge $E$ with a lift $\tilde E$ to $\tilde G$ so that $\tilde R_{\tilde E}$ converges to $P$.
\item
For every  \ct\ $\fG$ representing $\phi$ there is a non-linear  $NEG$ edge $E$ with a lift $\tilde E$ to $\tilde G$ so that $\tilde R_{\tilde E}$ converges to $P$.
\end{enumerate}
\end{lemma}

Following Lemma~\ref{l:noneg ray} we make the following definition, with justification given in Remark~\ref{rem:noneg ray}.

 \begin{definition}  \label{def:noneg ray} Let $\Phi\in\PA(\phi)$ with $\phi\in\Out(\f)$ and $P\in\Fix_+(\partial\Phi)$. If $\phi$ is rotationless then $P$ is an \noneg{\it-ray for $\Phi$}  if the equivalent conditions of Lemma~\ref{l:noneg ray} are satisfied.   If $\phi$ is not necessarily rotationless then $P$ is an \noneg{\it -ray for $\Phi$ } if $P$ is an \noneg-ray for $\Phi^K$ where $\phi^K$ is a rotationless iterate of $\phi$. Let $\cR_{\noneg}(\Phi)$ denote $\{R\mid \mbox{$R$ is an \noneg-ray for $\Phi$}\}$, let $\cR_{\noneg}(\phi)$ denote $\cup_{\Phi\in\PA(\phi)}\cR_{\noneg}(\Phi)$ and let $\cR(\phi)$ denote $\cup_{\Phi\in\PA(\phi)}\Fix_+(\Phi)$.
\end{definition}

\begin{remark} \label{rem:noneg ray} To see that Definiton~\ref{def:noneg ray} is well defined (i.e.\ independent of the choice of $K$), note  that if $\phi$ is rotationless and $f$ is a \ct\ representing $\phi$ then $f^k$ is a \ct\ representing $\phi^k$ for all $k \ge 1$ (Lemma~\ref{l:power of ct is ct}).  It follows that  $P$ is an \noneg-ray  for $\Phi$ if and only if it is a \noneg-ray for each $\Phi^k$.   If $\phi$ is not rotationless but has rotationless  iterates $\phi^K$ and $\phi^L$ then  $P$ is an \noneg-ray  for $\phi^K$ if and only if it is    an \noneg-ray  for $\phi^{KL}$ if and only if it is    an \noneg-ray  for $\phi^{L}$.
\end{remark}
\begin{definition}
Suppose $\phi\in\Out(\f)$ and $\Phi\in\PA(\phi)$. Define $j(\Phi):=i(\Phi)+\frac{1}{2}b(\Phi)$ where $b(\Phi)$ is the number of $\Fix(\Phi)$-orbits of $NEG$-rays of $\Phi$ and define $j(\phi):=\sum j(\Phi)$ where the sum is over representatives of isogredience classes of principal representatives of $\phi$.
\end{definition}

\begin{lemma}\label{l:neg-ray count}
Let $\Phi, \Psi\in\PA(\phi)$.
\begin{enumerate}
\item
If $\Fix_+(\Phi) \cap\Fix_+(\Psi)\not=\emptyset$ then $\Phi=\Psi$. In particular, $$\cR (\phi)=\sqcup_{\Phi\in\PA(\phi)}\Fix_+(\Phi)$$  and  $$\cR_{\noneg}(\phi)=\sqcup_{\Phi\in\PA(\phi)}\cR_{\noneg}(\Phi)$$  
\item
The stabilizer of $\Fix_+(\Phi)$ [resp. $\cR_{\noneg}(\Phi)$] under the action of $\f$ on $\cR(\phi)$ [resp. $\cR_{\noneg}(\phi)$]  is $\Fix(\Phi)$.
\item
The natural maps $$\sqcup_{i=1}^N\Fix_+(\Phi_i)/\Fix(\Phi_i)\to\cR (\phi)/\f$$ and  $$\sqcup_{i=1}^N\cR_{\noneg}(\Phi_i)/\Fix(\Phi_i)\to\cR_{\noneg}(\phi)/\f$$ are bijective where $\{\Phi_i\mid i=1,\dots, N\}$ is a set of representatives of isogredience classes in $\PA(\phi)$.
\end{enumerate}
\end{lemma}

\begin{proof}
(1): Since $\Phi$ and $\Psi$ both represent $\phi$, $\Phi\Psi^{-1}=i_a$ for some $a\in\f$. If $R\in \Fix_+(\Phi)\cap\Fix_+(\Psi)$ then $aR=R$. Since $R$ is not an endpoint of the axis of $a$, $a=1$. (Here we are using the shorthand notation $aR$ for $\partial i_a(R)$.)

(2): Suppose $a\in\Fix(\Phi)$ and $R\in\cR_{\noneg}(\Phi)$. Then $\partial\Phi(aR)=\Phi(a)\partial\Phi(R)=aR$ and so $aR\in\cR_{\noneg}(\Phi)$. Conversely, if $a\in\f$, $R\in\cR_{\noneg}(\Phi)$, and $\partial\Phi(aR)=aR$ then $aR=\partial\Phi(aR)=\Phi(a)\partial\Phi(R)=\Phi(a)R$. We conclude $a=\Phi(a)$ as in the proof of (1).  The same argument applies with $\cR_{\noneg}(\Phi)$ replaced by $ \Fix_+(\Phi)$.

(3): This is a consequence of (1), (2), and the observation that $\f$ acts on $\Fix_+(\Phi)$ [resp. $\cR_{\noneg}(\phi)$] by permuting the $\Fix_+(\Phi)$'s [resp. $\cR_{\noneg}(\Phi)$'s].
\end{proof}

\begin{lemma}\label{l:a and b increase}
\begin{enumerate}
\item

For $k>0$ and $\phi\in\Out(\f)$, $b(\phi^k)\ge b(\phi)$.

\item

For $k>0$ and $\phi\in\Out(\f)$, $a(\phi^k)\ge a(\phi)$.

\end{enumerate}
\end{lemma}

\begin{proof}

(1): By definition of $b(\phi)$ and Lemma~\ref{l:neg-ray count}(3), $$b(\phi)=|\sqcup_{i=1}^N\cR_{\noneg}(\Phi_i)/\Fix(\Phi_i)|=|\cR_{\noneg}(\phi)/\f|$$ Also by definition, if $R$ is an \noneg-ray for $\Psi\in\PA(\phi)$ then $R$ is an \noneg-ray for 
$\Psi^k\in\PA(\phi^k)$ and so $\cR_{\noneg}(\phi^k)\supseteq\cR_{\noneg}(\phi)$. Hence $$b(\phi^k)=|\cR_{\noneg}(\phi^k)/\f|\ge |\cR_{\noneg}(\phi)/\f |=b(\phi)$$

(2): The proof is the same as that of (1), replacing $\cR_{\noneg}(\Phi)$ with $\Fix_+(\Phi)$ and $\cR_{\noneg}(\phi)$ with $\cR(\phi)$. 
\end{proof}

\begin{notn}
For $H<\f$, $\hat r(H):=\max(0, \rk(H)-1)$. For $\Phi\in\Aut(\f)$, $\hat r(\Phi):=\hat r(\Fix(\Phi))$. For $\phi\in\Out(\f)$, $\hat r(\phi):=\sum\hat r(\Phi)$ where the sum is over representatives of $\PA(\phi)/\sim$.
\end{notn}

\begin{lemma}\label{l:hat r increases}
Let $\phi\in\Out(\f)$ and $k>0$ such that $\psi := \phi^k$ is rotationless. Then $\hat r(\psi)\ge\hat r(\phi)$.
\end{lemma}

\begin{proof} Assuming without loss that $\hat r(\phi) > 0$,  let $\Phi_1,\ldots,\Phi_s, \ldots  \Phi_m$  be representatives of $\PA(\phi)/\sim$ where    $\Fix(\Phi_i) \ge 2$ if and only if $i\le s$.    Let $\{\Psi_j\}$  be representatives of $\PA(\psi)/\sim$. For each $1 \le i \le s$ there exists $j = p(i)$ such that $\Phi_i^k \sim \Psi_j$.    If $j = p(i)$ then there exists $a \in F_n$ such that $\Psi_j=i_a\Phi_i^k i_{a^{-1}}$. Replacing $\Phi_i$ by $i_a\Phi i_{a^{-1}}$, we have $\Phi_i^k=\Psi_j$. Thus 
$$\Phi_i^k = \Psi_{p(i)}$$ for $1 \le i \le s$.   We may assume the $\Psi_j$'s are ordered so that the function $p$, whose domain is  $ \{1,\dots,s\}$, has image $\{1, \ldots, t\}$.

It suffices to show that 
$$ \sum_{j=1}^t\hat r(\Psi_j) \ge \sum_{i=1}^s\hat r(\Phi_i) $$
and so it also suffices to show that 
$$ \hat r(\Psi_j)\ge  \sum_{i \in p^{-1}(j)}   \hat r(\Phi_i )$$
for each $1 \le j \le t$.

Fix $j$ and let $\mathbb F = \Fix(\Psi_j)$.   For each $i \in p^{-1}(j)$, $\Phi_i^k = \Psi_j$ and so $\mathbb F = \Fix(\Phi_i^k)$. Thus, $\Phi_i$ preserves $\mathbb F$, $\Fix(\Phi_i) \subset \mathbb F$ and the restriction $\Phi_i|\mathbb F$ is a finite order automorphism of $\mathbb F$.     
Since $\mathbb F$ is its own normalizer (see the second bullet in the proof of Lemma~\ref{rotationless}) the restriction $\phi|\mathbb F\in\Out(\mathbb F)$ is well-defined and has finite order.  

We claim that if $i'  \in p^{-1}(j)$ and $i' \ne i$ then the subgroups $\Fix(\Phi_i)$  and $\Fix(\Phi_{i'})$ are not conjugate in $F_n$ and hence not conjugate in $\mathbb F$.  Indeed, if  $\Fix(\Phi_i)=h\Fix(\Phi_{i'})h^{-1}$ for some $h\in F_n$ then $\Fix(\Phi_i)=\Fix(i_h\Phi_{i'} i_{h^{-1}})$. Since these groups have rank at least two and both $\Phi_i$ and $i_h\Phi_{i'} i_{h^{-1}}$ represent $\phi$, $\Phi_i=i_h\Phi_{i'} i_{h^{-1}}$, in contradiction to the assumption that $\Phi_{i}$ and $\Phi_{i'}$ represent distinct isogredience classes. 

By Culler's Theorem~3.1 of \cite{mc:periodic}, applied to $\phi|\mathbb F$, the conjugacy classes in $\mathbb F$ determined by 
the subgroups $\{\Fix(\Phi_i): i \in p^{-1}(j)\}$ form a free factor system of $\mathbb F$.  Since $\hat r(\Psi_j) = \rk(\mathbb F)- 1$ and $\hat r(\Phi_i) = \rk(\Fix(\Phi_i))-1$, the desired displayed inequality holds.
\end{proof}

\begin{lemma}\label{l:j increases}
Suppose $\phi\in\Out(\f)$ and $k>0$ such that $\phi^k$ is rotationless. Then $j(\phi^k)\ge j(\phi)$.
\end{lemma}

\begin{proof}
By definition of $j(\phi)$, the lemma is a direct consequence of Lemmas~\ref{l:a and b increase} and \ref{l:hat r increases}.
\end{proof}

\begin{prop}
There is an algorithm to compute $j(\phi)$ for rotationless $\phi\in\Out(\f)$. 
\end{prop}

\begin{proof}
This follows from Proposition~\ref{old index is algorithmic} and the fact that $j(\phi)$ is the sum of $i(\phi)$ and one half the number of non-linear \noneg\ edges in any \ct\ representing $\phi$.
\end{proof}

\begin{prop}\label{p:index}
For $\phi\in\Out(\f)$, $j(\phi)\le n-1$.
\end{prop}

\begin{ex}
Start with the usual linear 2-rose $a_1\mapsto a_1$ and $a_2\mapsto a_2a_1$. Further attach edges $a_i$, $2<i\le n$, with $a_i\mapsto a_2^{2i}a_ia^{2i+1}_2$. If we subdivide each $a_i$, $2 < i\le n$ then the result $\fG$ is a \ct\ for a $\phi\in\Out(\f)$. Here $\CS_N(f)$ is the disjoint union of a pair of eyeglasses and $n-2$ lines each with two $NEG$-rays. Hence $j(\phi)=n-1$ and $i(\phi)=1$.
\end{ex}

\begin{proof}[Proof of Proposition~\ref{p:index}]
 By Lemma~\ref{l:j increases}, we may assume that $\phi$ is rotationless. Let $\fG$ be a \ct\ for $\phi$. We use the terminology of Section~\ref{s:filtration}.  In particular, 
$$\emptyset=G_0=G_{l_0}\subset G_{l_1}\subset G_{l_2}\subset\dots\subset G_{l_K}=G_N = G$$
is the core filtration and  $\Delta_i\chi^-:=\chi^-(G_{l_i})-\chi^-(G_{l_{i-1}})$.  

Recall that $\pstallings_N(f)$ is built up in three stages: starting with  principal vertices and fixed edges, lollipops corresponding to linear edges, edges corresponding to \eg\ Nielsen paths, and then rays corresponding (perhaps not bijectively) to oriented edges in $\E$ are added.   For each $1 \le 1 \le K$,   define $\CS_N(k)$ to be the subgraph of $\CS_N(f)$ corresponding to $G_{l_k}$.  More precisely, start with  principal vertices   in $G_{l_k}$ and fixed edges in $G_{l_k}$, then add lollipops corresponding to linear edges  in $G_{l_k}$ and edges corresponding to \eg\ Nielsen paths  in $G_{l_k}$  and then add rays corresponding to oriented edges in $\E \cap G_{l_k}$. 

For each component $\Gamma$ of $\CS_N(k)$ define 
 $$j(\Gamma):= \rk(\Gamma)+\frac{1}{2}a(\Gamma)+\frac{1}{2}b(\Gamma)-1$$ 
 where $a(\Gamma)$ is the number of ends of $\Gamma$ and $b(\Gamma)$ is the number of $NEG$-rays of $\Gamma$.    Then define 
$$j(k) = \sum_\Gamma j(\Gamma)$$
where the sum is taken over all components $\Gamma$ of $\CS_N(k)$.  Using Remark~\ref{r:fix at infinity}, $j(\AS_N(f))=j(\phi)=j(K)$. Proposition~\ref{p:index}  therefore follows from the more general inequality
\begin{equation}\label{e:index}
j(k)\le \chi^-(G_{l_k})\tag{$*_k$}
\end{equation}
for $0\le k\le K$, which we will prove by induction on $k$.

The $k=0$ case holds because $j(\emptyset) =  \chi^-(\emptyset) = 0$.

Suppose that ($*_{k-1}$) holds for some $k\in\{ 1, \dots, K\}$. Define $\Delta_k j:=j(k)-j(k-1)$. We check that, in each of the cases of Lemma~\ref{l:filtration}, $\Delta_k j \le \Delta_k\chi^-$. Once this has been done, ($*_k$) holds and the proposition is proven.
 
 We use the following formula to compute $\chi^-$ of a finite graph: $$\chi^-= \Sigma_v (\frac{val(v)}{2}-1)$$ where the sum is over the vertices $v$ and $val(v)$ is the number of directions at $v$. Thus, $\Delta_k\chi^-$ is the half number of ``new" directions (two for each new edge) minus the number of ``new" vertices. As we proceed through the process of verifying $\Delta_k j \le \Delta_k\chi^-$, we refer to vertices, directions, or rays of $G_{l_k}$ or $\CS_N(k)$ previously considered as {\it old}; others are {\it new}.

Lemma~\ref{l:not principal} implies that in cases (1a), (1b) and (1c) the endpoints of all new edges are principal.   

{\bf Case 1a.} $\CS_N(k)$ is obtained from $\CS_N(k-1)$   by attaching a circle component  and hence
$\Delta_k j=\Delta_k\chi^-=0$.

{\bf Case 1b.} If the edge $H:=H_{l_k}$ is fixed then $\CS_N(k)$  [resp. $G_{l_k}$] is obtained from $\CS_N(k-1)$ [resp. $G_{l_{k-1}}$] by adding a new edge to old vertices. Hence $\Delta_k j=\Delta_k\chi^-=1$.   We may therefore assume that the edge $H$ is not fixed and so $f(H)=H\cdot u$ with $u$ non-trivial. In this case, $\CS_N(k)$ is obtained from $\CS_N(k-1)$ by attaching to an old vertex either a lollipop (if $u$ is a Nielsen path) or the $NEG$-ray $H\cdot u \cdot f(u) \cdot f^2(u)\dots$ (otherwise). In each case, $\Delta_k j=\Delta_k\chi^-=1$.

{\bf Case 1c.} There are three subcases depending on whether or not the added edges are linear. Their shared initial vertex is a new principal vertex in $G_{l_k}$ and so $\CS_N(k)$ is obtained from $\CS_N(k-1)$ by adding a new component. The new component is a pair of eyeglasses if both edges are linear. It is a one point union of an $NEG$-ray and a lollipop if only one edge is linear, and it is a line with two $NEG$-rays if neither edge is linear. In all cases, $\Delta_kj=\Delta\chi^-=1$.

{\bf Case 2.} In Case~2, we will break the verification that $\Delta_k j \le \Delta_k\chi^-$ into two steps: first passing from $G_{l_{k-1}}$ to $G_{u_k}$ and then passing from $G_{u_k}$ to $G_{l_k}$.

Step~1: For the edge $H_j$ with $j$ as in Case~(2a), the contribution to $\CS_N(k)$ is a new \noneg -ray [or lollipop] with new initial vertex. (The terminal direction in $H_j$ is not fixed and so does not contribute.) Thus, the contribution to $\Delta_kj$ is 0 ($\frac{1}{2}$ to the $a$-count since the ray is new, another $\frac{1}{2}$ to the $b$-count since it is \noneg\ [or 1 to rank] and -1 for the new component). This is balanced with the contribution of $H_j$ to $\Delta_k \chi^-$, namely $\frac{1}{2}$ for each of its directions and -1 for the new vertex.

Step~2:   For each vertex $v \in H_{l_k}$, let    $\Delta_kj(v)$  and $\Delta_k\chi^-(v)$ be the contributions to $\Delta_kj$  and $\Delta_k\chi^-$ coming from $v$ and the   $H_{l_k}$-directions that are incident to $v$.   For each principal vertex $v \in H_{l_k}$ let $\kappa(v)$ be the number of $H_{l_k}$-directions  that are incident to $v$ and not fixed.   If $v \in H_{l_k}$ is not principal let $\kappa(v) = val(v) -2$. 
 Since each direction we encounter in this step is contained in $H_{l_k}$, and is hence \eg, the $b$-count does not change.
 
 As a first case assume that there are no \iNp s of height $l_k$.     If $v$ is not principal then $\Delta_kj(v) = 0$.    By Lemma~\ref{l:not principal},  $v$ is new   so the directions incident to $v$  are all in $H_{l_k}$ and     $\Delta_k\chi^-(v) = \frac{1}{2} (val(v) - 2) = \frac{1}{2} \kappa(v)$.  Thus  $\Delta_k\chi^-(v)  -\Delta_kj(v)  =\frac{1}{2} \kappa(v) \ge 0$.

 We next consider principal $v$, letting $L(v)$ be the number of fixed $H_{l_k}$-directions   that are based at $v$.   Thus $L(v) + \kappa(v)$ is the number of  $H_{l_k}$-directions   that are based at $v$.  If $v$ is old then $\Delta_kj(v) = \frac{1}{2}L(v)$  and $\Delta_k\chi^-(v) = \frac{1}{2}(L(v)+\kappa(v))$.  If $v$ is new then      $\Delta_kj(v)  = \frac{1}{2}L(v) -1$  and $\Delta_k\chi^-(v) = \frac{1}{2}(L(v)+\kappa(v))-1$.     As in the previous case, $\Delta_k\chi^-(v)  -\Delta_kj(v) =\frac{1}{2} \kappa(v) \ge 0$.

     The second and final case is that there is an \iNp\ $\rho$ of height $l_{k}$.  By \cite[Lemma~4.24]{fh:recognition}, $l_k = u_{k}+1$.   From \cite[Lemma 5.1.7]{bfh:tits1} and \cite[Corollary 4.19]{fh:recognition} it follows  that  if  the endpoints $w,w'$ of $\rho$ are distinct then  at least one of $w,w'$ is new and  if $w = w'$ then $w$ is new.   Let $d$ and $d'$ be the directions determined by $\rho$ at $w$ and $w'$ respectively.  
 
  Let $\V$ be the set of  vertices of $H_{l_k}$ that are not   endpoints of $\rho$.   Each $v \in \V$ is handled as in the no Nielsen path case.  The conclusion is that  
  
  $$\sum_{v \in \V} \Delta_k\chi^-(v)  - \sum_{v \in \V}\Delta_kj(v) = \sum_{v \in \V}\frac{1}{2} \kappa(v) \ge 0$$

 Let $\Delta_kj(\rho)$     and $\Delta_k\chi^-(\rho)$ be the contributions to $\Delta_kj$  and $\Delta_k\chi^-$ coming from the endpoints of $\rho$  and the   $H_{l_k}$-directions that are incident to the endpoints of $\rho$.  The remaining analysis breaks up into subcases as follows.
 
  \begin{itemize}
 
\item  Suppose that $\rho$ is a closed path based at a  new vertex $w$.  
The change in $CS_N(k)$ corresponding to $w$ is the addition of a new vertex $w$, an edge  representing $\rho$ with both endpoints at $w$,  one ray for each  fixed $H_{l_k}$-direction based at $w$  other than $d$ and $d'$ and then one ray corresponding to $d$ and $d'$.  (See Section~\ref{s:S}.)  We therefore have 
 $$\Delta_kj(\rho) = 1+ \frac{1}{2}(L(w) -1)   -1  =  \frac{1}{2}L(w) - \frac{1}{2} $$  
 $$\Delta_k\chi^-(\rho) = \frac{1}{2}(L(w)  + \kappa(w))  - 1$$  Thus 
 
$$ \Delta_k\chi^-(\rho)  -\Delta_kj(\rho)      = \frac{1}{2}(\kappa(w) - 1)  $$

Since there is always at least one illegal turn between  $H_{l_k}$-directions (for example, the one in $\rho)$   there must be at least one vertex $v \in \V \cup \{w\}$ with $\kappa(v) \ne 0$.  We conclude that   $\ \Delta_kj \le \Delta_k\chi^-$ as desired with equality if and only if $\sum \kappa(v) =1$ where the sum is taken over all vertices in $H_{i_k}$.

\item  Suppose that    $w$ is new and $w'$ is old.  The change in $CS(k)$ corresponding to $w$ and $w'$ is the addition of a new vertex $w$, an edge  representing $\rho$ connecting $w$ to $w'$,  one ray for each  fixed $H_{l_k}$-direction based at $w$ or $w'$  other than $d$ and $d'$ and then one ray corresponding to $d$ and $d'$.    This yields 
$$\Delta_kj(\rho) =   \frac{1}{2}(L(w) + L(w'))- \frac{1}{2} $$  
 $$\Delta_k\chi^-(\rho) = \frac{1}{2}(L(w) +L(w')+ \kappa(w) + \kappa(w'))- 1$$    
and the proof concludes as in the previous case.

   \item Suppose that $w$ and $w'$ are distinct and both new.   The change in $CS(k)$ corresponding to $w$ and $w'$ is the addition of two new vertices, an edge  representing $\rho$ connecting them,  one ray for each  fixed $H_{l_k}$-direction based at $w$ or $w'$  other than $d$ and $d'$ and then one ray corresponding to $d$ and $d'$. The calculation is the same as in  the previous case.  
\end{itemize}

\end{proof}

We have the following two corollaries to the proof of Proposition~\ref{p:index}.
\begin{corollary}
If $G$ has no $EG$-strata then $j(\phi)=n-1$. If $G$ has an $EG$-stratum without Nielsen paths, then $j(\phi)<n-1$. \qed
\end{corollary}

For the next corollary, recall  (Sections~\ref{s:principal} and \ref{s:sn(f)}) that there are bijections between Nielsen classes of principal vertices of $G$ and isogredience classes in $\PA(\phi)$ and the set of components of $\CS_N(f)$.
\begin{corollary}
\begin{enumerate}
\item Suppose that $\CS_N(f,[v])$ is a line $L$. Then one of its rays is $\noneg$.
\item
Each $\CS_N(f,[v])$ contributes at least 1/2 to $j(\phi)$.
\item
$|[\PA(\phi)]|\le 2j(\phi)$, i.e.\ $\CS_N(f)$ has at most $2j(\phi)$ components.
\end{enumerate}
\end{corollary}
 
\proof   (3) follows from (2)   We will prove (1) and (2) simultaneously.

Since $ \CS_N(f,[v])$ is a weakly core graph we may assume that each vertex of $\CS_N(f,[v])$  has valence two.   In particularly, if  $[w] = [v]$ then   $w$  is not the endpoint of an \iNp\ (see Section~\ref{s:S}) and $w$ has only two gates in $G$.   We view $\CS_N(f)$ as being built up as in the proof of Proposition~\ref{p:index}.  Let $w$ be the lowest principal vertex satisfying $[w] = [v]$  and let $G_{l_k}$ be the lowest core filtration element that contains $w$.  We consider the cases enumerated in the proof of  Proposition~\ref{p:index}.  Case (1a) is ruled out because $w$ is principal and so would have at least three gates.    Since $w$ is new in $H_{l_k}$ we are not in case (1b).  If we are in case (1c)   then both edges must be  non-linear and $j(\Gamma) \ge 1$.   It remains to rule out case (2).      Since $w$ is a new  vertex in the \eg\ strata $H_{l_k}$ there are at least two gates in $H_{l_k}$ based at $w$.  Since $w$ is principal but is  is not the endpoint of  an \iNp, there must be at least three gates at $w$.  This completes the proof of (1) and (2).
 
\begin{corollary}\label{c:improved ends count}
$|\big(\cup_{\Phi\in\PA(\phi)}\Fix_+(\partial\Phi)\big)/\f|\le 6(n-1)$.
\end{corollary}

\begin{proof} 
\begin{align*}
j(\phi)&=\sum_{[v]}j(\CS_N(f,[v]))\ge\sum_{[v]}(\frac{a(\CS_N(f,[v])}{2}-1)\\
&=\frac{|\mbox{ends of $\CS_N(f)$}|}{2}-|\mbox{components of $\CS_N(f)$}|\\
&\ge \frac{1}{2}\cdot |
\big(\cup_{\Phi\in\PA(\phi)}\Fix_+(\partial\Phi)\big)/\f\}|-2j(\phi)
\end{align*}

Thus $|\big(\cup_{\Phi\in\PA(\phi)}\Fix_+(\partial\Phi)\big)/\f| \le 6 j(\phi) \le 6(n-1)$ by Proposition~\ref{p:index}.
\end{proof}

\begin{remark}
We commented in Remark~\ref{r:better count} how Proposition~\ref{p:index} could be used to improve the bound in Lemma~\ref{l:bound on Fix+} and hence also Corollary~\ref{c:kn}. With more care, Proposition~\ref{p:index} could be used to further improve this bound.
\end{remark}

\section{Appendix: Hyperbolic and atoroidal automorphisms}

In this appendix we reprove a result of Brinkmann (Lemma~\ref{l:brinkmann}) and a result of Kapovich (Corollary~\ref{c:brinkmann}).

\begin{definition}
An outer automorphism $\phi\in\Out(F_n)$ is {\it hyperbolic} if for some $N>0$ and $\lambda>1$, $$\lambda \| \alpha\| \le  \max\{\| \phi^N(\alpha)\|,\|\phi^{-N}(\alpha)\|\}$$  for all non-trivial conjugacy classes $\alpha$ in $F_n$. 
We say $\phi$ is {\it atoroidal} if $\phi$ has no non-trivial periodic conjugacy classes. $\Phi\in\Aut(F_n)$ is {\it hyperbolic} if for some $N>0$ and $\lambda>1$, $$\lambda | a| \le \max\{| \Phi^N(a)|,|\Phi^{-N}(a)|\}$$ for all non-trivial $a$ in $F_n$. Here $\|\cdot\|$ and $|\cdot |$ denote respectively reduced word length and word length with respect to a fixed basis for $F_n$. The {\it mapping torus $M_\Phi$ of $\Phi\in\Aut(F_n)$} is the group with presentation $\langle F_n, t\mid tat^{-1}=\Phi(a) \mbox{ for each } a\in F_n\rangle$. 
\end{definition}

\begin{lemma}[\cite{pb:hyperbolic}]\label{l:brinkmann}
Suppose that no conjugacy class in $F_n$ is fixed by an iterate of $\phi$. Then $\phi$ is  hyperbolic.
\end{lemma}

\proof Suppose that  $\Lambda^\pm_1,\ldots, \Lambda^\pm_m$ are the lamination pairs for $\phi$.  After replacing $\phi$ by an iterate, we may assume that $\phi$ and $\phi^{-1}$  are rotationless. Since $\phi$ does not fix any conjugacy classes,  each   $\Lambda_i^\pm$  is non-geometric,  and so     (\cite[Part \urn{3} Theorem~F]{hm:subgroups}; see also \cite[Definition~5.1.4 and Theorem~6.0.1]{bfh:tits1}) there is a $\phi$-invariant free factor system $\A_i$ for which the following are equivalent for each conjugacy class $[a]$ in $F_n$.
 \begin{itemize}
 \item  $[a]$  is not weakly attracted to $\Lambda_i^+$ under iteration by $\phi$.
 \item    $[a]$  is not weakly attracted to $\Lambda_i^-$ under iteration by $\phi^{-1}$.
 \item  $[a]$ is carried by $\A_i$.  
 \end{itemize}
 In particular, $\A_i$ does not carry $\Lambda_i^\pm$.
 
We claim that no line is carried by every $\A_i$.  If this failed, then one could choose free factors $A_i$ such that $[A_i]$ is a component of $\A_i$ and such that $A := A_1 \cap \ldots \cap A_m$ is  non-trivial.  Thus   $A$ is a non-trivial free factor   that   does not  carry  any $\Lambda_i^\pm$.   There is a \ct\    representing $\phi$ in which $[A]$ is represented by a filtration element.  The lowest stratum  cannot be \eg\ and so must be a fixed loop.  But that contradicts the assumption that  there are no $\phi$-invariant conjugacy classes and so completes the proof of the claim. 

Choose \ct s $\fG$ representing $\phi$ and $f' :G' \to G'$ representing $\phi^{-1}$. Let $\lambda_i$ the expansion factor (\cite[Definition 3.3.2]{bfh:tits1}) for $\phi$ with respect to $\Lambda^+_i$,  $\mu_i$ the expansion factor for $\phi^{-1}$ with respect to $\Lambda^-_i$ and $\lambda=\min\{\lambda_i,\mu_i\}> 1$. 

 Let $H_i \subset G$ be the stratum corresponding to $\Lambda^+_i$.    By \cite[Lemma~4.2.2]{bfh:tits1} there exists a subpath  $\delta_i \subset G$ of $\Lambda^+_i$ (any subpath of $\Lambda_i^+$ that crosses sufficiently many edges of $H_i$ will do)  with the following property:  if   $\sigma \subset G$ is a circuit or path and $\sigma_0 \subset \sigma$ is a copy of $\delta_i $ or its inverse $\bar \delta_i$, then there is a splitting of $\sigma$ in which one of the terms is an edge of $H_i$ contained in $\sigma_0$.    Let $B^+(\sigma)$  be the maximum number of disjoint subpaths of $\sigma$ that are copies of some $\delta_i $ or $\bar \delta_i $.  Then $\sigma$ has a splitting in which $B^+(\sigma)$ terms are single edges in \eg\ strata.  It follows that   $$|f^k_\#(\sigma)| > C  B^+ (\sigma) \lambda^k$$ for all $k \ge 1$ where  $C$ is a positive constant and $| \cdot |$ denotes length.
 
 Define $\delta_i' \subset G'$ and $B^-(\sigma')$ symmetrically replacing $\Lambda^+_i$ and $\fG$ with $\Lambda^-_i$ with $f' :G' \to G'$.  Each path or circuit $\sigma' \subset G'$  has a splitting in which $B^-(\sigma')$ terms are single edges in \eg\ strata of $f' :G' \to G'$.  After decreasing $C$ if necessary, we have     $$|(f')^k_\#(\sigma')| > C B^- (\sigma') \lambda^k$$ for all $k \ge 1$.  

Let $h :G \to G'$  be a homotopy equivalence that preserves markings. After replacing $\phi$ (and hence $f$ and $f'$) by an iterate, we may assume that for each $i$  the neighborhood $V(\delta_i)$ of $\Lambda^+_i$ consisting of lines that contain either $\delta_i$ or $\bar \delta_i$ as a subpath is mapped into itself by $\phi$ and that  the neighborhood $V(\delta_i'$) of $\Lambda^-_i$ determined by $\delta_i'$  is mapped into itself by $\phi^{-1}$. By  \cite[Part \urn{3} Theorem H]{hm:subgroups}  there is  a positive integer $N_i$ so that if  $\beta \subset G$ is a line that is not carried by  $\A_i$ then either $ h_\#(\beta)$   contains a copy of $\delta'_i$ or $\bar \delta'_i$ or $f^{N_i}_\#(\beta)$ contains a copy of $\delta_i$  or  $\bar \delta_i$.  Let $N := \max\{N_i\}$. Since $f_\#$ maps each $V(\delta_i)$ into itself and since no line is carried by $\A_i$ for all $i$, we have  
$$B^+(f^N_\#(\beta)) + B^-(h_\#(\beta))\ge 1$$
 for all lines $\beta $.
 
Recall (see  \cite[Part \urn{1} Section 1.1.6]{hm:subgroups}) that for all paths $\alpha \subset G$ there is a path $h_{\#\#}(\alpha)  \subset G'$ obtained from $h_\#(\sigma)$ by removing initial and terminal segments of length at most the bounded cancellation constant of $h$ and satisfying $h_{\#\#}(V(\alpha)) \subset V(h_\#(\alpha))$. 
We claim that  
 there exists a  positive integer $L$  so that if $ \sigma \subset G$ has length at least $L$ then 
 $$ B^+(f_\#^N(\sigma)) + B^-(h_{\#\#} (\sigma)) \ge 1$$
 Indeed, if this  fails then there exists    $L_j\to \infty$ and paths $\sigma_j \subset G$ with length $ \ge L_j$ such that  for all $1 \le i \le m$
\begin{itemize}
\item $h_{\#\#} (\sigma_j)$    does not contain  $ \delta'_i$ or $\bar \delta'_i$ as a subpath.
\item $f_\#^{N}(\sigma_j)$    does not contain $\delta_i$ or $\bar \delta_i$ as a subpath.
\end{itemize}
By focusing on the \lq middle\rq\ of each $ \sigma_j$ and passing to a subsequence we may assume that there are subpaths $ \beta_j \subset  \sigma_j$ such that $\beta_1 \subset \beta_2 \subset \ldots$ is an increasing sequence of paths whose union is a line $\beta \subset G$.   As verified above,  there exists $1 \le i \le m$ such that $\A_i$ does not carry $ \beta$.     By \cite[Part \urn{1} Lemma 1.6(3)]{hm:subgroups}  $h_{\#\#} (\beta_j) \subset  h_{\#\#} (\beta_{j+1}) \cap  h_{\#\#} (\sigma_j)$ and $h_\#(\beta)$ is the union of the $h_{\#\#} (\beta_j)$'s. It follows that  $h_\#(\beta)$ does not contain $\delta'_i$ or $\bar \delta'_i$  as a subpath and so $f^N_\#(\beta)$ must  contain a copy of $\delta_i$ or $\bar \delta_i$.  But then $f^N_{\#\#}(\beta_j)$   contains a copy of $\delta_i$ or $\bar \delta_i$ for all sufficiently large $j$ and hence $f^N_{\#}(\sigma_j)$    contains a copy of $\delta_i$ or $\bar \delta_i$ for all sufficiently large $j$.     This contradiction completes the proof of the claim.

We are now ready to complete the proof.  After replacing $\phi$ with $\phi^N$, we may assume that $N=1$.  Given a circuit $\sigma  \subset G$ with $|\sigma|\ge 2L$ divide it into $[ \frac{|\sigma|}{L}]$ subpaths $\sigma_l$ of length at least $L$ where $[ \cdot ]$ is the greatest integer function.   Applying the preceding claim and the fact   that  the $f_{\#\#}(\sigma_l)$'s are disjoint subpaths of  $f_{\#}(\sigma)$   we have 
$$\max\{ B^+(f_\#(\sigma)), B^-(\sigma')\} \ge [ \frac{| \sigma|}{2L}]$$
  In conjunction with the above displayed inequalities   this completes the proof of the lemma if $|\hat \sigma| \ge 2 L$.  As there are only finitely many remaining $\hat \sigma$ and none of these is fixed by an iterate of $\phi$ we are done. \endproof

\begin{prop}\label{p:brinkmann}
Let $\Phi\in \Aut(F_n)$ represent $\phi\in\Out(F_n)$. The following are equivalent.
\begin{enumerate}
\item
$\phi$ is atoroidal.
\item
$\phi$ is hyperbolic.
\item
$\Phi$ is hyperbolic.
\item
$M_\Phi$ is hyperbolic.  
\end{enumerate}
\end{prop}

\begin{proof}
$(1)\implies (2)$ is Lemma~\ref{l:brinkmann}. 
{$(2)\implies (3)$ follows from the proof of Theorem~5.1 of \cite{bfh:tits0}.}
$(3)\implies (4)$ is a consequence of the first corollary of Section~5 of \cite{bf:combination}. $(4)\implies (1)$ since otherwise $M_\Phi$ contains $\Z^2$.
\end{proof}

\begin{corollary}[\cite{ik:hyperbolic}; see also {\cite[page~2]{fd:conjugacy}}]\label{c:brinkmann}
There is an algorithm with input $\phi\in\Out(F_n)$ that outputs {YES} or {NO} depending on whether or not $\phi$ is hyperbolic.
\end{corollary}

\begin{proof}
Construct a \ct\ $\fG$ for a $\phi^M$ with $M$ as in Corollary~\ref{c:kn}.   By Proposition~\ref{p:brinkmann}, $\phi$ is hyperbolic iff $\phi$ is atoroidal. By Lemma~\ref{lifting to S(f)}(2), $\phi$ is atoroidal if and only if $\pstallings(f)$ has no circuits and this can be checked algorithmically; see Remark~\ref{r:stallings}.  
\end{proof}

\bibliographystyle{amsalpha}
\bibliography{../../ref}

\end{document}